\input ./style/arxiv-general.cfg
\documentclass[aop,MSNbibl,seceqn,dvips]{arximspdf}
\makeatletter
   \@ifpackageloaded{graphicx}{}{\usepackage{graphicx}}
\makeatother
\usepackage{mathrsfs}

\doi{10.1214/14-AOP930} 
\volume{43}
\issue{4}
\pubyear{2015}
\firstpage{2151}
\lastpage{2203}

\makeatletter

\newcommand{\eqref}[1]{(\ref{#1})}
\newtheorem{theorem}{Theorem}[section]
\newtheorem{lemma}[theorem]{Lemma}
\newtheorem{proposition}[theorem]{Proposition}
\newtheorem{corollary}[theorem]{Corollary}

\newproclaim{definition}[theorem]{Definition}
\newproclaim{remark}[theorem]{Remark}

\newcommand{\mc}[1]{{\mathcal #1}}
\newcommand{\mf}[1]{{\mathfrak #1}}
\newcommand{\mb}[1]{{\mathbf #1}}
\newcommand{\bb}[1]{{\mathbb #1}}
\newcommand{\bs}[1]{{\bolds #1}}

\newcommand{\mt}[1]{{\texttt #1}}

\renewcommand{\Cap}{\operatorname{cap}}

\makeatother

\begin{document}
\begin{frontmatter}

\title{Zero-temperature limit of the Kawasaki dynamics for the Ising
lattice gas in a large two-dimensional torus}
\runtitle{Kawasaki dynamics at low temperature}

\begin{aug}
\author[A]{\fnms{B.}~\snm{Gois}\ead[label=e1]{brunog@impa.br}}
\and
\author[B]{\fnms{C.}~\snm{Landim}\corref{}\ead[label=e2]{landim@impa.br}}
\runauthor{B. Gois and C. Landim}
\affiliation{IMPA, and IMPA and Universit\'e de Rouen}
\address[A]{IMPA\\
Estrada Dona Castorina 110\\
CEP 22460 Rio de Janeiro\\
Brasil\\
\printead{e1}} 
\address[B]{IMPA\\
Estrada Dona Castorina 110\\
CEP 22460 Rio de Janeiro\\
Brasil\\
and\\
CNRS UMR 6085\\
Universit\'e de Rouen\\
Avenue de l'Universit\'e, BP.12\\
Technop\^ole du Madril\-let\\
F76801 Saint-\'Etienne-du-Rouvray\\
France\\
\printead{e2}}
\end{aug}

\received{\smonth{5} \syear{2013}}
\revised{\smonth{2} \syear{2014}}

%
\begin{abstract}
We consider the Kawasaki dynamics at inverse temperature $\beta$ for
the Ising lattice gas on a two-dimensional square of length $2L+1$
with periodic boundary conditions. We assume that initially the
particles form a square of length $n$, which may increase, as well
as $L$, with $\beta$. We show that in a proper time scale the
particles form almost always a square and that the center of mass of
the square evolves as a Brownian motion when the temperature
vanishes.
\end{abstract}

%
\end{frontmatter}

\section{Introduction and main results}
\label{sec-1}

We introduced in \cite{bl2,bl7} a general method to describe the
asymptotic evolution of one-parameter families of continuous-time
Markov chains among the ground states or from high energy sets to
lower energy sets. This method has been successfully applied in two
situations. For zero-range dynamics on a finite set which exhibit
condensation \cite{bl3,l1}, and for random walks evolving among
random traps \cite{jlt1,jlt2}. In the first model, the chain admits a
finite number of ground sets and one can prove that in an appropriate
time scale the process evolves as a finite state Markov chain, each
state corresponding to a ground set of the original chain. In the
second model, there is a countable number of ground states and one can
prove that in a certain time scale the process evolves as a
continuous-time Markov chain in a countable state space, each point
representing one of the ground states. In this paper, which follows
\cite{bl4,bl5}, we investigate a third case, where the limit dynamics
is a continuous process.

We consider the Ising lattice gas on a torus subjected to a Kawasaki
dynamics at inverse temperature $\beta$. Let $\bb T_L = \{-L, \ldots,
L\}^2$, $L\ge1$, be a square with periodic boundary
conditions. Denote by $\bb T^*_L$ the set of bonds of $\bb T_L$. This
is the set of unordered pairs $\{x,y\}$ of $\bb T_L$ such that $\Vert
x-y \Vert=1$, where $\Vert \cdot\Vert$ stands for the Euclidean
distance. The configurations are denoted by $\eta= \{\eta(x) \dvtx x\in
\bb T_L\}$, where $\eta(x) =1$ if site $x$ is occupied and $\eta(x)=0$
if site $x$ is vacant. The Hamiltonian $\bb H$, defined on the state
space $\Omega_L = \{0,1\}^{\bb T_L}$, is given by
\[
- \bb H (\eta) = \sum_{\{x,y\}\in\bb T^*_L} \eta(x) \eta(y).
\]
The Gibbs measure at inverse temperature $\beta$ associated to the
Hamiltonian $\bb H$, denoted by $\mu_\beta$, is given by
\[
\mu_\beta(\eta) = \frac{1}{Z_\beta} e^{-\beta\bb H(\eta)},
\]
where $Z_\beta$ is the normalizing partition function.

We consider the continuous-time Markov chain $\{\eta^\beta_t \dvtx t\ge
0\}$ on $\Omega_L$ whose generator $L_\beta$ acts on functions
$f\dvtx \Omega_L \to\bb R$ as
\[
(L_\beta f) (\eta) = \sum_{\{x,y\}\in\bb T^*_L}
c_{x,y}(\eta) \bigl[f\bigl(\sigma^{x,y}\eta\bigr) - f(\eta)
\bigr],
\]
where $\sigma^{x,y}\eta$ is the configuration obtained from $\eta$ by
exchanging the occupation variables $\eta(x)$ and $\eta(y)$:
\[
\bigl(\sigma^{x,y}\eta\bigr) (z) = \cases{ \eta(z), &\quad $\mbox{if $z\neq
x, y$},$ \vspace*{2pt}
\cr
\eta(y), &\quad $\mbox{if $z = x$},$ \vspace*{2pt}
\cr
\eta(x),
&\quad $\mbox{if $z = y$}.$}
\]
The rates $c_{x,y}$ are given by
\[
c_{x,y} (\eta) = \exp \bigl\{-\beta \bigl[\bb H\bigl(
\sigma^{x,y}\eta\bigr) - \bb H(\eta)\bigr]_+ \bigr\},
\]
and $[a]_+$, $a\in\bb R$, stands for the positive part of $a$: $[a]_+
= \max\{a,0\}$. We sometimes represent $\eta^\beta_t$ by
$\eta^\beta(t)$ and we frequently omit the index $\beta$ of
$\eta^\beta_t$.

A simple computation shows that the Markov process $\{\eta_t \dvtx t\ge
0\}$ is reversible with respect to the Gibbs measures $\mu_\beta$,
$\beta>0$, and ergodic on each irreducible component formed by the
configurations with a fixed total number of particles. Let
$\Omega_{L,K} = \{\eta\in\Omega_L \dvtx \sum_{x\in\bb T_L} \eta(x) =
K\}$, $0\le K\le|\bb T_L|$, and denote by $\mu_K = \mu_{\beta, K}$
the Gibbs measure $\mu_\beta$ conditioned on $\Omega_{L,K}$:
\[
\mu_{K} (\eta) = \frac{1}{Z_{\beta, K}} e^{-\beta\bb
H(\eta)},\qquad \eta\in
\Omega_{L,K},
\]
where $Z_K=Z_{\beta, K}$ is the normalizing constant $Z_{K} =
\sum_{\eta\in\Omega_{L,K}} \exp\{-\beta\bb H(\eta)\}$.

Let $D(\bb R_+, \Omega_{L,K})$ be the space of right-continuous with
left limits trajectories $e\dvtx \bb R_+ \to\Omega_{L,K}$ endowed with
the Skorohod topology. For each configuration $\eta\in\Omega_{L,K}$,
denote by $\mb P^\beta_{\eta}$ the probability measure on $D(\bb R_+,
\Omega_{L,K})$ induced by the Markov process $\{\eta_t \dvtx t\ge0\}$
starting from $\eta$. Expectation with respect to $\mb P^\beta_{\eta}$
is represented by $\mb E^\beta_{\eta}$. Sometimes we omit the index
$\beta$ in the notation. The reader should be warned that we consider
in this article two types of asymptotics involving the probability
measure $\mb P^\beta_{\eta}$. In Section~\ref{sec2}, for example, we
examine the limit of certain events in the case where the set $\bb
T_L$ and the number of particles $K$ are fixed while
$\beta\uparrow\infty$. In contrast, in all theorem stated in this
section $L$ and $K$ depend on $\beta$.

Assume from now on that $K=n^2$ for some $n\ge1$. Denote by $Q$ the
square $\{0,\ldots, n-1\}\times\{0,\ldots, n-1\}$. For $\mb x\in\bb
T_L$, let $Q_{\mb x} = \mb x + Q$ and let $\eta^{\mb x}$ be the
configuration in which all sites of the square $Q_{\mb x}$ are
occupied. Denote by $\Gamma$ the set of square configurations
\[
\Gamma = \bigl\{\eta^{\mb x}\dvtx \mb x\in\bb T_L \bigr\}.
\]
A simple computation \cite{bl5} shows that if $L>2n$ the ground states
of the energy $\bb H$ in $\Omega_{L,K}$ are the square configurations
\[
\bb H_{\mathrm{min}}:= \min_{\eta\in\Omega_{L,K}} \bb H(\eta) = \bb H\bigl(
\eta^{\mb x}\bigr) = - 2n(n-1),
\]
and $\bb H(\eta) > -2n(n-1)$ for all $\eta\in\Omega_{L,K} \setminus
\Gamma$.

We examine in this article the asymptotic evolution of the Markov
process $\{\eta_t \dvtx t\ge0\}$ among the $|\bb T_L|$ ground
states $\{\eta^{\mb x}\dvtx \mb x\in\bb T_L \}$ as the temperature
vanishes, while the volume and the number of particles increase not
too fast.

\subsection*{The trace process on $\Gamma$}
\label{sec01}

Denote by $\xi(t)$ the trace of the process $\eta(t)$ on $\Gamma$
and by $r_\beta(\mb x, \mb y) = R_\xi(\eta^{\mb x},\eta^{\mb y})$
the jump rates of the trace process $\xi(t)$. By symmetry, it is clear
that $r_\beta(\mb x, \mb y)= r_\beta(0,\mb y- \mb x) =: r_\beta(\mb y-
\mb x)$.

Let $\mb X(\eta^{\mb x}) = \mb x$, $\mb x\in\bb T_L$, so that $X
(t) = \mb X(\xi(t))$ is a random walk on $\bb T_L$ which jumps from
$\mb x$ to $\mb y$ at rate $r_\beta(\mb y-\mb x)$ and which starts
from the origin. Denote by $X_1(t)$, $X_2(t)$ the components of $X(t)$ and
by $N^{\mb x}_t$ the number of jumps of size $\mb x$ performed by the
random walk $X (s)$ in the time interval $[0,t]$. Since, by symmetry,
$r_\beta(- \mb x) = r_\beta(\mb x)$,
\[
X (t) - X (0) = \sum_{\mb x\in\bb T_L} \mb x
N^{\mb x}_t = \sum_{\mb x\in\bb T_L} \mb x
\bigl\{ N^{\mb x}_t - r_\beta(\mb x) t \bigr\} =
\sum_{\mb x\in\bb T_L} \mb x M^{\mb x}_t,
\]
where $M^{\mb x}_t$, $\mb x\in\bb T_L$, are orthogonal martingales
with quadratic variation $\langle  M^{\mb x}\rangle_t$ equal to $r_\beta(\mb x)
t$. Hence, if we represent by $\langle X_1\rangle_t$, $\langle X_2\rangle_t$ the predictable
quadratic variations of the components of the random walk $X(t)$,
%
\begin{equation}
\label{12} \langle X_1\rangle_t + \langle
X_2\rangle_t = \sum_{\mb x\in\bb T_L}
\Vert\mb x\Vert^2 r_\beta(\mb x) t.
\end{equation}

Let
%
\begin{equation}
\label{16} \frac{1}{\theta_\beta}:= \frac{1}t \mb E_{\eta^{\mb0}} \bigl[
\bigl\Vert X(t)\bigr\Vert^2\bigr] = \frac{1}t \mb E_{\eta^{\mb0}}
\bigl[\langle X_1\rangle_t + \langle X_2
\rangle_t\bigr] = \sum_{\mb x\in\bb T_L} \Vert\mb x
\Vert^2 r_\beta(\mb x).
\end{equation}
Hence, $\theta^{-1}_\beta$ is the diffusion constant of the centre
of the square, that is, the mean-square displacement per unit of time.
We prove in Section~\ref{sec8} that under the assumption that
%
\begin{equation}
\label{c09} \lim_{\beta\to\infty} n^4 L^2
e^{-\beta} =0,
\end{equation}
there exist constants $0<c_0<C_0<\infty$, independent of $\beta$, such
that
%
\begin{equation}
\label{39} c_0 \frac{ n} {L^2} e^{2\beta} \le
\theta_\beta \le C_0 n^2 e^{2\beta}.
\end{equation}
We adopt this convention throughout the paper, $0<c_0<C_0<\infty$ are
constants independent of $\beta$ whose value may change from line to
line. Let
%
\begin{equation}
\label{c13} \mb e (\beta) = L e^{-\beta/2} + \sqrt{n^7
\bigl[n^4e^{-\beta} + n L e^{-\beta/2}\bigr]},
\end{equation}
which is the order of the error which appears in Proposition
\ref{lc01}. Under the stronger assumption that
%
\begin{equation}
\label{c10} \limsup_{\beta\to\infty} \frac{ L^2}{n^2} \mb e (\beta) <
\infty,
\end{equation}
we have that
%
\begin{equation}
\label{c12} c_0 \frac{1} {n^2} e^{2\beta} \le
\theta_\beta \le C_0 n^2 e^{2\beta}.
\end{equation}
We comment below in Remark \ref{sb03} on the exact order of
$\theta_\beta$. Entropy effects, evidenced by the presence of many
paths of the same cost connecting two ground states, turn the
estimation of $\theta_\beta$ into hard problem.

\begin{theorem}
\label{s10}
Assume that $\eta_0 = \eta^{\mb0}$ and that $L=L(\beta)$,
$n=n(\beta)$ depend on the temperature. Consider a sequence
$\ell=\ell(\beta)$ such that $1\ll\ell\le L$, and assume that
%
\begin{equation}
\label{c06} \lim_{\beta\to\infty} \ell L \theta_\beta
e^{-2\beta} \mb e(\beta) = 0,
\end{equation}
and that for all $\delta>0$,
%
\begin{equation}
\label{c07} \lim_{\beta\to\infty} \ell(\ell+ n) \theta_\beta
e^{-2\beta} e^{-\delta\ell/n} = 0.
\end{equation}
Let $Z^\beta(t) = X(t \ell^2 \theta_\beta)/\ell$. Then, as $\beta\to
\infty$, $Z^\beta(t)$ converges in the Skorohod topology to a Brownian
motion.
\end{theorem}

If $\ell=L$, the limiting Brownian motion evolves on the
two-dimensional torus $[-1,1)^2$, while if $\ell\ll L$, the limiting
Brownian motion evolves on $\bb R^2$. We discuss in Remark \ref{sb01}
the assumptions \eqref{c06} and \eqref{c07}.

\textit{Energy landscape.}
Denote by $\bb H_j$, $j\ge0$, the set of configurations with energy
equal to $\bb H_{\mathrm{min}} + j = -2n(n-1)+j$:
\[
\bb H_j = \bigl\{\eta\in\Omega_{L,K} \dvtx \bb H(\eta) =
\bb H_{\mathrm{min}} +j\bigr\},
\]
and let
%
\begin{equation}
\label{04} \Delta_j = \bigl\{\eta\in\Omega_{L,K}\dvtx
\bb H (\eta) > \bb H_{\mathrm{min}} + j \bigr\},
\end{equation}
so that $\bb H_0 = \Gamma$ and $\{\bb H_{0j}, \Delta_j \}$ forms a
partition of the set $\Omega_{L,K}$, where $\bb H_{ij} = \bigcup_{i\le
k\le j} \bb H_k$.\vspace*{1pt}

We have shown in \cite{bl5} that the energy landscape of the Kawasaki
dynamics at low temperature starting from a ground state configuration
has the form illustrated by Figure~\ref{fig1a}. There are subsets
$\Omega^j$, $1\le j\le4$, of $\bb H_1$ and subsets $\Sigma_{j,j+1}$,
$0\le j\le3$, of $\bb H_2$, defined in Section~\ref{sec3}, satisfying
the following properties. Let $\Omega_c = \bb H_1 \setminus[\bigcup_j
\Omega^j]$, $\Sigma_c = \bb H_2 \setminus[\bigcup_j
\Sigma_{j,j+1}]$. Denote by $\Pi$, $\Pi'$ generic sets appearing in
Figure~\ref{fig1a}. The process $\eta_t$ may jump from a
configuration in $\Pi$ to a configuration in $\Pi$ or to a
configuration in $\Pi'$ if $\Pi$ and $\Pi'$ are joined by an edge in
Figure~\ref{fig1a}. In particular, starting from $\Gamma$, the process
$\eta_t$ may only reach $\Omega_c$ by crossing the set~$\Delta_2$.

\begin{figure}[b]

\includegraphics{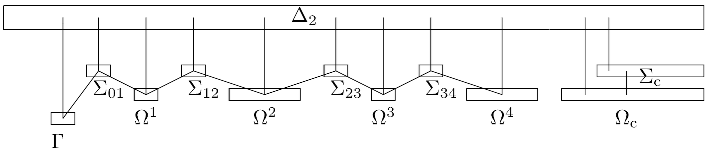}

\caption{The energy landscape of the Kawasaki dynamics at low
temperature. $\Gamma$ represents the set of ground states,
$\Omega^j$, $1\le j\le4$, disjoint subsets of $\bb H_1$,
$\Sigma_{i,i+1}$, $0\le i\le3$, disjoint subsets of $\bb H_2$, and
$\Omega_{\mathrm c}= \bb H_1 \setminus[\bigcup_{1\le j\le4} \Omega^j]$,
$\Sigma_{\mathrm c}= \bb H_2 \setminus[\bigcup_{0\le i\le3}
\Sigma_{i,i+1}]$. At low temperatures, during an
excursion between two ground states, with probability very close
to $1$, the process does not visit the set $\Delta_2$ and all the
analysis is reduced to the lower-left portion of the picture.}
\label{fig1a}
\end{figure}

Let $\Xi_\star$ be the set of configurations which can be reached from
$\eta^{\mb0}$ without crossing the set $\Delta_2$, $\Xi_\star=
\Gamma\bigcup_{1\le i\le4} \Omega^i \bigcup_{0\le j\le3} \Sigma_{j,j+1}$
in the notation of Figure~\ref{fig1a}. This set is described in the
next section and in Section~\ref{sec7}.

\begin{theorem}
\label{s18}
Assume that
%
\begin{equation}
\label{40} \lim_{\beta\to\infty} \ell^2
\theta_\beta e^{-2\beta} \bigl(n^8 + L^2
\bigr) e^{-\beta} = 0.
\end{equation}
Then, for every $t>0$,
%
\begin{equation}
\label{54} \lim_{\beta\to\infty} \mb P_{\eta^{\mb0}} \bigl[
\eta(s)\notin\Xi_\star\mbox{ for some } 0\le s\le t \ell^2
\theta_\beta \bigr] = 0.
\end{equation}
\end{theorem}

The set $\Xi_\star$ is small. In the next sections, we show that this
set has less than $C_0 L^2 (n^8 + L^2)$ elements and that it is
constituted of configurations whose particles form $n\times n$
squares, $(n+1)\times(n-1)$ and $(n+2)\times(n-2)$ rectangles.

Configurations whose particles form a $(n+k)\times(n-k)$ rectangle,
$4<k^2<n+k$, and whose remaining $k^2$ particles form a $1 \times k^2$
rectangle attached to one side of the large rectangle belong to $\bb
H_1 \cap\Xi_\star^c$. It follows from this observation and from the
definition of the set $\Xi_\star$, presented in the next section,
that even the cardinality of $\Xi_\star\cap\bb H_1$ is small
compared to the cardinality of~$\bb H_1$.

The previous two remarks show that in the time scale $\ell^2
\theta_\beta$ the Kawasaki dynamics is confined to a tiny portion of
the state space. It is this rather simple energy landscape which
permits the analysis of the evolution of the center of mass. In
dimension $3$, for example, the energy landscape is much more
complicated insomuch that, though we believe that a similar result holds,
the proof requires a much more intricated argument.

\textit{Evolution of the center of mass.}
Due to the periodic boundary conditions, the center of mass of a
configuration may not be well defined. However, the particles of a
configuration $\eta\in\Xi_\star\setminus(\Sigma_{0,1} \cup
\Sigma_{2,3})$ form a connected set and have, therefore, a well
defined center of mass. On the other hand, the particles of a
configuration $\eta$ in $\Sigma_{0,1}$ form a square in which a
particle at a corner of the square has been moved to a site whose
neighbors are vacant or became vacant after the displacement of the
particle at the corner. The detached particle may create an ambiguity
in the definition of the center of mass of $\eta$. To avoid this
problem we define the center of mass of a configuration $\eta\in
\Sigma_{0,1}$ as the center of mass of the large connected
component. Of course, if $n$ is large and $L$ is not too large
compared to $n^2$ the detached particle does not affect the center of
mass. This analysis extends to configurations in $\Sigma_{2,3}$. With
the previous convention and since by Theorem \ref{s18} the process is
confined to $\Xi_\star$ in the time scale $\ell^2 \theta_\beta$, we
may examine the evolution of the center of mass without ambiguity.

Denote by $\texttt{M}(\eta)$ the center of mass of a configuration
$\eta\in\Xi_\star$ and set, say, $\mt M(\xi)=\mb0$ for configurations
$\xi\notin\Xi_\star$. Denote by $[a]$ the integer part of $a>0$.

\begin{theorem}
\label{sc01}
Assume that $\eta(0)=\eta^{-\mb n}$, where $\mb n = ([n/2], [n/2])$,
and that the hypotheses of Theorems \ref{s10} and \ref{s18} are in
force. Suppose also that $(n^7+L)e^{-\beta}\to0$. Let $\mt M^\beta(t)
= \mt M(\eta(t \ell^2 \theta_\beta))/\ell$. There exists $c_0>0$ such
that if $n\ge c_0$ for all $\beta$ large enough, then, as
$\beta\uparrow\infty$, $\mt M^\beta(t)$ converges in the Skorohod
topology to a Brownian motion.
\end{theorem}

As in Theorem \ref{s10}, the limiting Brownian motion evolves on the
torus $[-1,1)^2$ if $\ell=L$ and on $\bb R^2$ if $\ell\ll L$.

\textit{Dynamics of the trace process $\xi(t)$}. To examine the
evolution of the trace of $\eta_t$ on the set $\Gamma$ of ground
states, we consider the trace of $\eta_t$ on the larger space $\Xi_{1}
= \Gamma\cup\Omega^1 \cup\Omega^3$. The set $\Omega^1$, defined in
Section~\ref{sec3}, is formed by configurations whose particles form a
$n\times n$ square in which one particle at the corner of the square
has been displaced and attached to one of the sides of the
square. Figure~\ref{fig2} shows configurations in $\Omega^1$. The set
$\Omega^3$ contains all configurations whose particles form a $(n\pm
1) \times(n\mp1)$ rectangle with an extra particle attached to one
of the sides. Figure~\ref{fig4} shows configurations in~$\Omega^3$.

\begin{figure}

\includegraphics{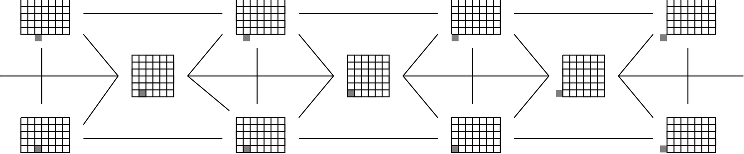}

\caption{The dynamics of the trace of $\eta_t$ on $\Gamma
\cup\Omega^1 \cup\Omega^3$.}
\label{fig9}
\end{figure}

Figure~\ref{fig9} illustrates the dynamics of the trace of $\eta_t$ on
$\Xi_{1}$ in the case $n=6$, denoted by $\zeta_1(t)$. A $1\times1$
square has been placed at each vertex occupied by a particle and the
gray square indicates the origin. The $n\times n$ squares represent
the set of configurations formed by a ground state, say $\eta^{\mb
x}$, as well as the $4(4n-2)$ configurations in $\Omega^1$ obtained
from $\eta^{\mb x}$ by displacing a particle at the corner of the
square to one of the sides. There are $16$ different types of such
configurations, as we may choose a particle from $4$ corners and move
it to $4$ sides. Each rectangle represents the set of $4n$
configurations whose particles form a $(n+1)\times(n-1)$ rectangle
with an extra particle attached to one of the sides.

Figure~\ref{fig9} presents a piece of a periodic horizontal band of
the graph which describes the dynamics of $\zeta_1(t)$. To form the
complete graph, one has first to replicate this periodic horizontal
band in the vertical direction to obtain a graph also periodic in the
vertical direction. Denote by $\bb G$ the graph obtained. To complete
the construction, one has to superpose $\bb G$ with the graph $\bb G$
rotated by $90^\circ$. In this rotated graph, the $(n+1)\times(n-1)$
rectangles become $(n-1)\times(n+1)$ rectangles.

Recall that each square or rectangle in Figure~\ref{fig9} represents a
set of configurations. Two sets $\Pi_1$, $\Pi_2$ are joined by a line
if the process $\zeta_1(t)$ may jump from one configuration of $\Pi_1$
to a configuration in $\Pi_2$. For example, $\zeta_1(t)$ jumps from a
square to the rectangle on the right and above the square by sliding
the small squares at the bottom of the square to its right side. In
fact, all jumps occur from such displacements of particle along the
sides.

Figure~\ref{fig9} also illustrates why the trace process $\xi(t)$ may
have long jumps. For this to happen, it is enough that the process
$\zeta_1(t)$ jumps from a square configuration to an upper rectangle
and then move from rectangle to rectangle in the horizontal
direction. The figure indicates that the probability that $\xi(t)$
performs a jump of length $\mb x$ should decay exponentially in $\mb
x$ at a rate independent of $\beta$.

The diagrams in Figures~\ref{fig1a}, \ref{fig9} have been used
extensively in the literature. We refer to \cite{hos1,gos,bhn1,bhs1} in the conservative case.

\begin{remark}[(The asymptotic behavior of $\theta_\beta$)]
\label{sb03} We
believe that the order of magnitude of $\theta_\beta$ is given by
\[
\frac{1}{\mu_\zeta(\eta^{\mb0})} \Cap_\zeta\bigl(\eta^{\mb0}, \Pi\bigr) \sim
\frac{1}{\mu_\beta(\eta^{\mb0})} \Cap\bigl(\eta^{\mb0}, \Pi\bigr).
\]
In this formula, $\zeta(t)$ is the trace process introduced in Section~\ref{sec2}, $\mu_\zeta$ its stationary state, $\Cap_\zeta$ the
associated capacity and $\Pi$ the union of all sets joined by a line
to $\eta^{\mb0}$ in the graph partially drawn in Figure~\ref{fig9}. If this insight is correct, by the computation presented
at the end of Section~\ref{sec4} the order of magnitude of
$\theta_\beta$ lies between $n e^{2\beta}$ and $n^2
e^{2\beta}$.

By definition \eqref{16} and by formula \eqref{38},
%
\begin{equation}
\label{b05} \theta_\beta^{-1} = \lambda_\zeta
\bigl(\eta^{\mb0}\bigr) \sum_{\mb x\in\bb T_L} \Vert\mb x
\Vert^2 \mb P_{\eta^{\mb0}} \bigl[H^+_\Gamma=
H_{\eta^{\mb x}}\bigr],
\end{equation}
where $\lambda_\zeta(\eta^{\mb0})$ represents the holding rate at
$\eta^{\mb0}$ of the trace process $\zeta(t)$, and $H_\Pi$,
$H^+_{\Pi}$ the hitting time of $\Pi$ and the return time to $\Pi$,
respectively. All these notions are defined in Section~\ref{sec2}. It
follows from the previous equation and from the definition of the
capacity that
\[
\theta_\beta^{-1} \ge \lambda_\zeta\bigl(
\eta^{\mb0}\bigr) \mb P_{\eta^{\mb0}} \bigl[H^+_\Gamma<
H^+_{\eta^{\mb0}}\bigr] = \frac{1}{\mu_\zeta(\eta^{\mb0})} \Cap_\zeta\bigl(
\eta^{\mb0}, \Gamma\setminus\eta^{\mb0}\bigr).
\]
By Lemma \ref{s11}, this expression is bounded below by $c_0
e^{-2\beta} n^{-2}$, which is the upper bound presented in \eqref{39}.
On the other hand, $\Cap_\zeta(\eta^{\mb0}, \Gamma\setminus
\eta^{\mb0})$ should be of the order of $\Cap_\zeta(\eta^{\mb0},
\Pi)$ which explains the first half of our claim.

To obtain an upper bound for $\theta_\beta^{-1}$, note that by
\eqref{b05},
\[
\theta_\beta^{-1} \le \lambda_\zeta\bigl(
\eta^{\mb0}\bigr) \mb P_{\eta^{\mb0}} \bigl[H_\Pi<
H^+_{\eta^{\mb0}}\bigr] \sum_{\mb x\in\bb T_L} \Vert\mb x
\Vert^2 \max_{\eta\in\Pi} \mb P_{\eta}
[H_\Gamma= H_{\eta^{\mb x}}].
\]
We have argued right before this remark that $\mb P^\zeta_{\eta}
[H^+_\Gamma= H_{\eta^{\mb x}}]$ decreases exponentially fast in
$\Vert\mb x\Vert$, at a rate independent of $\beta$. In particular,
the previous sum should be bounded by a finite constant independent of
$\beta$. On the other hand, by the definition of the capacity
\[
\lambda_\zeta\bigl(\eta^{\mb0}\bigr) \mb P_{\eta^{\mb0}}
\bigl[H_\Pi< H^+_{\eta^{\mb0}}\bigr] = \frac{1}{\mu_\zeta(\eta^{\mb0})}
\Cap_\zeta\bigl(\eta^{\mb0}, \Pi\bigr) \cdot
\]
By the proof of Lemma \ref{s11}, this expression is bounded above by
$C_0 n^{-1} e ^{-2\beta}$.

The exact asymptotic behavior of $\theta_\beta$ requires sharp
estimates of the jump rates of the process described in Figure~\ref{fig9}. By symmetry, there are only few different rates. To
estimate these rates, one has to examine the dynamics in which
particles slide along the sides of the square. This process, denoted
by $\zeta(t)$ in Section~\ref{sec2}, is a reversible, symmetric Markov
process on a seven-dimensional graph in which each vertex has degree
less than or equal to eight. Such estimates are certainly not easy,
but are conceivable. In possession of these estimates, one still has to
compute the probability that the process described in Figure~\ref{fig9} starting from $\eta^{\mb0}$ return to $\Gamma$ at
$\eta^{\mb x}$.
\end{remark}

\begin{remark}[(The hypotheses and the strategy)]
\label{sb01}
Condition \eqref{c06} has to be understood as follows. $\ell^2
\theta_\beta$ is the time scale. Each time the process visits a square
(a configuration in $\Gamma$) it remains there an exponential time
whose mean is of order $e^{2\beta}$. Therefore, $\ell^2 \theta_\beta
e^{-2\beta}$ represents the number of excursions between two
consecutive visits to $\Gamma$. In condition \eqref{c06}, one $\ell$
has been replaced by $L$ in the argument due to our incapacity to
estimate sharply the hitting probabilities \eqref{b05}.

The main idea of the article consists in taking a temperature low
enough, with respect to $n$ and $L$, to be allowed to discard all
jumps of lower order in $\beta$. For example, a square configuration
has $8$ jumps whose rates are $e^{-2\beta}$ and $4(n-2)$ jumps whose
rates are $e^{-3\beta}$. Hence, if $\ell^2 \theta_\beta e^{-2\beta}
ne^{-\beta}$ vanishes as $\beta\uparrow\infty$, a particle which is
not at the corner of a square jumps in a time interval $[0, t \ell^2
\theta_\beta]$, $t>0$, with a probability converging to $0$.

If we discard such jumps, from a square the process $\eta(t)$ can only
reach a configuration in which a particle at the corner of the square
has detached itself from the square. At this point, the detached
particle performs a rate one symmetric random walk on $\bb T_L$ until
it reaches a side of this square. While this particle moves, other
particles may also jump. There are a few rate $e^{-\beta}$ and rate
$e^{-2\beta}$ jumps and $O(n)$ rate $e^{-3\beta}$ jumps. In order to
apply the approach presented here, we need to guarantee that the
detached particle returns to the square before any another particle
moves. We have seen that a particle in the square jumps at rate at
most $e^{-\beta}$ while the detached particle returns to the square in
a time of order $L^2$. Hence, the probability that a lower order jump
occurs before the detached particle returns should be of order $L^2
e^{-\beta}$. We are only able to estimate this probability, in
Proposition \ref{s16}, by $L e^{-\beta/2}$. Thus, as in the previous
paragraph, to guarantee that no lower order jump occurs in the time
scale in which the center of mass evolves we need to impose that
$\ell^2 \theta_\beta e^{-2\beta} Le^{-\beta/2}$ vanishes. This
explains the factor $L e^{-\beta/2}$ appearing in the error term $\mb
e(\beta)$ introduced in \eqref{c13}.

Proceeding in this way, we face many cases presented in Sections~\ref{sec3}, \ref{sec2} and \ref{sec6}. The error term $\mb e(\beta)$
summarizes all the constraints, it represents the bound obtained in
Proposition \ref{lc01} for the probability that a lower order jump
occurs before the process returns to a square configuration when it
starts from a square configuration. This estimate is not sharp and
the conditions on the growth of $L$ and $n$ can certainly be improved.

Condition \eqref{c07} is spurious and comes from the fact that we are
not able to estimate correctly the hitting probabilities appearing in
\eqref{b05}. It can certainly be suppressed, but its removal depends
on the estimation of hitting probabilities of a seven dimensional
random walk.

Finally, note that in Theorems \ref{s10} and \ref{sc01}, $n$ has to be
large, but does not need to increase with $\beta$.
\end{remark}

\begin{remark}
\label{sb02}
The approach presented in this article can be applied to the case in
which the total number of particles is $n(n+k)$, for some integer $k$
independent of $\beta$, and the particles of the initial configuration
form a $n\times(n+k)$ rectangle. The energy landscape becomes more
complex for $k$ large and computations harder. For example, for $k=2$,
a $(n+1)\times(n+1)$ square without a particle at a corner is a
ground state.
\end{remark}

Since its origins, the Ising model has been known to properly represent
the condensation phenomena in two-dimensional systems formed by the
adsorption of gases on the surfaces of crystals \cite{p1,o1}. There
is a huge literature on this subject. In this perspective, the
present work can be seen as an investigation of the evolution of the
condensate.

The metastable behavior of the Ising lattice gas evolving under the
noncon\-serv\-a\-tive Glauber dynamics at very low temperature has
been examined 20 years ago and is well understood \cite{ov1}. In
contrast, there are not many results on the asymptotic behavior of the
Ising lattice gas under the conservative Kawasaki dynamics at very low
temperature.

This problem has been examined in two situations. Den Hollander et
al. \cite{hos1} and Gaudilli\`ere et al. \cite{gos} described the
critical droplet, the nucleation time and the typical trajectory
followed by the process during the transition from a metastable set to
the stable set in a two-dimensional Ising lattice gas evolving under
the Kawasaki dynamics at very low temperature and very low density in
a finite square in which particles are created and destroyed at the
boundary. This result has been extended to the anisotropic case by
Nardi et al. \cite{nos1} and to three dimensions by den Hollander et
al. \cite{hnos1}. Bovier et al. \cite{bhn1} presented the detailed
geometry of the set of critical droplets and provided sharp estimates
for the expectation of the nucleation time for this model in dimension
two and three.

More recently, Gaudilli\`ere et al. \cite{ghnos} proved that the
dynamics of particles evolving according to the Kawasaki dynamics at
very low temperature and very low density in a two-dimensional torus,
whose length increases as the temperature decreases, can be
approximated by the evolution of independent particles. Bovier et al.
\cite{bhs1} obtain sharp estimates for the expectation of the
nucleation time for this model where the length of the square
increases as the temperature vanishes.

Finally, Funaki \cite{f1,f2} examined, in the zero-temperature limit,
the motion of a rigid body formed by a system of interacting Brownian
particles with a pairwise potential in $\bb R^d$, and the nucleation
in the one-dimensional case.

The article is divided as follows. In the next section, we introduce a
subset $\Xi$ of $\Omega_{L,K}$. In Sections~\ref{sec2} and \ref{sec6},
we obtain sharp estimates for the jump rates of $\zeta(t)$, the trace
of $\eta(t)$ on $\Xi$, and we introduce a process $\widehat{\zeta}(t)$
whose rate jumps are close to the ones of $\zeta(t)$. In Section~\ref{sec9}, we couple $\zeta(t)$ and $\widehat{\zeta}(t)$ and we obtain
an estimate on the time these processes remain coupled. In Section~\ref{sec4}, we prove Theorem \ref{s10}, in Section~\ref{sec7}, Theorem
\ref{s18}, and in Section~\ref{sec08}, Theorem \ref{sc01}. We conclude
the article in Section~\ref{sec8} with estimates on the time-scale
$\theta_\beta$.

\section{The subset of configurations \texorpdfstring{$\Xi$}{$Xi$}}
\label{sec3}

We introduce in this section a subset $\Xi$ of $\bb H_{01}$, and we
estimate, in the next section, the jump rates of the trace of
$\eta(t)$ on~$\Xi$.

\subsection{The configurations \texorpdfstring{$\eta^{i,j}_{\mathbf{x}}$}{$eta^{i,j}_{\mathbf{x}}$}}
\label{ss1}

For a subset $B$ of $\bb T_L$, denote by $\partial_- B$,
$\partial_+ B$ the inner, outer boundary of $B$,
respectively. These are the set of sites which are at distance one
from $B^c$, $B$, respectively,
\begin{eqnarray*}
 \partial_- B& =& \bigl\{\mb x \in B \dvtx \exists \mb y \in B^c, |
\mb y - \mb x|=1 \bigr\},
\\
 \partial_+ B &=& \{\mb x \in\bb T_L\setminus B \dvtx \exists \mb y
\in B, |\mb y - \mb x|=1 \},
\end{eqnarray*}
where $| \cdot|$ stands for the sum norm, $\mb x = (x_1, x_2)$,
$|\mb x|= |x_1| + |x_2|$.

Let $Q^{i} = Q \setminus\{\mb w_i \}$, $0\le i\le3$, where
\begin{eqnarray*}
\mb w_0& =& \mb w = (0,0),\qquad \mb w_1 = (n-1,0), \\
\mb
w_{2} &=& (n-1,n-1), \qquad\mb w_3 = (0,n-1)
\end{eqnarray*}
are the corners of the square $Q$. For $\mb x\in\bb T_L$, $0\le i\le
3$, let $\mb x_i = \mb x + \mb w_i$, $Q_{\mb x}^{i} = \mb x + Q^{i}$.
We refer to the sets $Q_{\mb x}^{i}$ as quasi-squares, a square with a
corner missing.

Let $e_1=(1,0)$, $e_2=(0,1)$ be the canonical basis of $\bb R^2$, and
denote by $\partial_j Q_{\mb x}^{i}$, $0\le j\le3$, the $j$th outer
boundary of $Q_{\mb x}^{i}$:
\begin{eqnarray*}
 \partial_j Q_{\mb x}^{i}& =&\bigl\{\mb z\in
\partial_+ Q_{\mb x}^{i} \dvtx \exists \mb y\in
Q_{\mb x}^{i}; \mb y -\mb z = (1-j) e_2\bigr\},\qquad
j=0,2,
\\
 \partial_j Q_{\mb x}^{i} &=&\bigl\{\mb z\in
\partial_+ Q_{\mb
x}^{i} \dvtx \exists \mb y\in
Q_{\mb x}^{i}; \mb y -\mb z = (j-2) e_1\bigr\},\qquad
j=1,3.
\end{eqnarray*}
Let $Q^{i,j}_{\mb x} = \partial_j Q_{\mb x}^{i} \setminus Q_{\mb x}$,
$0\le i,j\le3$, $\mb x\in\bb T_L$, let $\mc E^{i,j}_{\mb x}$ be the
set of configurations in which all sites of the set $Q^{i}_{\mb
x}$ are occupied with an extra particle at some location of
$Q^{i,j}_{\mb x}$ and let $\Omega^1 = \Omega^1_{L,K}$ be the set of
all such configurations
\[
\mc E^{i,j}_{\mb x} = \bigl\{\sigma^{\mb x_i, \mb z}
\eta^{\mb x} \dvtx \mb z\in Q^{i,j}_{\mb x}\bigr\},\qquad
\Omega^1 = \bigcup_{\mb x\in\bb T_L} \bigcup
_{i,j=0}^3 \mc E^{i,j}_{\mb x}.
\]

We proved in \cite{bl5} that for any configuration $\xi\in\mc
E^{i,j}_{\mb x}$, the triple $(\mc E^{i,j}_{\mb x}, \mc E^{i,j}_{\mb
x}\cup\Delta_1, \xi)$ is a valley of depth $e^\beta$ in the sense
of \cite{bl2}, Definition 2.1. To select a representative
configuration of the well $\mc E^{i,j}_{\mb x}$, denote by $\hat{\mb
w}_j$, $0\le j\le3$, the point in the inner boundary of $Q$ whose
distance to $\mb w_j$ and to $\mb w_{j+1}$ is equal to $(n-1)/2$ if
$n$ is odd, and whose distance to $\mb w_j$, $\mb w_{j+1}$ is equal to
$(n/2) -1$, $n/2$, respectively, if $n$ is even. Let $\mb w_j^*$ the
site on the outer boundary of $Q$ at distance one from $\hat{\mb
w}_j$, and let $\eta^{i,j}_{\mb x}$, $\mb x \in\bb T_L$, be the
configuration $\sigma^{\mb x + \mb w_i,\mb x + \mb w^*_j}\eta^{\mb
x}$. Denote by $\widehat\Omega^1$ the space of such configurations
\[
\widehat\Omega^1_{\mb x} = \bigl\{ \eta^{i,j}_{\mb x}
\dvtx 0\le i,j\le3 \bigr\}, \qquad\widehat\Omega^1 = \bigcup
_{\mb x \in\bb T_L} \widehat\Omega^1_{\mb x}.
\]

\begin{figure}

\includegraphics{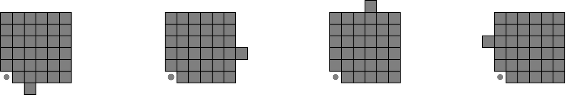}

\caption{The configurations $\eta^{0,0}_{\mb x}$, $\eta^{0,1}_{\mb
x}$, $\eta^{0,2}_{\mb x}$ and $\eta^{0,3}_{\mb x}$ for $n=6$. The
gray dot indicates the site $\mb x$. We placed a square
$[-1/2,1/2)^2$ around each particle.}
\label{fig2}
\end{figure}

\subsection{The configurations \texorpdfstring{$\eta_{\mathbf{x}}^{\mathfrak{a},(\mathbf{k},\bolds{\ell})}$}{$eta_{x}^{a,(k,ell)}$}}
\label{ss2}

Let $R^{\mf l}$, $R^{\mf s}$ be the rectangles $R^{\mf l} =\{1,\ldots,\break
n-1\}\times\{1, \ldots, n-2\}$, $R^{\mf s} = \{1, \ldots, n-2\} \times
\{1,\ldots, n-1\}$, where $\mf l$ stands for lying and $\mf s$ for
standing. Let $n_0^{\mf s} = n_2^{\mf s} = n-2$, $n_1^{\mf s} =
n_3^{\mf s} = n-1$ be the length of the sides of the standing
rectangle $R^{\mf s}$. Similarly, denote by $n_i^{\mf l}$, $0\le i\le
3$, the length of the sides of the lying rectangle $R^{\mf l}$:
$n_i^{\mf l} = n_{i+1}^{\mf s}$, where the sum over the index $i$ is
performed modulo $4$.

Denote by $\bb I_{\mf a}$, $\mf a \in\{\mf s,\mf l\}$, the set of
pairs $(\mb k, \bs\ell) = (k_0,\ell_0; k_1,\ell_1; k_2,\ell_2;
k_3,\ell_3)$ such that:
\begin{itemize}
\item$0\le k_i \le\ell_i \le n_i^{\mf a}$,
\item if $k_j=0$, then $\ell_{j-1}= n_{j-1}^{\mf a}$.
\end{itemize}
For $(\mb k,\bs\ell) \in\bb I_{\mf a}$, $\mf a \in\{ \mf s, \mf
l\}$, let $R^{\mf l}(\mb k,\bs\ell)$, $R^{\mf s}(\mb k,\bs\ell)$ be
the sets
\begin{eqnarray*}
R^{\mf l} (\mb k,\bs\ell) & =& R^{\mf l} \cup \bigl\{(a,0) \dvtx
k_0\le a\le\ell_0\bigr\} \cup \bigl\{(n,b) \dvtx
k_1\le b\le\ell_1\bigr\}
\\
&&{} \cup \bigl\{(n-a,n-1) \dvtx k_2\le a\le\ell_2\bigr\}
\cup \bigl\{(0,n-1-b) \dvtx k_3\le b\le\ell_3\bigr\},
\\
R^{\mf s} (\mb k,\bs\ell) & =& R^{\mf s} \cup \bigl\{(a,0) \dvtx
k_0\le a\le\ell_0\bigr\} \cup \bigl\{(n-1,b) \dvtx
k_1\le b\le\ell_1\bigr\}
\\
&&{} \cup \bigl\{(n-1-a,n) \dvtx k_2\le a\le\ell_2\bigr\}
\cup \bigl\{(0,n-b) \dvtx k_3\le b\le\ell_3\bigr\}.
\end{eqnarray*}
Note that a hole between particles on a side of a rectangle is not
allowed in the sets $R^{\mf a} (\mb k,\bs\ell)$, $R^{\mf s} (\mb
k,\bs\ell)$.

Denote by $I_{\mf a}$, $\mf a \in\{ \mf s, \mf l\}$, the set of pairs
$(\mb k, \bs\ell)\in\bb I_{\mf a}$ such that $|R^{\mf a} (\mb k,\bs
\ell)|=n^2$. For $(\mb k,\bs\ell) \in I_{\mf a}$, denote by $M_i(\mb
k,\bs\ell)$ the number of particles attached to the side $i$ of the
rectangle $R^{\mf a} (\mb k,\bs\ell)$:
\[
M_i(\mb k,\bs\ell) = \cases{ \ell_i-k_i+1, &\quad
$\mbox{if $k_{i+1}\ge1,$}$\vspace*{2pt}
\cr
\ell_i-k_i+2,
&\quad $\mbox{if $k_{i+1}=0$}$.}
\]
Clearly, for $(\mb k, \bs\ell)\in I_{\mf a}$, $\sum_{0\le i\le3}
M_i(\mb k,\bs\ell) = 3n-2 + A$, where $A$ is the number of occupied
corners, which are counted twice since they are attached to two sides.

Denote by $I^*_{\mf a} \subset I_{\mf a}$, the set of pairs $(\mb k,
\bs\ell)\in I_{\mf a}$ whose rectangles $R^{\mf a} (\mb k,\bs\ell)$
have at least two particles on each side: $M_i(\mb k,\bs\ell)\ge2$,
$0\le i\le3$. Note that if $(\mb k, \bs\ell)$ belongs to $I^*_{\mf
a}$, for all $\mb x\in R^{\mf a} (\mb k,\bs\ell)$, there exist $\mb
y$, $\mb z \in R^{\mf a} (\mb k,\bs\ell)$, $\mb y\neq \mb z$, with
the property $|\mb x-\mb y|=|\mb x-\mb z|=1$.

\begin{figure}[b]

\includegraphics{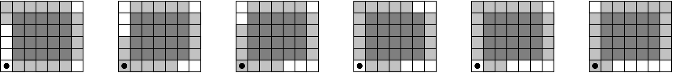}

\caption{Some configurations $\eta_{\mb x}^{\mf l, (\mb k,\bs\ell)}$
for $n=6$. The first one corresponds to the vector $(\mb k,\bs
\ell)= ((1,5); (1,4);(1,6); (0,1))$ and the last one to the vector
$(\mb k,\bs\ell)= ((0,1); (1,5); (0,5);\break (1,5))$. The inner gray
rectangle represents the set $\mb x + R^{\mf l}$ and the black dot
the site~$\mb x$.}\label{fig44}
\end{figure}

For $(\mb k,\bs\ell) \in I_{\mf a}$, $\mf a \in\{ \mf s, \mf l\}$,
$\mb x\in\bb T_L$, let $R^{\mf a}_{\mb x} (\mb k,\bs\ell) = \mb x
+ R^{\mf a} (\mb k,\bs\ell)$, and let $\eta_{\mb x}^{\mf a, (\mb
k,\bs\ell)}$ represent the configurations defined by
\[
\mbox{$\eta_{\mb x}^{\mf a, (\mb k,\bs\ell)} (\mb y) = 1$\qquad if and only if $\mb y \in
R^{\mf a}_{\mb x} (\mb k,\bs\ell)$}.
\]
Figure \ref{fig44}
presents examples of configurations $\eta_{\mb x}^{\mf l, (\mb k,\bs \ell)}$.
The configurations $\eta_{\mb x}^{\mf a, (\mb k,\bs\ell)}$, $(\mb
k,\bs\ell) \in I_{\mf a}$, belong to $\bb H_1$. Moreover,
the configurations $\eta_{\mb x}^{\mf a, (\mb k,\bs\ell)}$, $(\mb
k,\bs\ell) \in I_{\mf a} \setminus I^*_{\mf a}$, belong to $\Omega^1$
or form a $(n-1)\times(n+1)$ rectangle of particles with one extra
particle attached to a side of length $n+1$. Let $\Omega^2 =
\Omega^2_{L,K}$, be the set of configurations associated to the pairs
$(\mb k,\bs\ell)$ in $I^*_{\mf a}$:
\[
\Omega^2_{\mb x} = \bigl\{ \eta_{\mb x}^{\mf a, (\mb k,\bs\ell)}
\dvtx \mf a \in\{ \mf s, \mf l\}, (\mb k,\bs\ell) \in I^*_{\mf a} \bigr\},\qquad
\Omega^2 = \bigcup_{\mb x\in\bb T_L}
\Omega^2_{\mb x}.
\]

\subsection{The configurations \texorpdfstring{$\zeta^{\mathbf{a},j}_{\mathbf{x}}$}{$zeta^{\mathbf{a},j}_{\mathbf{x}}$}}
\label{ss3}

Let $T^{\mf l}$, $T^{\mf s}$ be the rectangles $T^{\mf l} =\{0,\ldots,
n\}\times\{0, \ldots, n-2\}$, $T^{\mf s} = \{0, \ldots, n-2\} \times
\{0,\ldots, n\}$, where $\mf l$ stands for lying and $\mf s$ for
standing. Denote by $T^{\mf a}_{\mb x}$, $\mf a \in\{ \mf s, \mf l\}$,
$\mb x\in\bb T_L$, the rectangle $T^{\mf a}$ translated by $\mb x$:
$T^{\mf a}_{\mb x} = {\mb x} + T^{\mf a}$, and by $\eta_{\mf a}^{\mb
x}$ the configuration in which all sites of $T^{\mf a}_{\mb x}$ are
occupied. Note that $\eta_{\mf a}^{\mb x}$ belongs to $\Omega_{L,K-1}$
and not to $\Omega_{L,K}$.

For $\mf a \in\{ \mf s, \mf l\}$, $\mb x\in\bb T_L$, $\mb z\in
\partial_+ T^{\mf a}_{\mb x}$, denote by $\eta_{\mf a}^{\mb x, \mb z}$
the configuration in which all sites of the rectangle $T^{\mf a}_{\mb
x}$ and the site $\mb z$ are occupied: $\eta_{\mf a}^{\mb x, \mb z}
= \eta_{\mf a}^{\mb x} + \mf d_z$, where $\mf d_y$, $y\in\bb T_L$,
represents the configuration with a unique particle at $y$ and
summation of configurations is performed componentwise.

Denote by $\partial_j T^{\mf a}_{\mb x}$, $0\le j\le3$, the $j$th
outer boundary of $T^{\mf a}_{\mb x}$:
\begin{eqnarray*}
 \partial_jT^{\mf a}_{\mb x} &=& \bigl\{\mb z\in
\partial_+ T^{\mf a}_{\mb x} \dvtx \exists \mb y\in
T^{\mf a}_{\mb x}; \mb y -\mb z = (1-j) e_2\bigr\},\qquad
j=0,2,
\\
\partial_j T^{\mf a}_{\mb x} &= &\bigl\{\mb z\in
\partial_+ T^{\mf a}_{\mb x} \dvtx \exists \mb y\in
T^{\mf a}_{\mb x}; \mb y -\mb z = (j-2) e_1\bigr\},\qquad
j=1,3.
\end{eqnarray*}
Let $\mc E^{\mf a,j}_{\mb x}$ be the set of configurations in which
all sites of the set $T^{\mf a}_{\mb x}$ are occupied with an extra
particle at some location of $\partial_jT^{\mf a}_{\mb x}$ and let
$\Omega^3= \Omega^3_{L,K}$ be the set of all such configurations:
\[
\mc E^{\mf a,j}_{\mb x} = \bigl\{ \eta_{\mf a}^{\mb x, \mb z}
\dvtx \mb z\in\partial_j T^{\mf a}_{\mb x}\bigr\},\qquad
\Omega^3 = \bigcup_{\mb x\in\bb T_L} \bigcup
_{\mf a =
\mf s, \mf l} \bigcup_{j=0}^3
\mc E^{\mf a,j}_{\mb x}.
\]

We proved in \cite{bl5} that for any configuration $\xi\in\mc E^{\mf
a,j}_{\mb x}$, the triple $(\mc E^{\mf a,j}_{\mb x}, \mc E^{\mf
a,j}_{\mb x}\cup\Delta_1, \xi)$ is a valley of depth $e^\beta$. We
select a representative of the well $\mc E^{\mf a,j}_{\mb x}$.

\begin{figure}

\includegraphics{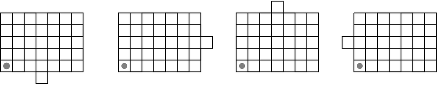}

\caption{The configurations $\zeta^{\mf l,i}_{\mb x}$, $0\le i\le3$,
for $n=6$. The gray dot represents the site $\mb x$.}
\label{fig4}
\end{figure}

Fix $\mf a = \mf s$, $\mf l$, and denote by ${\mb w}^*_{\mf a, j}$,
$0\le j\le3$, the mid-points of the outer boundary of $T^{\mf
a}$. Hence, if $n$ is even ${\mb w}^*_{\mf l, 0}= (n/2, -1)$, ${\mb
w}^*_{\mf l, 1} = (n+1, (n-2)/2)$, ${\mb w}^*_{\mf l, 2} = (n/2,
n-1)$, ${\mb w}^*_{\mf l, 3} = (-1, (n-2)/2)$, and if $n$ is odd ${\mb
w}^*_{\mf l, 0} = ((n-1)/2, -1)$, ${\mb w}^*_{\mf l, 1} = (n+1,
(n-3)/2)$, ${\mb w}^*_{\mf l, 2} = ((n+1)/2, n-1)$, ${\mb w}^*_{\mf l,
3} = (-1, (n-1)/2)$, with a similar definition when $\mf s$ replaces
$\mf l$. Let $\zeta^{\mf a,j}_{\mb x}$, $\mb x \in\bb T_L$, $\mf
a\in\{\mf s, \mf l\}$, $0\le j\le3$, be the configuration $\eta^{\mb
x}_{\mf a} + \mf d_{\mb x + {\mb w}^*_{\mf a, j}}$ and denote by
$\widehat\Omega^3$ the space of such configurations:
\[
\widehat\Omega^3_{\mb x} = \bigl\{ \zeta^{\mf a,j}_{\mb x}
\dvtx \mf a \in\{ \mf s, \mf l\}, 0\le j\le3 \bigr\},\qquad \widehat
\Omega^3 = \bigcup_{\mb x \in\bb T_L} \widehat
\Omega^3_{\mb x}.
\]

\subsection{The configurations \texorpdfstring{$\zeta_{\mathbf{x}}^{\mathbf{a},(\mathbf{k},\bolds{\ell})}$}{$zeta_{x}^{a,(k,ell)}$}}
\label{ss4}

Let $R^{2,\mf l}$, $R^{2,\mf s}$ be the rectangles $R^{2,\mf l}
=\{1,\ldots, n\}\times\{1, \ldots, n-3\}$, $R^{2,\mf s} = \{1, \ldots,
n-3\} \times\{1,\ldots, n\}$. Let $n_0^{2,\mf s} = n_2^{2,\mf s} =
n-3$, $n_1^{2,\mf s} = n_3^{2,\mf s} = n$ be the length of the sides
of the standing rectangle $R^{2,\mf s}$. Similarly, denote by
$n_i^{2,\mf l}$, $0\le i\le3$, the length of the sides of the lying
rectangle $R^{2,\mf l}$: $n_i^{2,\mf l} = n_{i+1}^{2,\mf s}$, where
the sum over the index $i$ is performed modulo $4$.

Denote by $\bb I_{2,\mf a}$, $\mf a \in\{\mf s,\mf l\}$, the set of
pairs $(\mb k, \bs\ell)$ such that:
\begin{itemize}
\item$0\le k_i \le\ell_i \le n_i^{2,\mf a}$,
\item if $k_j=0$, then $\ell_{j-1}= n_{j-1}^{2,\mf a}$.
\end{itemize}
For $(\mb k,\bs\ell) \in\bb I_{2,\mf a}$, $\mf a \in\{ \mf s, \mf
l\}$, let $R^{2,\mf l}(\mb k,\bs\ell)$, $R^{2,\mf s}(\mb k,\bs\ell)$
be the sets
\begin{eqnarray*}
R^{2,\mf l} (\mb k,\bs\ell) & =& R^{2,\mf l} \cup \bigl\{(a,0) \dvtx
k_0\le a\le\ell_0\bigr\} \cup \bigl\{(n+1,b) \dvtx
k_1\le b\le\ell_1\bigr\}
\\
&&{} \cup \bigl\{(n+1-a,n-2) \dvtx k_2\le a\le\ell_2\bigr\}
\cup \bigl\{(0,n-2-b) \dvtx k_3\le b\le\ell_3\bigr\},
\\
R^{2,\mf s} (\mb k,\bs\ell) & =& R^{2,\mf s} \cup \bigl\{(a,0) \dvtx
k_0\le a\le\ell_0\bigr\} \cup \bigl\{(n-2,b) \dvtx
k_1\le b\le\ell_1\bigr\}
\\
& &{}\cup \bigl\{(n-2-a,n+1) \dvtx k_2\le a\le\ell_2\bigr\}
\cup \bigl\{(0,n+1-b) \dvtx k_3\le b\le\ell_3\bigr\}.
\end{eqnarray*}

Denote by $I_{2,\mf a}$, $\mf a \in\{ \mf s, \mf l\}$, the set of
pairs $(\mb k, \bs\ell)\in\bb I_{2,\mf a}$ such that $|R^{2,\mf a}
(\mb k,\bs\ell)|=n^2$. For $(\mb k,\bs\ell) \in I_{2,\mf a}$, denote
by $M^{2,\mf a}_i(\mb k,\bs\ell)$ the number of particles attached to
the side $i$ of the rectangle $R^{2,\mf a} (\mb k,\bs\ell)$:
\[
M^{2,\mf a}_i(\mb k,\bs\ell) = \cases{ \ell_i-k_i+1,
& \quad $\mbox{if $k_{i+1}\ge1,$}$\vspace*{2pt}
\cr
\ell_i-k_i+2,
&\quad $\mbox{if $k_{i+1}=0$}$.}
\]
Clearly, for $(\mb k, \bs\ell)\in I_{2,\mf a}$, $\sum_{0\le i\le3}
M^{2,\mf a}_i(\mb k,\bs\ell) = 3n + A$, where $A$ is the number of
occupied corners, which are counted twice since they are attached to
two sides.

Denote by $I^*_{2,\mf a} \subset I_{2,\mf a}$, the set of pairs $(\mb
k, \bs\ell)\in I_{2,\mf a}$ whose rectangles $R^{2,\mf a} (\mb k,\bs
\ell)$ have at least two particles on each side: $M^{2,\mf a}_i(\mb
k,\bs\ell)\ge2$, $0\le i\le3$. Note that if $(\mb k, \bs\ell)$
belongs to $I^*_{2,\mf a}$, for all $\mb x\in R^{2,\mf a} (\mb k,\bs
\ell)$, there exist $\mb y$, $\mb z \in R^{2,\mf a} (\mb k,\bs\ell)$,
$\mb y\neq \mb z$, with the property $|\mb x-\mb y|=|\mb x-\mb
z|=1$.

For $(\mb k,\bs\ell) \in I_{2,\mf a}$, $\mf a \in\{ \mf s, \mf l\}$,
$\mb x\in\bb T_L$, let $R^{2,\mf a}_{\mb x} (\mb k,\bs\ell) = \mb x
+ R^{2,\mf a} (\mb k,\bs\ell)$, and let $\zeta_{\mb x}^{\mf a, (\mb
k,\bs\ell)}$ represent the configurations defined by
\[
\mbox{$\zeta_{\mb x}^{\mf a, (\mb k,\bs\ell)} (\mb y) = 1$\qquad if and only if $\mb y
\in R_{\mb x}^{2,\mf a}(\mb k,\bs\ell)$}.
\]
Figure
\ref{fig6} presents examples of configurations $\zeta_{\mb x}^{\mf l,
(\mb k,\bs \ell)}$.
The configurations $\zeta_{\mb x}^{\mf a, (\mb k,\bs\ell)}$, $(\mb
k,\bs\ell) \in I_{2,\mf a}$, have at least four particles attached to
the longer side, and the configurations $\zeta_{\mb x}^{\mf a, (\mb
k,\bs\ell)}$, $(\mb k,\bs\ell) \in I_{2,\mf a} \setminus
I^*_{2,\mf a}$, belong to $\Omega^3$, forming a $(n-1)\times(n+1)$
rectangle of particles with one extra particle attached to a side of
length $n-1$. Let $\Omega^4 = \Omega^4_{L,K}$, be the set of
configurations associated to the pairs $(\mb k,\bs\ell)$ in
$I^*_{2,\mf a}$:
\[
\Omega^4_{\mb x} = \bigl\{ \zeta_{\mb x}^{\mf a, (\mb k,\bs\ell)}
\dvtx \mf a \in\{ \mf s, \mf l\}, (\mb k,\bs\ell) \in I^*_{2,\mf a} \bigr\},
\qquad\Omega^4 = \bigcup_{\mb x\in\bb T_L}
\Omega^4_{\mb x}.
\]

\begin{figure}

\includegraphics{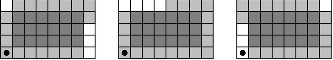}

\caption{Examples of configurations in $\Omega^4_{\mb x}$ for
$n=6$. In general, $3n$ particles (or $n-2$ holes) have to be placed
around the rectangle, respecting the constraints introduced above.
The black dot represents the site $\mb x$.}\label{fig6}
\end{figure}

Let $\Xi= \Xi(L,n)$ be the set of configurations
\[
\Xi = \Gamma \cup \widehat\Omega^1 \cup \Omega^2 \cup
\widehat\Omega^3 \cup \Omega^4.
\]
The set $\Xi$ has less than $C_0 L^2 n^7$ configurations, the main
contributions coming from $\Omega^2$ and $\Omega^4$:
%
\begin{equation}
\label{q1} |\Xi| \le C_0 L^2 n^7.
\end{equation}

\section{The trace of \texorpdfstring{$\eta(t)$}{$eta(t)$} on \texorpdfstring{$\Xi$}{$Xi$}}
\label{sec2}

We define in this section the trace of $\eta(t)$ on $\Xi$, denoted
by $\zeta(t)$, and we introduce a process $\widehat{\zeta}(t)$ which
approximates well the trace process. The parameters $n$ and $L$ are
fixed and $\beta\uparrow\infty$.

Denote by $\zeta(t)$ the trace of the process $\eta(t)$ on $\Xi$. By
Propositions 6.1 and 6.3 in \cite{bl2}, $\zeta(t)$ is a irreducible,
reversible process on $\Xi$, whose stationary measure, denoted by
$\mu_\zeta$, is the measure $\mu_K$ conditioned to $\Xi$:
$\mu_\zeta(\xi) = \mu_K(\xi)/\mu_K(\Xi) =
\mu_\beta(\xi)/\mu_\beta(\Xi)$, $\xi\in\Xi$.

For a subset $\Pi$ of $\Xi$, denote by $H_\Pi$ (resp., $H^+_\Pi$) the hitting
time of $\Pi$ (resp., the return time to $\Pi$):
%
\begin{eqnarray}
\label{b04} H_\Pi&:= &\inf \bigl\{ t > 0 \dvtx \zeta(t) \in\Pi
\bigr\},
\nonumber
\\[-8pt]
\\[-8pt]
\nonumber
 H^+_\Pi &=& \inf\bigl\{ t>0 \dvtx \zeta(t) \in\Pi, \zeta(s) \neq
\zeta(0) \textrm{ for some $0< s < t$}\bigr\}.
\nonumber
\end{eqnarray}

Denote by $R_\zeta(\eta,\xi)$ the jump rates of the trace process
$\zeta(t)$, and by $L_\zeta$ its generator:
\[
(L_\zeta f) (\eta) = \sum_{\xi\in\Xi}
R_\zeta(\eta,\xi) \bigl[f(\xi) - f(\eta)\bigr].
\]

Let $\mb P^\zeta_\xi$, $\xi\in\Xi$, be the probability measure on
the path space $D(\bb R_+, \Xi)$ induced by the Markov process
$\zeta(t)$ starting from $\xi$. Expectation with respect to $\mb
P^\zeta_\xi$ is represented by $\mb E^\zeta_\xi$. For two disjoint
subsets $A$, $B$ of $\Xi$, denote by $\Cap_\zeta(A,B)$ the capacity
between $A$ and $B$ for the $\zeta(t)$ process
\[
\Cap_\zeta(A,B) = \sum_{\eta\in A}
\mu_\zeta(\eta) \lambda_\zeta(\eta) \mb P^\zeta_\eta
\bigl[ H_B < H^+_A \bigr],
\]
where $\lambda_\zeta(\eta)$ stands for the holding rates of the
process $\zeta(t)$, $\lambda_\zeta(\eta) =\break  \sum_{\xi\in\Xi} R_\zeta
(\eta,\xi)$.\vspace*{1pt}

We will introduce below several Markov processes obtained from the
trace process $\zeta(t)$, either by reflecting it on a subset of $\Xi$
or by modifying its jump rates. All processes will be denoted by some
Greek letter with, sometimes, an accent. If the process is represented
by $\chi$, the jump rates, the invariant measure, the holding rates,
and the capacities are denoted by $R_\chi$, $\mu_\chi$,
$\lambda_\chi$ and $\Cap_\chi$, respectively.

\subsection*{The rates $\mathbf{R}(\eta,\xi)$}
The trace process $\zeta(t)$ may jump from any configuration of $\Xi$
to any other configuration. Most of these jumps, however, have very
small rates when $\beta$ is large. The main results of the next section
assert that in the zero-temperature limit the process $\zeta(t)$ is
very well approximated by a simple Markov process denoted by
$\widehat{\zeta}(t)$. We define in this subsection the jump rates $\mb
R (\cdot, \cdot)$ of $\widehat{\zeta}(t)$ and derive in the next
subsection some properties of $\mb R$.

Let $\{\mb x_t \dvtx t\ge0\}$ be the nearest-neighbor, symmetric random
walk on $\bb T_L$ which jumps from a site $\mb x$ to $\mb x \pm e_i$
at rate $1$. Denote by $\bb P^{\mb x}_{\mb y}$, $\mb y\in\bb T_L$,
the probability measure on $D(\bb R_+, \bb T_L)$ induced by $\mb
x_t$ starting from $\mb y$. Denote by $\mf p(\mb x, \mb y, G)$, $\mb
x\in\bb T_L$, $\mb y\in G$, $G\subset\bb T_L$, the probability
that the random walk starting from $\mb x$ reaches $G$ at $\mb y$:
%
\begin{equation}
\label{26} \mf p(\mb y, \mb z, G):= \bb P^{\mb x}_{\mb y} [\mb
x_{H_G} = \mb z ].
\end{equation}
By extension, for a subset $\mb A$ of $\bb T_L$, let $\mf p(\mb x,
\mb A, G) = \sum_{\mb y\in\mb A} \mf p(\mb x, \mb y, G)$. Moreover,
when $G=\partial_+Q$, we omit the set $G$ in the notation
\[
\mf p(\mb x, \mb A):= \mf p(\mb x, \mb A, \partial_+ Q).
\]

Let $\mb y_t = (\mb y^1_t,\mb y^2_t)$ be the continuous-time Markov
chain on $D_n = \{(j,k) \dvtx 0\le j < k\le n-1\} \cup\{(0,0)\}$ which
jumps from a site $\mb y$ to any of its nearest neighbor sites $\mb
z$, $|\mb y -\mb z|=1$, at rate $1$. Let $D^o_n = \{(0,0)\} \cup
\{(j,n-1) \dvtx 0\le j<n-1\}$ and let
%
\begin{equation}
\label{29} \mf q_n = \bb P^{\mb y}_{(0,1)} [
H_{D^o_n} < H_{(0,0)} ],
\end{equation}
where $\bb P^{\mb y}_{(0,1)}$ stands for the distribution of $\mb y_t$
starting from $(0,1)$.

Let $\mb z_t = (\mb z^1_t,\mb z^2_t)$ be the continuous-time Markov
chain on $E_n = \{0,\ldots, n-1\}^2 \cup\{\mf d\}$ which jumps from a
site $\mb z$ to any of its nearest neighbor sites $\mb z'$, $|\mb z'
-\mb z|=1$, at rate $1$, and which jumps from $(1,1)$ (resp., from $\mf
d$) to $\mf d$ [resp., to $(1,1)$] at rate $1$. Let $E^+_n = \{(j,n-1)
\dvtx 0\le j\le n-1\}$, $E^-_n = \{(j,0) \dvtx 1\le j\le n-1\}$, $\partial E_n
= E^+_n \cup E^-_n \cup\{(0,0)\}$ and for $0\le k\le n-1$, let
%
\begin{eqnarray}
\label{30}  \mf r^+_n &=& \bb P^{\mb z}_{(0,1)}
[ H_{\partial E_n} = H_{E^+_n} ],\qquad \mf r^-_n = \bb
P^{\mb z}_{(0,1)} [ H_{\partial E_n} = H_{E^-_n} ],
\nonumber
\\[-8pt]
\\[-8pt]
\nonumber
\mf r^0_n (k) &=& \bb P^{\mb z}_{(k,1)}
[ H_{\partial E_n} = H_{(0,0)} ],\qquad \mf r_n = \mf
r^+_n + \mf r^-_n,
\end{eqnarray}
where $\bb P^{\mb z}_{(k,1)}$ stands for the distribution of $\mb z_t$
starting from $(k,1)$.

We are finally in a position to define $\mb R(\eta^{\mb x},\xi)$.
Let
%
\begin{equation}\qquad
\label{50} \mb R\bigl(\eta^{\mb0}, \eta^{0,0}_{\mb0}
\bigr) = \frac{1 + \mf A_{0,3} + \mf r^-_n + \mf q_n}{
4 + \mf q_n + \mf r_n - \mf A},\qquad \mb R \bigl(\eta^{\mb0},
\eta^{0,1}_{\mb0}\bigr) = \frac{\mf A_{1,2} + \mf r^+_n}{
4 + \mf q_n + \mf r_n - \mf A},
\end{equation}
where
\begin{eqnarray*}
 \mf A &=& \mf p\bigl(\mb w_0 - 2e_2, \{\mb
w_0 - e_2, \mb w_0 - e_1\}
\bigr) + \mf p\bigl(\mb w_0 - e_1 - e_2, \{
\mb w_0 - e_2, \mb w_0 - e_1\}
\bigr),
\\
 \mf A_{0,3} &=& \mf p\bigl(\mb w_0 - 2e_2,
Q^{0,0}_{\mb0} \cup Q^{0,3}_{\mb0}\bigr) + \mf
p\bigl(\mb w_0 - e_1 - e_2,
Q^{0,0}_{\mb0} \cup Q^{0,3}_{\mb0}\bigr),
\\
 \mf A_{1,2} &=& \mf p\bigl(\mb w_0 - 2e_2,
Q^{0,1}_{\mb0} \cup Q^{0,2}_{\mb0}\bigr) + \mf
p\bigl(\mb w_0 - e_1 - e_2,
Q^{0,1}_{\mb0} \cup Q^{0,2}_{\mb0}\bigr).
\end{eqnarray*}

For $\mb x\in\bb T_L$, let
%
\begin{equation}
\label{45} \mc V\bigl(\eta^{\mb x}\bigr) = \bigl\{
\eta^{i,j}_{\mb x} \dvtx 0\le i,j\le3\bigr\},\qquad \mc G_{\mb x}
= \bigl\{\eta^{\mb x}\bigr\}\cup\mc V\bigl(\eta^{\mb x}\bigr).
\end{equation}
We extend $\mb R(\eta^{\mb0}, \cdot)$ to $\mc V(\eta^{\mb0})$ by
symmetry, and we set
\[
\mb R\bigl(\eta^{\mb x}, \eta^{i,j}_{\mb x}\bigr) = \mb R
\bigl(\eta^{\mb0}, \eta^{i,j}_{\mb0}\bigr),\qquad \mb R\bigl(
\eta^{\mb x}, \xi\bigr) = 0,\qquad 0\le i,j\le3, \xi\notin\mc V\bigl(
\eta^{\mb x}\bigr).
\]

\subsection*{The rates $\mathbf{R}(\eta,\xi)$ on $\Xi\setminus\Gamma$}

Denote by $\mc N(\mc E^{i,j}_{\mb x})$, $\mc N$ for neighborhood,
$0\le i,j\le3$, $\mb x\in\bb T_L$, the configurations in $\bb H_2$
which can be reached from a configuration in $\mc E^{i,j}_{\mb x}$ by
a rate $e^{-\beta}$ jump. The set $\mc N(\mc E^{2,2}_{\mb w})$, for
instance, has the following $3n$ elements. There are $n+1$
configurations obtained when the top particle detaches itself from the
others: $\sigma^{\mb w_{2}, \mb z} \eta^{\mb w}$, where $\mb
z=(-1,n)$, $(a,n+1)$, $0\le a\le n-2$, $(n-1,n)$. There are $n-1$
configurations obtained when the particle at $\mb w_2-e_2$ moves
upward: $\sigma^{\mb w_2 - e_2, \mb z} \eta^{\mb w}$, $\mb z = (a,n)$,
$0\le a\le n-2$. There are $n-2$ configurations obtained when the
particle at $\mb w_2-e_1$ moves to the right: $\sigma^{\mb w_2 - e_1,
\mb z} \eta^{\mb w}$, $\mb z = (a,n)$, $0\le a\le n-3$. To complete
the description of the set $\mc N(\mc E^{2,2}_{\mb w})$, we have to
add the configurations $\sigma^{\mb w_3, \mb w_3+e_2}\sigma^{\mb w_2,
\mb w_3 +e_1+e_2} \eta^{\mb w}$ and $\sigma^{\mb w_2 - e_1, \mb w_2
-e_1 +e_2} \sigma^{\mb w_2, \mb w_2-2e_1+e_2} \eta^{\mb w}$.

Recall that $L$ and $K=n^2$ are fixed in this section. We proved in
\cite{bl5} that for each $\zeta\in\mc N(\mc E^{i,j}_{\mb x})$, there
exists a probability measure $\bb M(\zeta, \cdot)$ concentrated on
$\Gamma\cup\Omega^1\cup\Omega^2$ such that
\[
\lim_{\beta\to\infty} \mb P_{\zeta}[ H_{\bb H_{01}} =
H_\xi] = \bb M (\zeta, \xi).
\]
Define the rates $\mb R(\eta^{i,j}_{\mb x}, \cdot)$ by
\[
\mb R\bigl(\eta^{i,j}_{\mb x}, \xi\bigr):=
\cases{ \displaystyle\sum
_{\zeta\in\mc N(\mc E^{i,j}_{\mb x})} \bb M \bigl(\zeta, \mc E^{k,l}_{\mb y}
\bigr), &\quad $\mbox{if } \xi\in\widehat\Omega^1, \xi=
\eta^{k,l}_{\mb y},$ \vspace*{2pt}
\cr
\displaystyle\sum
_{\zeta\in\mc N(\mc E^{i,j}_{\mb x})} \bb M (\zeta, \xi), &\quad $\mbox{if } \xi\in\Gamma\cup
\Omega^2 $, \vspace*{2pt}
\cr
0, & \quad $\mbox{if } \xi\in\widehat
\Omega^3 \cup\Omega^4 $.}
\]
Note that the rates $\mb R(\eta^{i,j}_{\mb x}, \cdot)$ depend on $n$
and $L$, but not on $\beta$. In the next sections, however, $n$ and
$L$ will be taken as functions of $\beta$ and $\mb R(\eta^{i,j}_{\mb
x}, \cdot)$ will depend on $\beta$ through $n$, $L$. This remark
holds for all rates $\mb R(\eta, \xi)$ defined below in this section.

For each $\eta\in\Omega^2 \cup\Omega^4$, denote by $\mc N(\eta)$ the
set of configurations in $\bb H_2$ which can be reached from $\eta$ by
a rate $e^{-\beta}$ jump. Each set $\mc N(\eta)$ has at most $8$
configurations. Fix $\eta\in\Omega^2$ (resp., $\eta\in\Omega^4$).
We proved in \cite{bl5} that for each $\zeta\in\mc N(\eta)$, there exists
a probability measure, denoted by $\bb M (\zeta, \cdot)$, concentrated on
$\Omega^1\cup\Omega^2\cup\Omega^3$ (resp., $\Omega^3\cup\Omega^4$) such
that
\[
\lim_{\beta\to\infty} \mb P_{\zeta}[ H_{\bb H_{01}} =
H_\xi] = \bb M (\zeta, \xi).
\]
For $\eta\in\Omega^2$, let
\[
\mb R(\eta, \xi) = \cases{ \displaystyle\sum_{\zeta\in\mc N (\eta)} \bb M (\zeta,
\xi), &\quad $\mbox{if }\xi\in\Omega^2 $, \vspace*{2pt}
\cr
\displaystyle\sum
_{\zeta\in\mc N (\eta)} \bb M \bigl(\zeta, \mc E^{k,l}_{\mb y}
\bigr), &\quad $\mbox{if }\xi\in\widehat\Omega ^1, \xi=
\eta^{k,l}_{\mb y} $, \vspace*{2pt}
\cr
\displaystyle\sum
_{\zeta\in\mc N (\eta)} \bb M \bigl(\zeta, \mc E^{\mf a,k}_{\mb y}
\bigr), &\quad $ \mbox{if }\xi\in\widehat \Omega^3, \xi=
\zeta^{\mf a,k}_{\mb y} $, \vspace*{2pt}
\cr
0, &\quad  $\mbox{if } \xi\in
\Gamma\cup\Omega^4 $,}
\]
and for $\eta\in\Omega^4$, let
\[
\mb R(\eta, \xi) = \cases{ \displaystyle\sum_{\zeta\in\mc N (\eta)} \bb M (\zeta,
\xi), &\quad $\mbox{if }\xi\in\Omega^4 $, \vspace*{2pt}
\cr
\displaystyle\sum
_{\zeta\in\mc N (\eta)} \bb M \bigl(\zeta, \mc E^{\mf a,k}_{\mb y}
\bigr), &\quad $\mbox{if }\xi\in\widehat \Omega^3, \xi=
\zeta^{\mf a,k}_{\mb y} $, \vspace*{2pt}
\cr
0, & \quad$\mbox{if } \xi\in
\Gamma\cup\widehat\Omega^1\cup\Omega^2 $.}
\]

For each $\mb x\in\bb T_L$, $\mf a\in\{\mf l,\mf s\}$, $0\le j\le3$,
denote by $\mc N(\mc E^{\mf a,j}_{\mb x})$ the set of configurations
in $\bb H_2$ which can be reached from a configuration in $\mc E^{\mf
a,j}_{\mb x}$ by a rate $e^{-\beta}$ jump. We proved in \cite{bl5}
that for each $\zeta\in\mc N(\mc E^{\mf a,j}_{\mb x})$, there exists
a probability measure $\bb M (\zeta, \cdot)$
concentrated on $\Omega^2\cup\Omega^3\cup\Omega^4$ such that
\[
\lim_{\beta\to\infty} \mb P_{\zeta}[ H_{\bb H_{01}} =
H_\xi] = \bb M (\zeta, \xi).
\]
Define the rates $\mb R (\zeta^{\mf a,j}_{\mb x}, \xi)$ by
\[
\mb R \bigl(\zeta^{\mf a,j}_{\mb x}, \xi\bigr) =
 \cases{ \displaystyle\sum
_{\zeta\in\mc N(\mc E^{\mf a,j}_{\mb x})} \bb M (\zeta, \xi),
 &\quad $\mbox{if } \xi\in
\Omega^2 \cup\Omega^4 $,\vspace*{2pt}
\cr
\displaystyle\sum
_{\zeta\in\mc N(\mc E^{\mf a,j}_{\mb x})}
\bb M \bigl(\zeta, \mc E^{\mf a,k}_{\mb y}
\bigr), & \quad
$\mbox{if }\xi\in\widehat\Omega^3,\xi = \zeta^{\mathfrak{ a},k}_{\mathbf y},$\vspace*{2pt}\cr
0, &\quad $\mbox{if } \xi\in\Gamma\cup\widehat\Omega^1$.}
\]

In \cite{bl5}, we expressed the probability measures $\bb M$ introduced
above in terms of hitting probabilities of simple finite-state Markov
processes. We recall below some of these explicit expressions.

\subsection*{The process $\widehat{\zeta}(t)$}
Denote by $\widehat{\zeta}(t)$ the continuous-time Markov chain on
$\Xi$ whose jump rates, denoted by $R_{\widehat{\zeta}} (\eta,\xi)$,
are given by
\[
R_{\widehat{\zeta}} (\eta,\xi) = \cases{
e^{-\beta} \mb R (\eta,\xi), &\quad $\eta\in\Xi\setminus\Gamma, \xi\in\Xi,$
\vspace*{2pt}\cr
e^{-2\beta} \mb R (\eta,\xi), &\quad $\eta\in\Gamma, \xi\in\Xi.$}
\]

In contrast to the jump rates of the trace process, the jump rates
$\mb R(\eta,\xi)$ vanish for a great number of pair of configurations.
In the next section, we show that the trace process $\zeta(t)$ is well
approximated by the Markov process $\widehat{\zeta}(t)$. In contrast
to $\zeta(t)$, the process $\widehat{\zeta}(t)$ may not be reversible
with respect to $\mu_\zeta$ since the detailed balance conditions
between configurations in $\Gamma$ and configurations in $\Xi
\setminus\Gamma$ may fail. The reflection of $\widehat{\zeta}(t)$ on
subsets of $\Xi\cap\bb H_1$ is however reversible with respect to
the uniform measure in view of \eqref{eq:2} below.

Denote by $\mb P^{\widehat{\zeta}}_\xi$, $\xi\in\Xi$, the probability
measure on the path space $D(\bb R_+, \Xi)$ induced by the Markov
process $\widehat{\zeta}(t)$ starting from $\xi$. Expectation with
respect to $\mb P^{\widehat{\zeta}}_\xi$ is represented by $\mb
E^{\widehat{\zeta}}_\xi$.

\subsection*{Properties of the rates $\mathbf{R}(\eta,\xi)$}

We claim that for each fixed $n$ and $L$,
%
\begin{equation}
\label{eq:2} \mb R(\eta,\xi) = \mb R(\xi,\eta),\qquad \eta, \xi\in\Xi\setminus
\Gamma.
\end{equation}
These identities can be checked directly from the definition of the
rates $\mb R(\xi,\eta)$ or can be deduced from the results of the
previous section. Indeed, since the trace process $\zeta(t)$ is
reversible, $\mu_K(\eta) R_{\zeta} (\eta,\xi) = \mu_K(\xi)
R_{\zeta}(\xi,\eta)$. For $\eta$, $\xi\in\Xi\setminus\Gamma$,
$\mu_K(\eta) = \mu_K(\xi)$ so that $R_{\zeta} (\eta,\xi) =
R_{\zeta}(\xi,\eta)$. The results of the previous section are in force
if we fix $n$ and $L$ and let $\beta\uparrow\infty$. Therefore, by
Proposition \ref{s08}, $e^{-\beta} \mb R (\eta,\xi) = e^{-\beta} \mb R
(\xi,\eta) + o(e^{-\beta})$, from which the claim follows.

Let $\xi^\star_1 = \sigma^{\mb w_0, \mb w_0-e_2} \eta^{\mb0}$, $\xi
^\star_2 =
\sigma^{\mb w_0, \mb w_0-e_1} \eta^{\mb0}$ and for $j=1$, $2$, let
%
\begin{equation}
\label{47} \bb M_j(\xi) = \lim_{\beta\to\infty} \mb
P_{\xi^\star_j} [ H_{\bb
H_{01}} = H_{\xi} ].
\end{equation}
We have computed $\bb M_j$ in \cite{bl5}, with the difference that the
sets $\mc E^{i,j}_{\mb x}$ in \cite{bl5} have been replaced here by the
configurations $\eta^{i,j}_{\mb x}$ so that $\bb M_j(\mc E^{i,j}_{\mb
x})$ in \cite{bl5} corresponds to $\bb M_j(\eta^{i,j}_{\mb x})$
here. We have that
\begin{eqnarray*}
\bb M_1\bigl(\eta^{0,0}_{\mb0}\bigr) &=&
\frac{1}2 \biggl\{ \frac{1 + \mf r^-_n + \mf A_{0,3}}{
4 + \mf q_n + \mf r_n - \mf A} + \frac{1 + \mf r^-_n + \mf A_0 - \mf A_3}{
4 + \mf q_n + \mf r_n + \mf A(e_1) - \mf A(e_2)} \biggr\},
\\
 \bb M_2\bigl(\eta^{0,0}_{\mb0}\bigr) &=&
\frac{1}2 \biggl\{ \frac{1 + \mf r^-_n + \mf A_{0,3}}{
4 + \mf q_n + \mf r_n - \mf A} - \frac{1 + \mf r^-_n + \mf A_0 - \mf A_3}{
4 + \mf q_n + \mf r_n + \mf A(e_1) - \mf A(e_2)} \biggr\},
\end{eqnarray*}
where
\begin{eqnarray*}
 \mf A_j &=& \mf p\bigl(\mb w_0 - 2e_2,
Q^{0,j}_{\mb0}\bigr) + \mf p\bigl(\mb w_0 -
e_1 - e_2, Q^{0,j}_{\mb0}\bigr),
\\
 \mf A(e_i) &=& \mf p(\mb w_0 - 2e_2, \mb
w_0-e_i) + \mf p(\mb w_0 - e_1
- e_2, \mb w_0-e_i ).
\end{eqnarray*}

\begin{lemma}
\label{s19}
For all $0\le i,j\le3$, $\mb x\in\bb T_L$,
%
\begin{eqnarray}
\label{42} \mb R\bigl(\eta^{i,j}_{\mb x},
\eta^{i,j\pm1}_{\mb x}\bigr) &\ge& (1/4)^3,\qquad \mb R\bigl(
\eta^{i,i-1}_{\mb x},\eta^{\mb x}\bigr) = \mb R\bigl(
\eta^{i,i}_{\mb x},\eta^{\mb x}\bigr) \ge 1/4,
\nonumber
\\[-8pt]
\\[-8pt]
\nonumber
 \mb R\bigl(\eta^{i,i}_{\mb x}, \eta^{i-1,i}_{\mb x}
\bigr) &= &\mb R\bigl(\eta^{i-1,i}_{\mb x}, \eta^{i,i}_{\mb x}
\bigr) \ge 1.
\end{eqnarray}
Moreover,
%
\begin{equation}
\label{46} \mb R\bigl(\eta^{i,i}_{\mb x},
\Omega^2\bigr) = \frac{1}n, \qquad\mb R\bigl(\eta^{i,i+1}_{\mb x},
\Omega^2\bigr) = \frac{1}n + \frac{1}{n-1} \cdot
\end{equation}
\end{lemma}

\begin{pf}
To fix ideas consider the case $i=j=0$, $\mb x=\mb w$. Recall the
explicit form of $\bb M$ given in \cite{bl5}. As we have
seen in Section~\ref{sec3}, the set $\mc N(\mc E^{0,0}_{\mb w})$ has $3n$
elements which can be divided into five classes. Consider first the set
$\mc A_1= \{\sigma^{\mb w+e_2, \mb z} \eta^{\mb w} \dvtx \mb z\in
Q^{0,0}\}$ which has $n-1$ configurations. By \cite{bl5}, Lemma 4.2,
%
\begin{equation}
\label{35} \sum_{\zeta\in\mc A_1} \bb M (\zeta,\xi) = \mb1
\bigl\{\xi= \eta^{3,0}_{\mb w}\bigr\} + \sum
_{k=1}^{n-1} \mf r^0_n(k)
\bb M_1(\xi),
\end{equation}
where $\mf r^0_n(k)$ has been defined in \eqref{30} and $\bb M_1(\xi)$ in
\eqref{47}.

The second class consists of the set of configurations $\mc A_2=
\{\sigma^{\mb w+e_1, \mb z} \eta^{\mb w} \dvtx \mb z=(k,-1), 2\le k\le
n-1\}$ which has $n-2$ elements. By \cite{bl5}, Lemma 4.2,
\[
\bb M \bigl(\zeta,\eta^{0,0}_{\mb w}\bigr) = 1,\qquad \zeta\in\mc
A_2.
\]
In particular, this subset of $\mc N(\mc E^{0,0}_{\mb w})$ does not
contribute to the rate $\mb R(\eta^{0,0}_{\mb x}, \cdot)$.

Two special configurations form each class:
%
\begin{eqnarray}
\label{43} & &\bb M \bigl(\sigma^{\mb w, \mb w_1-e_1-e_2} \sigma^{\mb w_1,
\mb w_1-e_2}
\eta^{\mb0}, \xi\bigr)
\nonumber
\\[-8pt]
\\[-8pt]
\nonumber
&&\qquad = \frac{1}n \mb1\bigl\{\xi= \sigma^{\mb w, \mb w_1-e_1-e_2}
\sigma^{\mb w_2, \mb w_1-e_2}\eta^{\mb0}\bigr\} + \biggl(1-\frac{1}n
\biggr) \mb1\bigl\{\xi= \eta^{0,0}_{\mb w} \bigr\}
\end{eqnarray}
and
\[
\bb M \bigl(\sigma^{\mb w, \mb w+e_1-e_2} \sigma^{\mb w + e_1,
\mb w +2e_1-e_2}\eta^{\mb0},
\eta^{0,0}_{\mb w}\bigr) = 1.
\]
This latter configuration also does not contribute to the rate $\mb
R(\eta^{0,0}_{\mb0}, \cdot)$.

The last set is $\mc A_5 = \{\sigma^{\mb w, \mb z} \eta^{\mb0} \dvtx \mb
z\in\Gamma_{0,0}\}$, where $\Gamma_{0,0}$ stands for the set of sites
$\mb z\in\bb T_L$ which do not belong to $Q$ and are at distance $1$
from $Q^{0,0}$. This set has $n+1$ elements and
%
\begin{eqnarray}
\label{44} \sum_{\zeta\in\mc A_5} \bb M (\zeta,\xi) &=& \sum
_{\mb z\in\Gamma_{0,0}} \sum_{j=0}^3
\mf p\bigl(z,Q^{0,j}\bigr) \mb1\bigl\{\xi= \eta^{0,j}_{\mb0}
\bigr\}
\nonumber
\\[-8pt]
\\[-8pt]
\nonumber
&&{}+ \sum_{\mb z\in\Gamma_{0,0}} \bigl\{ \mf p(z,\mb
w-e_2) \bb M_1(\xi) + \mf p(z,\mb w-e_1) \bb
M_2(\xi) \bigr\},
\end{eqnarray}
where $\bb M_1(\xi)$ and $\bb M_2(\xi)$ have been defined in \eqref{47}.

Now that we have explicit formulas for the rates $\mb
R(\eta^{0,0}_{\mb0}, \cdot)$, we may prove the bounds claim in the
lemma. By \eqref{35} $\mb R(\eta^{0,0}_{\mb0},\eta^{3,0}_{\mb0})
\ge1$, and by \eqref{44},
\[
\mb R\bigl(\eta^{0,0}_{\mb0},\eta^{0,1}_{\mb0}
\bigr) \ge \mb p(\mb w_1+e_1-e_2, \mb
w_1+e_1) = \tfrac{1}4,
\]
where $\mb p$ represents the jump probabilities of a nearest-neighbor,
symmetric, discrete-time random walk on $\bb T_L$. By similar reasons,
$\mb R(\eta^{0,0}_{\mb0}, \eta^{\mb0}) \ge1/4$ and $\mb
R(\eta^{0,0}_{\mb0}, \eta^{0,3}_{\mb0}) \ge(1/4)^3$. By symmetry,
$\mb R(\eta^{0,3}_{\mb0}, \eta^{\mb0}) = \mb R(\eta^{0,0}_{\mb0},
\eta^{\mb0})$. This proves \eqref{42} in view of \eqref{eq:2}.

To prove \eqref{46}, note that $\bb M (\zeta, \Omega^2)$ vanishes for
all configurations but for $\zeta= \sigma^{\mb w, \mb w_1-e_1-e_2}
\sigma^{\mb w_1, \mb w_1-e_2}\eta^{\mb0}$ in which case it is equal
to $1/n$, as observed in \eqref{43}. A~similar assertion holds for
$\bb M (\zeta, \Omega^2)$, $\zeta\in\mc N(\mc E^{0,1}_{\mb w})$, but
in this case there are two configurations which contribute to $\mb
R(\eta^{0,1}_{\mb0}, \Omega^2)$. We leave the details to the reader.
\end{pf}

\begin{lemma}
\label{s17}
For $n\ge9$, we have that
\[
\max_{\eta\in\mc V(\eta^{\mb x})} \mb P^{\widehat{\zeta}}_{\eta} [
H_{\mc G_{\mb x}^c} < H_{\eta^{\mb x}}] \le \frac{23}n \cdot
\]
\end{lemma}

\begin{pf}
Fix $\mb x=\mb0$ and let $\chi(t)$ be the Markov process $\chi(t) =
\widehat{\zeta}(t\wedge H_{\mc B})$, where $\mc B = \mc V(\eta^{\mb
0})^c$. We identify all configurations in $\mc G_{\mb0}^c$, turning
$\chi(t)$ a Markov process on $\mc G_{\mf d} = \mc G_{\mb0} \cup
\{\mf d\}$, where $\mf d$ replaces $\mc G_{\mb0}^c$.

There are two types of configurations in $\mc V(\eta^{\mb0})$,
configurations in which the attached particle is on the same side of
the empty corner and configurations in which this is not the case. We
denote by $\mc C_1 = \{\eta^{i,i}_{\mb0}, \eta^{i+1,i}_{\mb0} \dvtx 0\le i\le3\}$ the first set, and by $\mc C_2 = \{\eta^{i+2,i}_{\mb0}, \eta^{i+3,i}_{\mb0} \dvtx 0\le i\le3\}$ the second one. Define
$F\dvtx \mc G_{\mf d} \to\{0,1,2,3\}$ by
\[
F\bigl(\eta^{\mb0}\bigr) =0,\qquad F(\mf d)=3,\qquad F(\eta) = i,\qquad \eta\in\mc
C_i, i=1,2.
\]
By symmetry $\widehat\chi(t) = F(\chi(t))$ is a four state Markov
process and by construction
\[
\mb P^{\widehat{\zeta}}_{\eta} [ H_{\mc G_{\mb0}^c} < H_{\eta^{\mb
0}}] =
\mb P^{\widehat\chi}_{i} [ H_{3} < H_{0}]
\]
if $\eta\in\mc C_i$, where $\mb P^{\widehat\chi}_{i}$ represents the
distribution of the process $\widehat\chi(t)$ starting from $i$.

Denote by $R_{\widehat\chi}$ the jump rates of the process $\widehat
\chi(t)$. Since $\mb R(\eta,\xi) = \mb R(\xi,\eta)$ for $\eta\in\mc
C_1$, $\xi\in\mc C_2$, and since $|\mc C_1| = |\mc C_2|$, $R_{\widehat
\chi} (1,2) = R_{\widehat\chi} (2,1)$. By Lemma \ref{s19},
$R_{\widehat\chi} (1,0) = \mb R(\eta^{0,0}_{\mb0}, \eta^{\mb0}) \ge
1/4$, $R_{\widehat\chi} (1,2) \ge\mb R(\eta^{0,0}_{\mb0},
\eta^{3,0}_{\mb0}) \ge1$, and $R_{\widehat\chi} (i,3) \le n^{-1} +
(n-1)^{-1}$ for $i=1$, $2$.

Consider the Markov process $Y_t$ on $\{0,1,2,3\}$ with jump rates $r$
given by $a:=r(1,2) = r(2,1) = R_{\widehat\chi} (1,2)\ge1$, $r(1,0)
= 1/4$ and $r(1,3) = r(2,3) = 2/(n-1)$. A coupling shows that for $i=1$,
$2$,
\[
\mb P^{\widehat\chi}_{i} [ H_{3} < H_{0}] \le
\mb P^{Y}_{i} [ H_{3} < H_{0}] \le
\frac{4\varepsilon
[2a + (1/4) + \varepsilon]}{a + 4\varepsilon
[2a + (1/4) + \varepsilon] },
\]
where $\varepsilon= 2/(n-1)$ and $\mb P^{Y}_{i}$ represents the
distribution of the process $\widehat\chi(t)$ starting from $i$. At last
inequality is an identity if $i=2$. The right-hand side of the
previous displayed equation is uniformly bounded in $a\ge1$ by
$(9+4\varepsilon)\varepsilon\le10 \varepsilon$ provided $\varepsilon\le1/4$.
This proves the lemma.
\end{pf}

The next result states that the chain $\widehat{\zeta}(t)$ jumps with
a probability bounded below by a positive constant independent of
$\beta$ from a configuration $\zeta^{\mf l, 3}_{\mb x}$ to the
configuration $\zeta^{\mf l, 2}_{\mb x}$.

\begin{lemma}
\label{ls02}
Denote by $\mc V(\zeta^{\mf l, 3}_{\mb x})$ the set of configurations
$\xi$ in $\Xi$ such that $\mb R(\zeta^{\mf l, 3}_{\mb x},\xi) >0$,
$\mc V(\zeta^{\mf l, 3}_{\mb x})=\{\xi\in\Xi\dvtx  \mb R(\zeta^{\mf l,
3}_{\mb x},\xi) >0\}$. There exists $c_0>0$ such that
\[
\mb P^{\widehat{\zeta}}_{\zeta^{\mf l, 3}_{\mb x}} [ H_{\zeta^{\mf l, 2}_{\mb x}} = H_{\mc V(\zeta^{\mf l, 3}_{\mb x})}]
\ge c_0.
\]
\end{lemma}

\begin{pf}
By definition of the chain $\widehat{\zeta}(t)$,
\[
\mb P^{\widehat{\zeta}}_{\zeta^{\mf l, 3}_{\mb x}} [ H_{\zeta^{\mf l, 2}_{\mb x}} = H_{\mc V(\zeta^{\mf l, 3}_{\mb x})}] =
\frac{\mb R(\zeta^{\mf l, 3}_{\mb x}, \zeta^{\mf l,
2}_{\mb x})}{
\sum_{\xi\in\Xi} \mb R(\zeta^{\mf l, 3}_{\mb x},\xi)} \cdot
\]
By definition of the rates $\mb R (\eta,\xi)$, $\mb R(\zeta^{\mf l,
3}_{\mb x}, \zeta^{\mf l, 2}_{\mb x}) = \mb R(\zeta^{\mf l, 3}_{\mb
x}, \zeta^{\mf l, 0}_{\mb x}) \ge c_0$, while\break  $\mb R(\zeta^{\mf l,
3}_{\mb x}, \Omega^2) =0$ and $\mb R(\zeta^{\mf l, 3}_{\mb x},
\Omega^4) \le C_0/n$. By \cite{law1}, Proposition 2.4.5, $\mb
R(\zeta^{\mf l, 3}_{\mb x}, \zeta^{\mf l, 1}_{\mb x})$, whose value
involves the sum of $n$ terms, is bounded above by $C_0$. This proves
the lemma.
\end{pf}

The bounds \eqref{51} and \eqref{52} below are needed to estimate the
capacity between the configurations $\eta^{\mb0}$ and $\eta^{e_i}$
for the trace process $\zeta(t)$.

Consider the path $\xi_0 = \eta^{\mb0}, \xi_1, \ldots, \xi_M =
\eta^{e_1}$, $M=3n-2$, from $\eta^{\mb0}$ to $\eta^{e_1}$ in $\Xi$
obtained by sliding particles around three sides of the square $Q$:
$\xi_0 = \eta^{\mb0}, \xi_1 = \eta^{2,1}_{\mb0}, \xi_2 =
\eta^{3,1}_{\mb0}, \xi_3 = \eta^{0,1}_{\mb0}, \xi_4 = \eta^{\mf l,
((1,n-1),(n-2,n-1),(0,n-1),(1,n-2))}_{\mb0}$,
\[
\xi_5 = \eta^{\mf l, ((1,n-1),(n-2,n-1),(0,n),(0,n-3))}_{\mb0}, \ldots,
\xi_{M-4} = \eta^{\mf l, ((1,n-1), (1,n-2),(1,n-1),(0,1))}_{\mb0},
\]
$\xi_{M-3} = \eta^{1,3}_{\mb0}, \xi_{M-2} = \eta^{2,3}_{e_1},
\xi_{M-1} = \eta^{3,3}_{e_1}, \xi_M = \eta^{e_1}$. Figure~\ref{fig8}
presents the first configurations of this sequence in the case $n=6$.

We claim that
%
\begin{eqnarray}
\label{51} \mb R(\xi_0, \xi_1)& \ge&
\tfrac{1}6, \qquad\mb R(\xi_1, \xi_2) \ge 1,
\nonumber
\\[-8pt]
\\[-8pt]
\nonumber
 \mb R(
\xi_2, \xi_3) &\ge& 1,\qquad \mb R(\xi_3,
\xi_4) \ge \frac{1}n \cdot
\end{eqnarray}
The first inequality follows from the definition \eqref{50} of $\mb
R(\eta^{\mb0}, \cdot)$ and from the fact that $\mf q_n\le1$, $\mf
r_n\le1$. The other three estimates have been proven in Lemma
\ref{s19}. Similar estimates hold for the last jumps of the sequence
$(\xi_i)$. On the other hand, in view of the definition of the jump
rates $\mb R$, it is clear that
%
\begin{equation}
\label{52} \mb R(\xi_i, \xi_{i+1}) \ge
\frac{1}n
\end{equation}
for the middle terms. The expression $n^{-1}$ comes from the
probability of a symmetric nearest-neighbor random walk on $\{0, 1,
\ldots, n\}$ starting from $1$ to reach $n$ before $0$.

\begin{figure}

\includegraphics{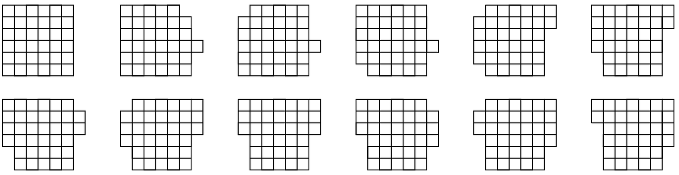}

\caption{The first $12$ configurations of the path
from $\eta^{\mb0}$ to $\eta^{e_1}$ when $n=6$. Observe the subpath
starting from the sixth configuration and ending at the ninth, as well
as the one starting from the ninth configuration and ending at the
twelfth.}
\label{fig8}
\end{figure}

\section{The jump rates of \texorpdfstring{$\zeta(t)$}{$zeta(t)$} and \texorpdfstring{$\widehat{\zeta}(t)$}{$widehat{zeta}(t)$}}
\label{sec6}

The main results of this section, Propositions \ref{s16} and
\ref{s08}, provide estimates on the difference between the jump rates
$R_\zeta(\eta,\xi)$ and $\mb R (\eta,\xi)$. The dependence of the
errors on the parameters $n$ and $L$ are explicit to allow them depend
on $\beta$.

The first lemma presents an explicit formula for the jump rates
$R_\zeta$ of the trace process. Denote by $\Cap_K (A,B)$ the capacity
between two disjoint subsets $A$ and $B$ of $\Omega_{L,K}$ for the
$\eta(t)$ process
\[
\Cap_K (A,B) = \sum_{\eta\in A}
\mu_K(\eta) \lambda(\eta) \mb P^\beta_\eta
\bigl[ H_B < H^+_A \bigr],
\]
where $\lambda(\eta)$ stands for the holding rates of the process
$\eta(t)$ and $R(\eta,\xi) = R_{\beta}(\eta,\xi) $ for the jump rates,
\[
R (\eta,\xi) = \cases{ c_{x,y}(\eta), &\quad $\mbox{if $\xi=
\sigma^{x,y}\eta$ for some $x$, $y\in \bb T_L$,
$|y-x|=1$,}$ \vspace*{2pt}
\cr
0, & \quad$\mbox{otherwise} $,}
\]
$\lambda(\eta) = \sum_{\xi\in\Omega_{L,K}} R (\eta,\xi)$.
Denote by $D_K$ the Dirichlet form associated to the process $\eta(t)$
on the ergodic component $\Omega_{L,K}$: For any function
$f\dvtx \Omega_{L,K}\to\bb R$,
\[
D_K(f) = \frac{1}2 \sum_{\eta, \xi\in\Omega_{L,K}}
\mu_K(\eta) R (\eta,\xi) \bigl[ f(\xi) - f(\eta)
\bigr]^2.
\]

\begin{lemma}
\label{s02}
For $\eta$, $\xi\in\Xi$, $\xi\neq\eta$,
%
\begin{equation}
\label{07} R_\zeta(\eta, \xi) = \frac{\Cap_K(\eta,\Delta_1)}{\mu_K(\eta)} \sum
_{\zeta\in\Delta_1} \mb P_{\eta} [ H_{\Delta_1} =
H_\zeta ] \mb P_{\zeta} [ H_{\Xi} = H_{ \xi}
].
\end{equation}
\end{lemma}

\begin{pf}
The proof of this lemma relies on the fact that the process $\eta(t)$
has to cross the set $\Delta_1$ when going from a configuration
$\eta\in\Xi$ to a configuration $\xi\in\Xi$, $\xi\neq \eta$, $\mb
P_\eta[H_{\Delta_1} < H_\xi]=1$ if $\xi\neq\eta$, $\eta$,
$\xi\in\Xi$. In contrast, there are configurations $\eta\in\Xi$
which can be hit from $\eta$ without crossing the set $\Delta_1$, for
example, the configuration $\eta^{0,0}_{\mb0}$. With a positive
probability which does not vanish in the zero-temperature limit, the
attached particle may jump to the right and then return to its
original position while all the other particles remain still. In this
event, $H^+_\eta<H_{\Delta_1}$.

Fix $\eta$, $\xi\in\Xi$, $\xi\neq\eta$.
By \cite{bl2}, Proposition 6.1,
%
\begin{equation}
\label{03} R_\zeta(\eta, \xi) = \lambda(\eta) \mb
P_{\eta} \bigl[ H^+_{\Xi} = H_{ \xi} \bigr].
\end{equation}
The process $\eta(t)$ has to cross $\Delta_1$ when going from a
configuration $\eta$ in $\Xi$ to a configuration $\xi\neq \eta$ in
$\Xi$. Therefore, on the event $\{H^+_{\Xi} = H_{\xi}\}$ we have
that $H_{\Delta_1} < H^+_{\Xi}$, and by the strong Markov property
\[
\mb P_{\eta} \bigl[ H^+_{\Xi} = H_{ \xi} \bigr] =
\sum_{\zeta\in\Delta_1} \mb P_{\zeta} [ H_{\Xi}
= H_{ \xi} ] \mb P_{\eta} \bigl[ H_\zeta=
H_{\Delta_1} < H^+_{\Xi} \bigr].
\]
We have seen that $H_{\Delta_1} < H_{\Xi\setminus\{\eta\}}$ almost
surely with respect to $\mb P_{\eta}$. Hence, the event $\{H_{\Delta_1}
< H^+_{\Xi}\}$ can be rewritten as $\{H_{\Delta_1}
< H^+_{\eta}\}$ and the previous sum becomes
\[
\sum_{\zeta\in\Delta_1} \mb P_{\zeta} [
H_{\Xi} = H_{ \xi} ] \mb P_{\eta} \bigl[
H_\zeta= H^+_{\Delta_1 \cup\{\eta\}} \bigr].
\]
On the other hand, since
\[
\mb P_{\eta} [ H_\zeta= H_{\Delta_1} ] = \mb
P_{\eta} \bigl[ H^+_\eta< H_{\Delta_1},
H_\zeta= H_{\Delta_1} \bigr] + \mb P_{\eta} \bigl[
H^+_\eta> H_{\Delta_1}, H_\zeta= H_{\Delta_1}
\bigr],
\]
observing that the event on the second term of the right-hand side
is $H_\zeta= H^+_{\Delta_1 \cup\{\eta\}}$, and
applying the strong Markov property to the first term, we obtain that
\[
\mb P_{\eta} \bigl[ H_\zeta= H^+_{\Delta_1 \cup\{\eta\}} \bigr] = \mb
P_{\eta} [ H_\zeta= H_{\Delta_1} ] \mb P_{\eta}
\bigl[ H_{\Delta_1} < H^+_\eta \bigr].
\]
Therefore, in view of \eqref{03} and the previous identities,
\[
R_\zeta(\eta, \xi) = \lambda(\eta) \mb P_{\eta} \bigl[
H_{\Delta_1} < H^+_\eta \bigr] \sum_{\zeta\in\Delta_1}
\mb P_{\zeta} [ H_{\Xi} = H_{ \xi} ] \mb
P_{\eta} [ H_\zeta= H_{\Delta_1} ].
\]
To conclude the proof of the lemma, it remains to recall the
definition of the capacity.
\end{pf}

\subsection*{The jump rates of $\zeta(t)$ on $\Gamma$}
\label{ss6}
We examine in this subsection the jump rates of order $e^{-2\beta}$ of
the trace process $\zeta(t)$. The main result reads as follows.

\begin{proposition}
\label{s16}
There exists a finite constant $C_0$ independent of $\beta$ such that
\[
\max_{\xi\in\mc G_{\mb x}} \bigl| R_{\zeta}\bigl(\eta^{\mb x}, \xi
\bigr) - e^{-2\beta} \mb R \bigl(\eta^{\mb x}, \xi\bigr) \bigr| \le
e^{-2\beta} \kappa_1,
\]
where $\kappa_1$ is a remainder absolutely bounded by $C_0 L
e^{-\beta/2}$. Moreover, for all \mbox{$\mb x\in\bb T_L$},
\[
R_{\zeta}\bigl(\eta^{\mb x}, \Xi\setminus\mc G_{\mb x}
\bigr) \le e^{-2\beta} \kappa_1.
\]
\end{proposition}

It follows from this proposition that
%
\begin{equation}
\label{53} \frac{1}{R_{\zeta}(\eta^{\mb x}, \xi)} \le \frac{e^{2\beta}} {\mb R (\eta^{\mb x}, \xi)} \biggl\{ 1 +
\frac{\kappa_1}{
\mb R (\eta^{\mb x}, \xi) - \kappa_1} \biggr\}.
\end{equation}

The proof of Proposition \ref{s16} relies on the next two
lemmas. Recall the definition of $\bb M_j(\xi)$ given in \eqref{47}.

\begin{lemma}
\label{s20}
Let $\xi^\star_1 = \sigma^{\mb w_0, \mb w_0-e_2} \eta^{\mb0}$, $\xi
^\star_2 =
\sigma^{\mb w_0, \mb w_0-e_1} \eta^{\mb0}$ and let
\[
P_j (\xi) = \mb P_{\xi^\star_j} [ H_{\Xi} =
H_{\xi} ],\qquad j=1, 2.
\]
Then, there exists a finite constant $C_0$ such that for all
$\xi\in\Xi$
\[
\bigl| P_j(\xi) - \bb M_j(\xi) \bigr| \le C_0 L
e^{-\beta/2}.
\]
\end{lemma}

This bound, as well as the next ones, hold for all $\beta$, $L$ and
$K=n^2$. While $\bb M_j(\xi)$ depends only on $L$ and $n$, $P_j(\xi)$
also depends on $\beta$. This estimate will be used to replace a two
step limit in which we first send $\beta\uparrow\infty$, and then $n$,
$L\uparrow\infty$, by a diagonal procedure in which all parameters
diverge simultaneously.

\begin{pf*}{Proof of Lemma \ref{s20}}
We prove this lemma for $\xi= \eta^{0,0}_{\mb0}$ and leave the other
cases to the reader. Set $P_j = P_j(\eta^{0,0}_{\mb0})$, $j=1$, $2$, and
recall the definition of $\mf q_n$, $\mf r^\pm_n$ introduced in
\eqref{29}, \eqref{30}. We claim that
%
\begin{equation}
\label{19} (4 + \mf q_n + \mf r_n - \mf A)
(P_1 + P_2) = 1 + \mf A_{0,3} + \mf
r^-_n + L \varepsilon(\beta/2),
\end{equation}
where $\varepsilon(\beta)$ represents here and below an error term
absolutely bounded by $C_0 e^{-\beta}$.

\begin{figure}

\includegraphics{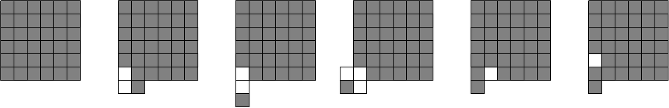}

\caption{The six configurations which can be reached from $\xi^\star_1$ by
a rate $1$ jump. The first two belong to $\bb H_{01}$.}\label{fig88}\vspace*{-5pt}
\end{figure}

To prove \eqref{19}, observe that the configuration $\xi^\star_1$ may jump
at rate $1$ to $6$ configurations, at rate $e^{-\beta}$ to $1$
configuration, at rate $e^{-2\beta}$ to $7$ configurations, and at
rate $e^{-3\beta}$ to $O(n)$ configurations.
Figure \ref{fig88} shows the six configurations which can be attained from
$\xi^\star_1$ with a jump of rate $1$.
Among the configurations
which can be reached at rate $1$, two belong to $\Gamma\cup\bb H_1$,
one of them being $\eta^{\mb0}$. Hence, if we denote by $\mc N_2(\xi
^\star_1)$
the set of the remaining four configurations which can be reached from
$\xi^\star_1$ by a rate $1$ jump, and by $\eta^\star_1 = \sigma^{\mb w_0,
\mb w_0-e_2+e_1} \eta^{\mb0}$ the configuration in $\Gamma\cup\bb
H_1$ which is not $\eta^{\mb0}$, decomposing $\mb P_{\xi^\star_1}  [
H_{\Xi} = H_{\xi}  ]$ according to the first jump we obtain that
%
\begin{equation}
\label{18} 6 \mb P_{\xi^\star_1} [ H_{\Xi} = H_{\xi} ]
= \mb P_{\eta^\star_1} [ H_{\Xi} = H_{\xi} ] + \sum
_{\eta' \in\mc N_2(\xi^\star_1)} \mb P_{\eta'} [ H_{\Xi} =
H_{\xi} ] + \varepsilon(\beta).
\end{equation}

The particle at position $-e_2 + e_1$ of the configuration
$\eta^\star_1= \sigma^{\mb w_0, \mb w_0-e_2+e_1} \eta^{\mb0}$, is
attached to the side $Q^{0,0}_{\mb0}$ of the quasi-square $Q^0 _{\mb
0}$. This particle moves along the side of the square at rate $1$,
while all the other jumps occur at rate at most $e^{-\beta}$. Let
$H_{\mathrm{hit}}$ be the first time the attached particle reaches the
center of the side, and let $H_{\mathrm{mv}}$ be the first time a jump of
rate $e^{-\beta}$ or less occurs. The random time $H_{\mathrm{mv}}$ is
bounded below by a mean $C_0 e^\beta$ exponential random time $\mf
e_\beta$ independent of the lateral displacement of the attached
particle. Therefore,
\[
\mb P_{\eta^\star_1} [ H_{\mathrm{mv}} \le H_{\mathrm{hit}} ] \le \mb
P_{\eta^\star_1} [ \mf e_\beta\le H_{\mathrm{hit}} ] = \int
_0^\infty\mb P_{\eta^\star_1} [ H_{\mathrm{hit}}
\ge t ] \alpha e ^{-\alpha t} \,dt,
\]
where $\alpha^{-1} = C_0 e^\beta$. It is not difficult to bound this
integral by $ 2 \{ \alpha\mb E_{\eta^\star_1} [ H_{\mathrm{hit}}
]\}^{1/2}$. We shall use repeatedly this estimate. Since $\mb
E_{\eta^\star_1} [ H_{\mathrm{hit}} ] \le C_0 n^2$, the last integral is
bounded by $C_0 n e^{-\beta/2}$.

On the event $\{H_{\mathrm{hit}} \le H_{\mathrm{mv}}\}$, $H_{\Xi} = H_{\xi}$,
where $\xi= \eta^{0,0}_{\mb0}$ is the configuration obtained from
$\eta^\star_1$ by moving the particle attached to the side of the
square to the middle position. Therefore,
\[
\mb P_{\eta^\star_1} [ H_{\Xi} = H_{\xi} ] = 1 + n
\varepsilon(\beta/2),
\]
and \eqref{18} becomes
%
\begin{equation}
\label{23} 6 \mb P_{\xi^\star_1} [ H_{\Xi} = H_{\xi} ]
= 1 + \sum_{\eta' \in\mc N_2(\xi^\star_1)} \mb P_{\eta'} [
H_{\Xi} = H_{\xi} ] + n \varepsilon(\beta/2).
\end{equation}

We examine the four configurations of $\mc N_2(\xi^\star_1)$
separately. In two configurations, $\sigma^{\mb w_0, \mb w_0-2e_2}
\eta^{\mb0}$ and $\sigma^{\mb w_0, \mb w_0-e_2-e_1} \eta^{\mb0}$, a
particle is detached from the quasi-square $Q^{0}_{\mb0}$. The
detached particle performs a rate $1$ symmetric random walk on $\bb
T_L$ until it reaches the boundary of the square $Q_{\mb0}$. Among
the remaining particles, two jumps have rate $e^{-\beta}$ and all the
other ones have rate at most $e^{-2\beta}$. Denote by $H_{\mathrm{hit}}$
the time the detached particle hits the outer boundary of the square
$Q_{\mb0}$ and by $H_{\mathrm{mv}}$ the first time a particle on the
quasi-square moves.

As above, $H_{\mathrm{mv}}$ can be bounded below by a mean $C_0 e^\beta$
exponential random variable $\mf e_\beta$ independent of the
displacement of the detached particle. Therefore, for $\eta' =
\sigma^{\mb w_0, \mb w_0-2e_2} \eta^{\mb0}$, $\sigma^{\mb w_0, \mb
w_0-e_2-e_1} \eta^{\mb0}$,
\[
\mb P_{\eta'} [ H_{\mathrm{mv}} \le H_{\mathrm{hit}} ] \le \mb
P_{\eta'} [ \mf e_\beta\le H_{\mathrm{hit}} ] = \int
_0^\infty\mb P_{\eta'} [ H_{\mathrm{hit}}
\ge t ] \alpha e ^{-\alpha t} \,dt,
\]
where $\alpha^{-1} = C_0 e^\beta$. Since, by \cite{lpw1}, Proposition
10.13, $\mb E_{\eta'} [ H_{\mathrm{hit}} ] \le C_0 L^2$, this last
integral is bounded by $C_0 L e^{-\beta/2}$. Hence,
%
\begin{equation}
\label{17} \mb P_{\eta'} [ H_{\Xi} = H_{\xi} ] =
\mb P_{\eta'} [ H_{\Xi} = H_{\xi}, H_{\mathrm{hit}}
\le H_{\mathrm{mv}} ] + L \varepsilon(\beta/2).
\end{equation}
On the event $H_{\mathrm{hit}} \le H_{\mathrm{mv}}$ the configuration $\eta
(H_{\mathrm{hit}})$ is either $\xi^\star_1$, $\xi^\star_2$ or it belongs to~$\bb
H_1$. Therefore, by the strong Markov property, the previous
expression is equal to
\begin{eqnarray*}
&& \sum_{i=1}^2 \mb P_{\eta'}
\bigl[ \eta(H_{\mathrm{hit}} ) = \xi^\star_i \bigr] \mb
P_{\xi^\star_i} [ H_{\Xi} = H_{\xi} ]
\\
& &\qquad{}+ \sum_{j=0}^3 \mb E_{\eta'}
\bigl[ \mb1\bigl\{ \eta(H_{\mathrm{hit}} ) \in \mc E^{0,j}_{\mb0}
\bigr\} \mb P_{\eta(H_{\mathrm{hit}} )} [ H_{\Xi} = H_{\xi} ] \bigr] +
L \varepsilon(\beta/2).
\end{eqnarray*}

Once the detached particle hits the side of the quasi-square, as we
have seen above, it attains the middle position of the side before
anything else happens with a probability close to $1$. Since $\xi=
\eta^{0,0}_{\mb0}$, the previous sum over $j$ is equal to
\[
\mb P_{\eta'} \bigl[ \eta(H_{\mathrm{hit}} ) \in \mc
E^{0,0}_{\mb0} \bigr] + n \varepsilon(\beta/2).
\]
By definition \eqref{26} of $\mf p$, the contribution of the terms
$\eta' = \sigma^{\mb w_0, \mb w_0-2e_2} \eta^{\mb0}$ and $\eta' =
\sigma^{\mb w_0, \mb w_0-e_2-e_1} \eta^{\mb0}$ to the sum appearing
on the right-hand side of \eqref{23} is
%
\begin{eqnarray}
\label{25} && \bigl\{ \mf p(\mb w_0 - 2e_2, \mb
w_0 - e_2) + \mf p(\mb w_0 -
e_1 - e_2, \mb w_0 - e_2) \bigr
\} \mb P_{\xi^\star_1} [ H_{\Xi} = H_{\xi} ]
\nonumber
\\
&&\qquad{} + \bigl\{ \mf p(\mb w_0 - 2e_2, \mb w_0
- e_1) + \mf p(\mb w_0 - e_1 -
e_2, \mb w_0 - e_1) \bigr\} \mb
P_{\xi^\star_2} [ H_{\Xi} = H_{\xi} ]
\\
&&\qquad{} + \mf p\bigl(\mb w_0 - 2e_2, Q^{0,0}_{\mb0}
\bigr) + \mf p\bigl(\mb w_0 - e_1-e_2,
Q^{0,0}_{\mb0}\bigr) + L \varepsilon(\beta/2).\nonumber
\end{eqnarray}
In this equation, we have replaced $n \varepsilon(\beta/2)$ by $L
\varepsilon(\beta/2)$ because $n\le2 L$.

It remains to analyze the two configurations of $\mc N_2(\eta)$, $\eta' =
\sigma^{\mb w_0+e_2, \mb w_0-e_2} \eta^{\mb0}$ and $\eta' =
\sigma^{\mb w_0+e_1, \mb w_0-e_2} \eta^{\mb0}$. In the first one, if
we denote by $\mb z^1_t$ the horizontal position of the particle
attached to the side $Q^0$ of the square $Q$ and by $\mb z^2_t$ the
vertical position of the hole on the side $Q^3$, it is not difficult
to check that $(\mb z^1_t, \mb z^2_t)$ evolves as the Markov chain
described just before \eqref{30} with initial condition $(\mb z^1_0,
\mb z^2_0) = (0,1)$.

Denote by $H_{\mathrm{hit}}$ the time the hole hits $0$ or $n-1$ and by
$H_{\mathrm{mv}}$ the first time a jump of rate $e^{-\beta}$ or less
occurs. The arguments presented above in this proof show that the
probability $\mb P_{\eta'} [H_{\mathrm{mv}} \le H_{\mathrm{hit}}]$ is bounded by
$ n \varepsilon(\beta/2)$. Hence, \eqref{17} holds with this new
meaning for $H_{\mathrm{mv}}$, $H_{\mathrm{hit}}$, and with $n$ in place of
$L$. On the event $H_{\mathrm{hit}} \le H_{\mathrm{mv}}$ three situations can
occur. If the process $(\mb z^1_t, \mb z^2_t)$ reached $(0,0)$
(resp., $E^+_n$, $E^-_n$), the process $\eta(t)$ returned to the
configuration $\xi^\star_1$ (resp., hit a configuration in $\mc
E^{3,0}_{\mb0}$, $\mc E^{0,0}_{\mb0}$). At this point, we may repeat
the previous arguments to replace the set $\mc E^{3,0}_{\mb0}$, $\mc
E^{0,0}_{\mb0}$ by the configurations $\eta^{3,0}_{\mb0}$,
$\eta^{0,0}_{\mb0}$, respectively. Hence, by definition of $\mf
r^\pm_n$,
\[
\mb P_{\eta'} [ H_{\Xi} = H_{\xi} ] = \mf
r^-_n + (1-\mf r_n) \mb P_{\xi^\star_1} [
H_{\Xi} = H_{\xi} ] + n \varepsilon (\beta/2)
\]
for $\eta' = \sigma^{\mb w_0+e_2, \mb w_0-e_2} \eta^{\mb0}$,
$\xi=\eta^{0,0}_{\mb0}$.

Assume now that $\eta' = \sigma^{\mb w_0+e_1, \mb w_0-e_2} \eta^{\mb
0}$. In this case, if we denote by $\mb y^1$ the horizontal
position of the particle attached to the side of the quasi-square and
by $\mb y^2$ the horizontal position of the hole, the pair $(\mb y^1_t, \mb y^2_t)$ evolve according to the Markov process introduced just
before \eqref{29} with initial condition $(\mb y^1_0, \mb y^2_0) = (0, 1)$. Denote by $H_{\mathrm{hit}}$ the time the hole hits $0$ or $n-1$
and by $H_{\mathrm{mv}}$ the first time a jump of rate $e^{-\beta}$ or less
occurs. The arguments presented above in this proof show that the
probability $\mb P_{\eta'} [H_{\mathrm{mv}} \le H_{\mathrm{hit}}]$ is bounded by
$ n \varepsilon(\beta/2)$. Hence, \eqref{17} holds with this new
meaning for $H_{\mathrm{mv}}$, $H_{\mathrm{hit}}$, and with $n$ in place of
$L$. On the event $H_{\mathrm{hit}} \le H_{\mathrm{mv}}$, if $\mb y^2_{H_{\mathrm{hit}}}=0$, at time $H_{\mathrm{hit}}$ the process $\eta(t)$ has returned
to the configuration $\xi^\star_1$, while if $\mb y^2_{H_{\mathrm{hit}}}=n-1$,
at time $H_{\mathrm{hit}}$ the process $\eta(t)$ has reached a
configuration in $\mc E^{1,0}_{\mb0}$. At this point, we may repeat
the previous arguments to replace the set $\mc E^{1,0}_{\mb0}$ by the
configurations $\eta^{1,0}_{\mb0}$. Since the random walk reaches
$n-1$ before $0$ with probability $\mf q_n$, by the strong Markov
property
\[
\mb P_{\eta'} [ H_{\Xi} = H_{\xi} ] = (1-\mf
q_n) \mb P_{\xi^\star_1} [ H_{\Xi} = H_{\xi}
] + n \varepsilon (\beta/2)
\]
for $\eta' = \sigma^{\mb w_0+e_1, \mb w_0-e_2} \eta^{\mb0}$,
$\xi=\eta^{0,0}_{\mb0}$.

Therefore, the contribution of the last two configurations of
$\mc N_2(\xi^\star_1)$ to the sum on the right-hand side of \eqref{23} is
%
\begin{equation}
\label{24} \mf r^-_n + ( 2 - \mf r_n - \mf
q_n ) \mb P_{\xi^\star_1} [ H_{\Xi} = H_{\xi}
] + n \varepsilon (\beta/2).
\end{equation}

Equations \eqref{23}, \eqref{25} and \eqref{24} yield a linear
equation for $P_1=\mb P_{\xi^\star_1} [ H_{\Xi} = H_{\xi} ]$ in terms of
$P_1$ and $P_2=\mb P_{\xi^\star_2} [ H_{\Xi} = H_{\xi} ]$. Analogous
arguments provide a similar equation for $P_2$ in terms of $P_1$ and
$P_2$. Adding these two equations, we obtain~\eqref{19}, while
subtracting them gives an expression for the difference $P_1-P_2$. The
assertion of the lemma follows from these equations for $P_1+P_2$ and
$P_1-P_2$.
\end{pf*}

\begin{lemma}
\label{s13}
We have that
\begin{eqnarray*}
 8 (4 + \mf q_n + \mf r_n - \mf A) \mb
P_{\eta^{\mb0}} \bigl[ H^+_{\Xi} = H_{\eta^{0,0}_{\mb0}} \bigr]& = &1 + \mf
A_{0,3} + \mf r^-_n + \mf q_n +
\kappa_1,
\\
 8 (4 + \mf q_n + \mf r_n - \mf A) \mb
P_{\eta^{\mb0}} \bigl[ H^+_{\Xi} = H_{\eta^{0,1}_{\mb0}} \bigr] &= &\mf
A_{1,2} + \mf r^+_n + \kappa_1,
\\
(4 + \mf q_n + \mf r_n - \mf A) \mb
P_{\eta^{\mb0}} \bigl[ H^+_{\Xi} = H_{\eta^{\mb0}} \bigr] &=& 1 +
\kappa_1,
\end{eqnarray*}
where $\kappa_1$ is a remainder absolutely bounded by $C_0 L
e^{-\beta/2}$.
\end{lemma}

Note that using the symmetry of the model we can deduce from this
result the values of $\mb P_{\eta^{\mb0}} [ H^+_{\Xi} =
H_{\eta^{i,j}_{\mb0}}]$ for all $0\le i,j\le3$. For example,
\begin{eqnarray*}
\mb P_{\eta^{\mb0}} \bigl[ H^+_{\Xi} &=& H_{\eta^{0,3}_{\mb0}} \bigr] =
\mb P_{\eta^{\mb0}} \bigl[ H^+_{\Xi} = H_{\eta^{0,0}_{\mb0}} \bigr],
\\
\mb P_{\eta^{\mb0}} \bigl[ H^+_{\Xi}& =& H_{\eta^{0,2}_{\mb0}} \bigr] =
\mb P_{\eta^{\mb0}} \bigl[ H^+_{\Xi} = H_{\eta^{0,1}_{\mb0}} \bigr].
\end{eqnarray*}
Moreover, summing over all configurations in $\mc G_{\mb0}$, we get
that
\[
\mb P_{\eta^{\mb0}} \bigl[ H^+_{\Xi} \neq H_{\mc G_{\mb0}} \bigr] =
\kappa_1.
\]

\begin{pf*}{Proof of Lemma \ref{s13}}
We prove the lemma for $\xi= \eta^{0,0}_{\mb0}$. Decomposing the
event $\{H^+_{\Xi} = H_{\xi}\}$ according to the first jump of the
process $\eta(t)$ we obtain that
\[
\mb P_{\eta^{\mb0}} \bigl[ H^+_{\Xi} = H_{\xi}\bigr] =
\frac{1}8 \sum_{\eta\in\mc N(\eta^{\mb0})} \mb P_{\eta} [
H_{\Xi} = H_{\xi} ] + n \varepsilon(\beta),
\]
where $\mc N(\eta^{\mb0})$ is the set of $8$ configurations which can be
obtained from $\eta^{\mb0}$ by a jump of rate $e^{-2\beta}$. Lemma
\ref{s20} provides a formula $\mb P_{\eta} [ H_{\Xi} = H_{\xi} ]$
for each $\eta$ in $\mc N(\eta^{\mb0})$, which concludes the proof of the
lemma.
\end{pf*}

\begin{corollary}
\label{s12}
There exists a finite constant $C_0$ such that
\begin{eqnarray*}
 p_{\zeta}\bigl(\eta^{\mb0}, \eta^{0,0}_{\mb0}
\bigr) &=& \frac{1}8 \frac{1 + \mf A_{0,3} + \mf r^-_n + \mf q_n}{
1 + \mf A_{0,3} + \mf A_{1,2} + \mf r_n
+ \mf q_n} + \kappa_1,
\\
 p_{\zeta}\bigl(\eta^{\mb0}, \eta^{0,1}_{\mb0}
\bigr) &=& \frac{1}8 \frac{\mf A_{1,2} + \mf r^+_n}{
1 + \mf A_{0,3} + \mf A_{1,2} + \mf r_n
+ \mf q_n} + \kappa_1,
\end{eqnarray*}
where $\kappa_1$ is a remainder absolutely bounded by $C_0 L
e^{-\beta/2}$.
\end{corollary}

Here again, using the symmetry of the model, we can deduce from this
result the values of $p_{\zeta}(\eta^{\mb x}, \eta^{i,j}_{\mb x})$
for all $0\le i,j\le3$, $\mb x\in\bb T_L$. Moreover, it follows
from this result that
%
\begin{equation}
\label{31} \sum_{\xi\notin\mc V(\eta^{\mb x})} p_{\zeta}\bigl(
\eta^{\mb x}, \xi\bigr) \le \kappa_1.
\end{equation}

\begin{pf*}{Proof of Corollary \ref{s12}}
By the displayed equation after (6.9) in \cite{bl2},
\[
p_{\zeta}\bigl(\eta^{\mb0}, \xi\bigr) = \mb P_{\eta^{\mb0}} [
H_{\Xi\setminus\{\eta^{\mb0}\}} = H_{\xi} ].
\]
Intersecting the previous event with the partition $\{H^+_{\Xi} =
H^+_{\{\eta^{\mb0}\}}\}$, $\{H^+_{\Xi} < H^+_{\{\eta^{\mb0}\}}\}$, we
obtain by the strong Markov property that
\[
\mb P_{\eta^{\mb0}} [ H_{\Xi\setminus\{\eta^{\mb0}\}} = H_{\xi} ] =
\frac{\mb P_{\eta^{\mb0}}  [ H^+_{\Xi} =
H_{\xi}  ]}{\mb P_{\eta^{\mb0}}  [ H^+_{\Xi} \neq
H^+_{\eta^{\mb0}}  ]} \cdot
\]
It remains to recall the statement of Lemma \ref{s13} to conclude the
proof of the corollary.
\end{pf*}

\begin{pf*}{Proof of Proposition \ref{s16}}
Since $R_{\zeta}(\eta^{\mb0}, \xi) = \lambda(\eta^{\mb0}) \mb
P_{\eta^{\mb0}}  [ H^+_{\Xi} = H_{\xi}  ]$ and since $\lambda
(\eta^{\mb0}) = 8 e^{-2\beta} (1+ n e^{-\beta})$, the assertion
follows from Lemma \ref{s13}.
\end{pf*}

\subsection*{The jump rates of $\zeta(t)$ on $\Xi\setminus\Gamma$}
We examine in this subsection the jump rates of order $e^{-\beta}$ of
the trace process $\zeta(t)$. Let
\[
\kappa_2 = n^4 e^{-\beta} + n L e^{-(1/2)\beta}.
\]

\begin{proposition}
\label{s08}
For $\eta\in\Xi\setminus\Gamma$, let $\mc V(\eta) = \{\xi\in\Xi\dvtx
\mb R(\eta,\xi)>0\}$. There exists a finite constant $C_0$ such that
for all $\eta\in\Xi\setminus\Gamma$,
\begin{eqnarray*}
 \max_{\xi\in\mc V(\eta)} \bigl| R_\zeta(\eta, \xi) -
e^{-\beta} \mb R(\eta,\xi) \bigr| &\le&C_0 e^{-\beta}
\kappa_2,
\\
\max_{\eta\in\Xi\setminus\Gamma} R_\zeta\bigl(\eta, \mc V(
\eta)^c\bigr)& \le& C_0 e^{-\beta}
\kappa_2.
\end{eqnarray*}
\end{proposition}

It follows from this result and some simple algebra that there exists
a finite constant $C_0$ such that for all $\eta\in\Xi\setminus
\Gamma$, $\xi\in\Xi$,
%
\begin{equation}
\label{41} \frac{1}{R_\zeta(\eta, \xi)} \le \frac{e^{\beta}}{
\mb R(\eta,\xi)} \biggl\{ 1 +
\frac{
C_0 \kappa_2}{\mb R(\eta,\xi)
- C_0 \kappa_2} \biggr\}.
\end{equation}

The proof of Proposition \ref{s08} is divided in several lemmas.

\begin{lemma}
\label{s03}
There exists a finite constant $C_0$ such that for all $\mb
x\in\bb T_L$, $0\le i, j\le3$,
\begin{eqnarray*}
 \max_{\xi\in\mc V(\eta^{i,j}_{\mb x})}\biggl | R_\zeta\bigl(\eta^{i,j}_{\mb x},
\xi\bigr) - e^{-\beta} \sum_{\zeta\in\mc N(\mc E^{i,j}_{\mb x})} \bb M (
\zeta, \xi)\biggr | &\le& C_0 e^{-\beta} \kappa_2,
\\
 R_\zeta\bigl(\eta^{i,j}_{\mb x}, \mc V\bigl(
\eta^{i,j}_{\mb x}\bigr)^c\bigr) &\le& C_0
e^{-\beta} \kappa_2.
\end{eqnarray*}
\end{lemma}

\begin{pf}
Let $\eta= \eta^{i,j}_{\mb x}$ and fix a configuration $\xi\in
\Xi$, $\xi\neq\eta$. Recall the formula for the jump rate
$R_\zeta(\eta,\xi)$ provided by Lemma \ref{s02}. We first claim that
%
\begin{equation}
\label{08}\bigl | \mu_K(\eta)^{-1} \Cap_K(
\eta,\Delta_1) - \bigl|\mc N\bigl(\mc E^{i,j}_{\mb x}
\bigr)\bigr| e^{-\beta} \bigr| \le C_0 n^3 e^{-2\beta}
\end{equation}
for some finite constant $C_0$. Indeed, let $W$ be the equilibrium
potential: $W(\zeta) = \mb P_\zeta[H_\eta<H_{\Delta_1}]$, $\zeta\in
\mc E^{i,j}_{\mb x}$. Note that $\zeta$ in the previous formula is a
subscript and not a superscript. The probability refers to the process
$\eta(t)$ and not to the trace process $\zeta(t)$. By the Dirichlet
principle \cite{ds1}, $\Cap_K(\eta,\Delta_1) = D_K(W)$. To fix
ideas, assume that $i=j=2$, $\mb x=\mb w$. Denote by $\zeta_j$, $1\le
j\le n-1$, the configuration of $\mc E^{2,2}_{\mb w}$ which has a
particle at site $\mb u_j=(j-1, n)$. By the strong Markov property,
\[
\bigl\{1+ 4 e^{-\beta} + \delta_1\bigr\} W(
\zeta_1) = W(\zeta_2),
\]
where $\delta_1 \in(0,C_0 e^{-2\beta})$ represents the rate at which
the process jumps from $\zeta_1$ to configurations in $\Delta_2$. We
may write similar identities for $W(\zeta_j)$, $1\le j\le n-1$. Since
$0\le W\le1$ and since $W(\eta)=1$, we deduce from these identities
that $|W(\zeta_{j+1}) - W(\zeta_j)|\le5ne^{-\beta}$ and that
$|W(\zeta_j) - 1|\le5n^2e^{-\beta}$. To conclude the proof of
\eqref{08}, it remains to recall the explicit expression for the
Dirichlet form $D_K (W)$.

We turn to the second term in \eqref{07}. We claim that
%
\begin{equation}
\label{09} \max_{\zeta\in\Delta_1} \biggl| \mb P_{\eta} [
H_\zeta= H_{\Delta_1} ] - \frac{1}{|\mc N(\mc E^{i,j}_{\mb x})|} \mb1\bigl\{\zeta\in
\mc N\bigl(\mc E^{i,j}_{\mb x}\bigr) \bigr\} \biggr| \le
C_0 n^2 e^{-\beta}
\end{equation}
for some finite constant $C_0$. To fix ideas, assume again that
$i=j=2$, $\mb x=\mb w$, and observe that $|\mc N(\mc E^{2,2}_{\mb
w})|=3n$. It is clear that $\mb P_{\eta} [H_\zeta= H_{\Delta_1}]$ is
bounded by $C_0 e^{-\beta}$ if $\zeta$ does not belong to $\mc N(\mc
E^{2,2}_{\mb w})$. Fix $\zeta\in\mc N(\mc E^{2,2}_{\mb w})$ and let
$V(\zeta_j) = \mb P_{\zeta_j} [ H_\zeta= H_{\Delta_1}] $, $1\le j\le
n-1$. By the strong Markov property,
%
\begin{equation}
\label{10} \bigl\{1+ 4 e^{-\beta} + \delta_1\bigr\} V(
\zeta_1) = V(\zeta_2) + e^{-\beta} \mb1(\zeta,
\zeta_1),
\end{equation}
where $0<\delta_1<C_0 e^{-2\beta}$ represents the rate at which the
process jumps from $\zeta_1$ to configurations in $\Delta_2$, and $\mb
1(\zeta,\zeta_1)$ is equal to $1$ if $\zeta$ can be obtained from
$\zeta_1$ by moving one particle, and is equal to $0$
otherwise. Similar identities can be obtained for $\zeta_j$, $2\le
j\le n-1$. For some configurations, the factor $4$ is replaced by
$3$. Summing all these identities and dividing by $e^{-\beta}$, we
obtain that
\[
3 \sum_{j=1}^{n-1} V(\zeta_j) +
V(\zeta_1) + V(\zeta_2) + V(\zeta_{n-2}) +
e^{\beta} \sum_{j=1}^{n-1}
\delta_j V(\zeta_j) = 1.
\]
It also follows from the identities \eqref{10} and from the bound $0\le
V\le1$ that $|V(\zeta_{2}) - V(\zeta_{1})| \le6 e^{-\beta}$,
$|V(\zeta_{j+2}) - V(\zeta_{j+1})| \le|V(\zeta_{j+1}) - V(\zeta_j)| +
6 e^{-\beta}$ for $1\le j\le n-3$ and for $\beta$ large. Iterating
these inequalities, we obtain that $|V(\zeta_{k}) - V(\zeta_j)| \le6
n^2 e^{-\beta}$, $1\le j\le k\le n-1$. Hence, summing and subtracting
$V(\eta)$ in the last displayed formula we get that
\[
\biggl| V(\eta) - \frac{1}{3n}\biggr | \le C_0n^2e^{-\beta},
\]
which proves \eqref{09} in the case $i=j=2$.

It remains to evaluate the third term in the sum \eqref{07}.
We claim that
%
\begin{equation}
\label{11} \max_{\zeta\in\mc N(\mc E^{i,j}_{\mb x})} \max_{\xi\in\Xi} \bigl| \mb
P_{\zeta} [ H_{\Xi} = H_{\xi} ] - \bb M (\zeta, \xi)
\bigr| \le C_0 L e^{-\beta/2}.
\end{equation}
We prove \eqref{11} in the case $i=j=2$, $\mb x=\mb w$,
$\zeta=\sigma^{\mb w_2,\mb z}\eta^{\mb w}$ for some site $\mb z$ at
distance $1$ from $Q^{2,2}_{\mb w}$. The other cases, which are
simpler, are left to the reader. For $\zeta=\sigma^{\mb w_2,\mb
z}\eta^{\mb w}$, the probability measure $\bb M (\zeta,
\cdot)$ gives positive weight only to the configurations $\eta^{\mb
0}$, $\eta^{2,j}_{\mb w}$, $0\le j\le3$, $\eta^{1,2}_{\mb w}$,
$\eta^{3,2}_{\mb w}$, $\eta^{1,1}_{\mb w}$ and $\eta^{3,1}_{\mb w}$.
We examine the case $\xi= \eta^{2,2}_{\mb w}$ and leave the others
to the reader.

Until the particle at site $\mb z$ hits the outer boundary of the
square $Q$, it evolves as a rate $1$, symmetric, nearest-neighbor
random walk, while all the other particles move at rate at most
$e^{-\beta}$. Denote by $H_{\mathrm{hit}}$ the time the particle initially
at site $\mb z$ hits $\partial_+ Q$ and by $H_{\mathrm{mv}}$ the time of
the first jump of a particle sitting in the quasi-square~$Q^{2,2}_{\mb
0}$.

The random time $H_{\mathrm{mv}}$ is bounded below by a mean $C_0 e^\beta$
random time $\mf e_\beta$ independent of the motion of the detached
particle. Hence, as in the proof of Lemma~\ref{s13},
%
\begin{equation}
\label{05} \mb P_{\zeta} [ H_{\mathrm{mv}} < H_{\mathrm{hit}} ]
\le \mb P_{\zeta} [ H_{\mathrm{hit}} > \mf e_\beta ] = \int
_0^\infty\mb P_{\zeta} [ H_{\mathrm{hit}}
> t ] \alpha e^{- \alpha t} \,dt,
\end{equation}
where $\alpha^{-1} = C_0 e^\beta$. Since, by \cite{lpw1}, Proposition
10.13, $\mb E_{\zeta} [ H_{\mathrm{hit}} ] \le C_0 L^2$, this last
integral is bounded by $C_0 L e^{-\beta/2}$. Therefore,
%
\begin{equation}
\label{06} \mb P_{\zeta} [ H_{\Xi} = H_{\xi} ] =
\mb P_{\zeta} [ H_{\Xi} = H_{\xi}, H_{\mathrm{hit}}
< H_{\mathrm{mv}} ] + L \varepsilon(\beta/2).
\end{equation}

On the set $H_{\mathrm{hit}} < H_{\mathrm{mv}}$, $\eta(H_{\mathrm{hit}})$ belongs to
the sets $\mc E^{2,j}_{\mb0}$, $0\le j\le3$, or is equal to the
configurations $\eta^\star_1 = \sigma^{\mb w_2,\mb w_2+e_2}\eta^{\mb w}$,
$\eta^\star_2 = \sigma^{\mb w_2,\mb w_2+e_1}\eta^{\mb w}$. Hence, by the
strong Markov property, the previous expression is equal to
\begin{eqnarray*}
&& \sum_{j=0}^3 \mb E_{\zeta}
\bigl[ \mb1 \{ H_{\mathrm{hit}} < H_{\mathrm{mv}} \} \mb1 \bigl\{
\eta(H_{\mathrm{hit}}) \in\mc E^{2,j}_{\mb0} \bigr\} \mb
P_{ \eta(H_{\mathrm{hit}})} [ H_{\Xi} = H_{\xi} ] \bigr]
\\
&&\qquad{} + \sum_{i=1}^2 \mb P_{\zeta}
\bigl[ H_{\mathrm{hit}} < H_{\mathrm{mv}}, \eta(H_{\mathrm{hit}}) =
\eta^\star_i \bigr] \mb P_{\eta^\star_i} [
H_{\Xi} = H_{\xi} ] + L \varepsilon(\beta/2).
\end{eqnarray*}
Proceeding as in the proof of Lemma \ref{s13}, we may replace the sets
$\mc E^{2,j}_{\mb0}$ by the configurations $\eta^{2,j}_{\mb0}$ with
an error bounded by $C_0 n e^{-\beta/2} \le C_0 L e^{-\beta/2}$. Since
$\xi=\eta^{2,2}_{\mb0}$, only the term $j=2$ gives a positive
contribution. By symmetry, $\mb P_{\eta^\star_i}  [ H_{\Xi} =
H_{\xi}  ] = P_i(\xi^{0,0}_{\mb w})$, $i=1$, $2$, where $P_i$ has
been introduced in Lemma \ref{s20}. Hence, by the definition
\eqref{26} of $\mf p$, the previous sum is equal to
\[
\mf p\bigl(\mb z, Q^{2,2}_{\mb0}\bigr) + \mf p(\mb z, \mb
w_2 + e_2) P_1 \bigl(\xi^{0,0}_{\mb w}
\bigr) + \mf p(\mb z, \mb w_2 + e_1) P_2
\bigl(\xi^{0,0}_{\mb w}\bigr) + L \varepsilon(\beta/2).
\]
At this point, \eqref{11} follows from Lemma \ref{s20}.

The assertion of the lemma is a consequence of Lemma \ref{s02} and of
estimates \eqref{08}, \eqref{09} and \eqref{11}. We first use estimate
\eqref{08} and the fact that $\mb P_\eta[H_{\Delta_1} = H_\zeta]$ is
a probability measure in $\zeta$ to replace $\mu_K(\eta)^{-1}
\Cap_K(\eta,\Delta_1)$ by $|\mc N(\mc E^{i,j}_{\mb x})| e^{-\beta}$
with an error bounded by $C_0 n^3 e^{-2\beta}$. Then we consider
separately the sum over $\zeta\in\mc N(\mc E^{i,j}_{\mb x})$ and the
sum over $\zeta\in\Delta_1\setminus\mc N(\mc E^{i,j}_{\mb x})$. For
the first one, we use \eqref{09} to replace $\mb P_\eta[H_{\Delta_1}
= H_\zeta]$ by $|\mc N(\mc E^{i,j}_{\mb x})|^{-1} \mb1\{\zeta\in\mc
N(\mc E^{i,j}_{\mb x})\}$ with an error bounded by $C_0 n^4
e^{-2\beta}$ because $|\mc N(\mc E^{i,j}_{\mb x})|$ is less than or
equal to $C_0n$. To estimate the sum over $\zeta\in\Delta_1\setminus
\mc N(\mc E^{i,j}_{\mb x})$, we proceed as follows. Since $|\mc N(\mc
E^{i,j}_{\mb x})| \le C_0 n$, we have that
\begin{eqnarray*}
&& \bigl|\mc N\bigl(\mc E^{i,j}_{\mb x}\bigr)\bigr| e^{-\beta} \sum
_{\zeta\in\Delta_1\setminus\mc N(\mc E^{i,j}_{\mb x})} \mb P_\eta[H_{\Delta_1}
=H_\zeta] \mb P_{\zeta} [ H_{\Xi} = H_{\xi}
]
\\
&&\qquad \le C_0 n e^{-\beta} \sum_{\zeta\in\Delta_1\setminus\mc N(\mc E^{i,j}_{\mb x})}
\mb P_\eta[H_{\Delta_1} =H_\zeta].
\end{eqnarray*}
This last sum can be written as
\[
1 - \sum_{\zeta\in\mc N(\mc E^{i,j}_{\mb x})} \mb P_\eta[H_{\Delta_1}
= H_\zeta] = \sum_{\zeta\in\mc N(\mc E^{i,j}_{\mb x})} \biggl\{
\frac{1}{|\mc N(\mc E^{i,j}_{\mb x})|} - \mb P_\eta[H_{\Delta_1} = H_\zeta]
\biggr\}.
\]
By \eqref{09}, this expression is bounded by $C_0 n^3 e^{-\beta}$.
Finally, we use \eqref{11} to replace $\mb P_\zeta[H_{\Xi} = H_\xi]$
by $\bb M (\zeta,\xi)$ with an error bounded by $C_0 n L
e^{-(3/2)\beta}$ because $|\mc N(\mc E^{i,j}_{\mb x})|$ is less than
or equal to $C_0n$. This concludes the proof of the first assertion of
the lemma. The proof of the second assertion is analogous.
\end{pf}

\begin{lemma}
\label{s05}
There exists a finite constant $C_0$ such that for all $\mb
x\in\bb T_L$, $\mf a\in\{\mf s, \mf l\}$, $(\mb k,\bs
\ell) \in I^*_{\mf a}$
\[
\max_{\xi\in\Xi}\biggl | R_\zeta\bigl(\eta_{\mb x}^{\mf a, (\mb k,\bs\ell)},
\xi\bigr) - e^{-\beta} \sum_{\zeta\in\mc N(\eta_{\mb x}^{\mf a,
(\mb k,\bs\ell)})} \bb M (
\zeta, \xi) \biggr| \le C_0 n^{1/2} e^{-(3/2)\beta}.
\]
\end{lemma}

\begin{pf}
Fix $\mb x\in\bb T_L$, $\mf a\in\{\mf s, \mf l\}$, $(\mb k,\bs
\ell) \in I^*_{\mf a}$ and let $\eta= \eta_{\mb x}^{\mf a, (\mb
k,\bs\ell)}$. In view of \eqref{07}, we need to compute three
expressions.

We first claim that
\[
\bigl| \mu_K(\eta)^{-1} \Cap_K(\eta,
\Delta_1) -\bigl |\mc N(\eta)\bigr| e^{-\beta} \bigr| \le C_0
e^{-2\beta}
\]
for some finite constant $C_0$. Indeed, let $W$ be the equilibrium
potential: $W(\xi) = \mb P_\xi[H_\eta<H_{\Delta_1}]$. By the Dirichlet
principle, $\Cap_K(\eta,\Delta_1) = D_K(W)$. From $\eta$, the
process jumps at rate $e^{-\beta}$ (resp., $e^{-2\beta}$,
$e^{-3\beta}$) to $|\mc N(\eta)|$ (resp., at most $C_0$, $4n$)
configurations in $\Delta_1$. Hence, since $ne^{-\beta}\ll1$,
$\mu_K(\eta)^{-1} D_K(W) = |\mc N(\eta)| e^{-\beta} \pm C_0
e^{-2\beta}$, proving the claim.

By similar reasons,
\[
\max_{\zeta\in\Delta_1}\biggl | \mb P_{\eta} [ H_\zeta=
H_{\Delta_1} ] - \frac{1}{|\mc N(\eta)|} \mb1\bigl\{\zeta\in\mc N(\eta)
\bigr\}\biggr |
\le C_0 e^{-\beta}
\]
for some finite constant $C_0$. Finally, we claim that
\[
\max_{\zeta\in\mc N(\eta)} \max_{\xi\in\Xi} \bigl| \mb
P_{\zeta} [ H_{\Xi} = H_{\xi} ] - \bb M (\zeta, \xi)
\bigr| \le C_0 \sqrt{ n e^{-\beta} }.
\]
There are two different cases. Assume that $\zeta$ is obtained from
$\eta$ by moving a particle on a side which has $3\le m\le n$
particles. In this case, the configuration $\zeta$ has a hole in one of
the sides of the square which moves according to a rate~$1$
nearest-neighbor, symmetric random walk until it reaches the boundary
of the set of $m$ particles. Let $H_{\mathrm{hit}}$ be the time the hole
attains the boundary and let $H_{\mathrm{mv}}$ be the first time a jump of
rate $e^{-\beta}$ or less occurs. $H_{\mathrm{mv}}$ is bounded below by a
mean $C_0 e^{\beta}$ exponential random variable $\mf e_\beta$,
independent of the displacement of the hole. Since $\mb E_\zeta
[H_{\mathrm{hit}}] = (m-1)/2$, by an argument repeatedly used in the
previous proofs, $\mb P_{\zeta} [ H_{\mathrm{mv}} < H_{\mathrm{hit}}] \le C_0
(ne^{-\beta})^{1/2}$. Hence,
\[
\mb P_{\zeta} [ H_{\Xi} = H_{\xi} ] = \mb
P_{\zeta} [ H_{\Xi} = H_{\xi}, H_{\mathrm{hit}} <
H_{\mathrm{mv}} ] + \sqrt{n} \varepsilon(\beta/2).
\]

On the set $\{H_{\mathrm{hit}} < H_{\mathrm{mv}}\}$, $H_{\mathrm{hit}} = H_{\Xi}$
and either $\eta(H_{\mathrm{hit}})$ is the original configuration or it is
the one in which a row or a column of particles in a side of the
rectangle has been translated by one unit. The first case has
probability $(m-1)/m$ and the second one $1/m$.

This argument can be extended to the case where $\zeta$ is obtained
from $\eta$ by moving a particle on a side which has $2$
particles.

To conclude the proof, it remains to put together the previous three
estimates, as in the previous lemma.
\end{pf}

The proof of the next lemma is similar to that of Lemma \ref{s03}
and the proof of the following one is identical to that of Lemma
\ref{s05}.

\begin{lemma}
\label{s07}
There exists a finite constant $C_0$ such that for all $\mb
x\in\bb T_L$, $\mf a\in\{\mf s, \mf l\}$, $0\le j\le3$,
\[
\max_{\xi\in\Xi} \biggl| R_\zeta\bigl(\zeta^{\mf a,j}_{\mb x},
\xi\bigr) - e^{-\beta} \sum_{\zeta\in\mc N(\mc E^{\mf a,j}_{\mb x})} \bb M (
\zeta, \xi) \biggr| \le C_0 e^{-2\beta}.
\]
\end{lemma}

\begin{lemma}
\label{s06}
There exists a finite constant $C_0$ such that for all $\mb
x\in\bb T_L$, $\mf a\in\{\mf s, \mf l\}$, $(\mb k,\bs
\ell) \in I_{2, \mf a}^*$,
\[
\max_{\xi\in\Xi} \biggl| R_\zeta\bigl(\zeta_{\mb x}^{\mf a, (\mb k,\bs\ell)},
\xi\bigr) - e^{-\beta} \sum_{\zeta\in\mc N(\zeta_{\mb x}^{\mf a,
(\mb k,\bs\ell)})} \bb M (
\zeta, \xi) \biggr| \le C_0 n e^{-(3/2)\beta}.
\]
\end{lemma}

\section{Coupling \texorpdfstring{$\zeta(t)$}{$zeta(t)$} and \texorpdfstring{$\widehat{\zeta}(t)$}{$widehat{zeta}(t)$}}
\label{sec9}

In the previous section, we estimate the difference between the jump
rates of the trace process $\zeta(t)$ and of the asymptotic process~$\widehat{\zeta}(t)$.
We use these estimates here to compare the hitting
times of these processes. The first result of this section shows that
we may couple the processes $\zeta(t)$ and $\widehat{\zeta}(t)$ in such
a way that they stay together for a long amount of time. Consider two
continuous-time Markov processes $X^1$, $X^2$ taking values on the
same countable state space $E$. Denote by $R_1$, $R_2$ the respective
jump rates.

\begin{lemma}
\label{bs01}
Assume that
\[
\sup_{\eta\in E} \sum_{\xi\in E} \bigr|
R_1(\eta,\xi) - R_2(\eta,\xi) \bigr| \le \alpha
\]
for some $\alpha<\infty$. Then, for every $\eta\in E$, there exists a
coupling $\overline{\mb P}$ of the two processes such that $X^1(0) =
X^2(0)=\eta$ $\overline{\mb P}$-a.s., and such that for all $t>0$
\[
\overline{\mb P} [ T_{\mathrm{cp}} \le t ] \le P[\mf e_\alpha\le t ],
\]
where $T_{\mathrm{cp}}$ is the first time the processes separate, $T_{\mathrm{cp}}
= \inf\{ t>0 \dvtx X^1_t \neq X^2_t\}$, and where $\mf e_\alpha$ is a mean
$\alpha^{-1}$ exponential random variable.
\end{lemma}

The proof of this result is elementary and left to the reader. By this
lemma and by Propositions \ref{s08} and \ref{s16}, for each
configuration $\eta\in\Xi$, there exists a coupling of the
processes $\zeta(t)$ and $\widehat{\zeta}(t)$, denoted by $\mb
P^{\zeta, \widehat{\zeta}}_{\eta}$, such that $\mb P^{\zeta, \widehat
\zeta}_{\eta} [\zeta(0) = \widehat{\zeta}(0) =\eta]=1$ and such that
for all $t>0$
%
\begin{equation}
\label{b01} \mb P^{\zeta, \widehat{\zeta}}_{\eta} [ T_{\mathrm{cp}} \le t ]
\le P[\mf e_\beta\le t ],
\end{equation}
where $\mf e_\beta$ is a mean $e^{\beta} \kappa_2^{-1}$ exponential
random variable.

\begin{proposition}
\label{lc01}
Recall from \eqref{c13} the definition of the error $\mb e(\beta)$.
For a subset $A$ of $\Lambda_L$, let $\Gamma_A = \{\eta^{\mb x}\in
\Omega_{L,K} \dvtx \mb x\in A\}$. Assume that $\lim_\beta\kappa_1
=0$. Then
\[
\max_{A\subset\Lambda_L} \bigl| \mb P^{\zeta}_{\eta^{\mb0}} \bigl[
H^+_{\Gamma} = H^+_{\Gamma_A} \bigr] - \mb P^{\widehat{\zeta}}_{\eta^{\mb0}}
\bigl[ H^+_{\Gamma} = H^+_{\Gamma_A} \bigr] \bigr| \le \mb e(\beta).
\]
\end{proposition}

\begin{pf}
Fix a subset $A$ of $\Lambda_L$ and let $\tau_1$ be the time of the
first jump so that
\[
\mb P^{\zeta}_{\eta^{\mb0}} \bigl[ H^+_{\Gamma} =
H^+_{\Gamma_A} \bigr] = \sum_{\eta\in\Xi} \mb
P^{\zeta}_{\eta^{\mb0}} \bigl[ \zeta(\tau_1) = \eta \bigr]
\mb P^{\zeta}_{\xi} [ H_{\Gamma} = H_{\Gamma_A} ].
\]
By Proposition \ref{s16}, this expression is equal to
\[
\sum_{\eta\in\mc V(\eta^{\mb0})} \mb P^{\widehat{\zeta}}_{\eta^{\mb
0}}
\bigl[ \widehat{\zeta}(\tau_1) = \eta \bigr] \mb P^{\zeta}_{\eta}
[ H_{\Gamma} = H_{\Gamma_A} ] + R(\beta),
\]
where $R(\beta)$ is a remainder absolutely bounded by $e^{-2\beta}
\kappa_1 [\lambda_\zeta(\eta^{\mb0})^{-1} + \lambda_{\widehat
\zeta} (\eta^{\mb0})^{-1}]$, and $\lambda_\zeta(\eta^{\mb0})$,
$\lambda_{\widehat{\zeta}} (\eta^{\mb0})$ are the holding rates at
$\eta^{\mb0}$ of the processes $\zeta(t)$, $\widehat{\zeta}(t)$,
respectively. By \eqref{50} and by Proposition \ref{s16},
$\lambda_{\widehat{\zeta}} (\eta^{\mb0}) \ge(1/4) e^{-2\beta}$ and
$\lambda_{\zeta} (\eta^{\mb0}) \ge e^{-2\beta} \{(1/4) -
\kappa_1\}$. This proves that $|R(\beta)| \le\kappa_1$.

We estimate $\mb P^{\zeta}_{\eta} [ H_{\Gamma} = H_{\Gamma_A} ]$ for a
fixed $\eta\in\mc V(\eta^{\mb0})$. Recall the definition of the
coupling $\mb P^{\zeta, \widehat{\zeta}}_\eta$ introduced after Lemma
\ref{bs01}. Under this coupling, on the set $\{H_{\Gamma} \le T_{\mathrm{cp}}\}$ the processes $\zeta(t)$ and $\widehat{\zeta}(t)$ reach
the set $\Gamma$ at the same time and at the same
configuration. Hence,
\[
\mb P^{\zeta}_{\eta} [ H_{\Gamma} = H_{\Gamma_A} ]
= \mb P^{\zeta, \widehat{\zeta}}_{\eta} \bigl[ H^\zeta_{\Gamma}
= H^\zeta _{\Gamma_A} \bigr] \le \mb P^{\widehat{\zeta}}_{\eta}
[ H_{\Gamma} = H_{\Gamma_A} ] + \mb P^{\zeta, \widehat{\zeta}}_{\eta}
\bigl[ H^\zeta_{\Gamma} \ge T_{\mathrm{cp}} \bigr].
\]
In this formula, $H^\zeta$ represents the hitting time for the process
$\zeta(t)$. A similar inequality holds interchanging the roles of
$\zeta$ and $\widehat{\zeta}$. Therefore,
\[
\bigl| \mb P^{\zeta}_{\eta} [ H_{\Gamma} = H_{\Gamma_A}
] - \mb P^{\widehat{\zeta}}_{\eta} [ H_{\Gamma} = H_{\Gamma_A}
] \bigr| \le \mb P^{\zeta, \widehat{\zeta}}_{\eta} \bigl[ H^\zeta_{\Gamma}
\ge T_{\mathrm{cp}} \bigr] + \mb P^{\zeta, \widehat{\zeta}}_{\eta} \bigl[
H^{\widehat{\zeta}}_{\Gamma} \ge T_{\mathrm{cp}} \bigr].
\]
We estimate the first probability on the right-hand side, the
arguments needed for the second one being similar.

Consider a sequence $T_\beta$ to be chosen later. By Lemma \ref{bs01}
and by \eqref{b01},
\[
\mb P^{\zeta, \widehat{\zeta}}_{\eta} \bigl[ H^\zeta_{\Gamma}
\ge T_{\mathrm{cp}} \bigr] \le \mb P^{\zeta}_{\eta} \bigl[
H^\zeta_{\Gamma} \ge T_\beta \bigr] + P[\mf
e_\beta\le T_\beta],
\]
where $\mf e_\beta$ is a mean $e^{\beta} \kappa^{-1}_2$ exponential
random variable. Since $1-e^{-x} \le x$, $x>0$, the second term is
bounded by $e^{-\beta} T_\beta \kappa_2$. On the other hand, by
Chebyshev's inequality and by Lemma \ref{lc03} below,
\[
\mb P^{\zeta, \widehat{\zeta}}_{\eta} \bigl[ H^\zeta_{\Gamma}
\ge T_\beta \bigr] \le \frac{1}{T_\beta} \mb E^{\zeta}_{\eta}
[ H_{\Gamma} ] \le \frac{C_0 |\Xi| e^\beta}{L^2 T_\beta} \cdot
\]
Choosing $T_\beta= e^\beta(|\Xi|/\kappa_2)^{1/2}$ we obtain that
\[
\bigl| \mb P^{\zeta}_{\eta} [ H_{\Gamma} = H_{\Gamma_A}
] - \mb P^{\widehat{\zeta}}_{\eta} [ H_{\Gamma} = H_{\Gamma_A}
] \bigr| \le \sqrt{|\Xi| \kappa_2/ L^2}.
\]
Estimating $\mb P^{\zeta, \widehat{\zeta}}_{\eta} [
H^{\widehat{\zeta}}_{\Gamma} \ge T_{\mathrm{cp}} ]$ in a similar way, we
complete the proof of the proposition.
\end{pf}

\subsection*{The trace process on equivalent classes}
Denote by $\{\tau_x \dvtx x\in\bb Z^2\}$ the group of translations in
$\bb Z^2$: For any configuration $\eta\in\Omega_{L,K}$, $\tau_x \eta$
is the configuration defined by $(\tau_x \eta)(y) = \eta(x+y)$, where
the summation is performed modulo $L$. Two configurations $\eta$,
$\xi$, are said to be equivalent, $\eta\sim\xi$, if $\eta=\tau_x\xi$
for some $x\in\bb Z^2$. Denote by $\widetilde\Xi$ the equivalence
classes of $\Xi$, and by $\bs\eta$, $\bs\xi$ the equivalent
classes, that is, the elements of $\widetilde\Xi$.

Let $\widetilde\Psi\dvtx  \Xi\to\widetilde\Xi$ be the function which
associates to a configuration its equivalence class. Since the
dynamics is translation invariant, for all $\bs\xi\in\widetilde\Xi$
\[
R_\zeta\bigl(\eta, \widetilde\Psi^{-1}(\bs\xi)\bigr) =
R_\zeta\bigl(\eta', \widetilde\Psi^{-1}(\bs
\xi)\bigr)\qquad \mbox{if } \widetilde\Psi(\eta) = \widetilde\Psi \bigl(
\eta'\bigr).
\]
In particular,
\[
\bs\zeta(t) = \widetilde\Psi\bigl(\zeta(t)\bigr)
\]
is a Markov chain which jumps from a class $\bs\eta$ to a class $\bs
\xi$ at rate $\mb R_\zeta(\bs\eta, \bs\xi):= R_\zeta(\eta, \bs\xi)
= \sum_{\xi\in\bs\xi} R_\zeta(\eta, \xi)$ if $\eta$ belongs to the
equivalent class $\bs\eta$. Moreover, the probability measure
$\mu_{\bs\zeta}$ on $\widetilde\Xi$ defined by $\mu_{\bs\zeta} (\bs
\eta) = \mu_{\zeta} (\Psi^{-1}(\bs\eta))$ is reversible for the
Markov chain $\bs\zeta(t)$.

\subsection*{An auxiliary dynamics}

Denote by $\bs\eta^{\mb0}$ the equivalence class of the
configurations in $\Gamma$. We introduce in this subsection a process
$\widetilde{\bs\zeta} (t)$ which behaves as $\bs\zeta(t)$ until
$\bs\zeta(t)$ hits the class $\bs\eta^{\mb0}$. Let $\widetilde{\bs
\zeta} (t)$ be the Markov process on $\widetilde\Xi$ whose jump
rates, denoted by $\mb R_{\widetilde{\bs\zeta}}(\bs\eta,\bs\xi)$,
are defined as follows:
\[
\mb R_{\widetilde{\bs\zeta}}(\bs\eta,\bs\xi) =
\mb R_{\bs\zeta}(\bs\eta,\bs\xi),\qquad
\bs\eta\in\widetilde\Xi\setminus\bigl\{\bs\eta^{\mb0}\bigr\}, \bs\xi\in
\widetilde\Xi.
\]
It is clear that we may couple $\bs\zeta(t)$ and $\widetilde{\bs
\zeta} (t)$ in such a way that starting from a configuration $\bs
\eta\in\widetilde\Xi\setminus\{\bs\eta^{\mb0}\}$, both processes
evolve together until they reach simultaneously $\bs\eta^{\mb
0}$. Let $\mu_{\widetilde{\bs\zeta}}$ be the uniform measure on
$\widetilde\Xi$. Since $\mu_{\bs\zeta}(\bs\eta) = \mu_{\bs
\zeta}(\bs\xi)$ for $\bs\eta, \bs\xi\in\widetilde\Xi\setminus
\{\bs\eta^{\mb0}\}$, and since the process $\bs\zeta(t)$ is
reversible, the measure $\mu_{\widetilde{\bs\zeta}}$ satisfies the
detailed balance conditions on $\widetilde\Xi\setminus\{\bs
\eta^{\mb0}\}$ for the rates $\mb R_{\widetilde{\bs\zeta}}$. We
define $\mb R_{\widetilde{\bs\zeta}}(\bs\eta^{\mb0},\bs\xi)$
$\bs\xi\in\widetilde\Xi\setminus\{\bs\eta^{\mb0}\}$ to fulfill
the detailed balance conditions with respect to the uniform measure
$\mu_{\widetilde{\bs\zeta}}$:
\[
\mb R_{\widetilde{\bs\zeta}}\bigl(\bs\eta^{\mb0},\bs\xi\bigr) = \mb
R_{\widetilde{\bs\zeta}}\bigl(\bs\xi,\bs\eta^{\mb0}\bigr),\qquad \xi\in\widetilde\Xi
\setminus\bigl\{\bs\eta^{\mb0}\bigr\}.
\]

\begin{lemma}
\label{lc03}
There exists a finite constant $C_0$ such that
\[
\max_{\eta\in\mc V(\eta^{\mb0})} \mb E^{\zeta}_{\eta} [
H_{\Gamma} ] \le C_0 \bigl(|\Xi|/L^2\bigr)
e^\beta.
\]
\end{lemma}

\begin{pf}
Fix a configuration $\eta$ in $\mc V(\eta^{\mb0})$ and denote by
$\bs\eta$ the equivalence class of $\eta$. Recall the
definition of the processes $\bs\zeta(t)$, $\widetilde{\bs
\zeta}(t)$ introduced above. Let $\mb P^{\bs
\zeta}_{\bs\eta}$, $\mb P^{\widetilde{\bs\zeta}}_{\bs\eta}$ be the
distributions of the processes $\bs\zeta(t)$, $\widetilde{\bs
\zeta}(t)$, respectively, starting from $\bs\eta$. Expectations
with respect to $\mb P^{\bs\zeta}_{\bs\eta}$, $\mb P^{\widetilde
{\bs\zeta}}_{\bs\eta}$ are represented by $\mb E^{\bs
\zeta}_{\bs\eta}$, $\mb E^{\widetilde{\bs\zeta}}_{\bs\eta}$,
respectively.

Denote by $H_{\bs\eta^{\mb0}}$ the time the processes
$\bs\zeta(t)$, $\widetilde{\bs\zeta}(t)$ hit the equivalence class
$\bs\eta^{\mb0}$. Since $H_{\bs\eta^{\mb0}} = H_\Gamma$
and since the processes $\bs\zeta(t)$, $\widetilde{\bs
\zeta}(t)$ evolve together until the hitting time of $\bs\eta^{\mb0}$,
\[
\mb E^{\zeta}_{\eta} [ H_{\Gamma} ] = \mb
E^{\bs\zeta}_{\bs\eta} [ H_{\bs\eta^{\mb0}} ] = \mb
E^{\widetilde{\bs\zeta}}_{\bs\eta} [ H_{\bs\eta^{\mb0}} ].
\]

Since the chain $\widetilde{\bs\zeta}(t)$ is reversible, by
\cite{bl2}, Proposition 6.10,
\[
\mb E^{\widetilde{\bs\zeta}}_{\bs\eta} [ H_{\bs\eta^{\mb0}} ] \le
\frac{1}{\Cap_{\widetilde{\bs\zeta}} (\bs\eta,\bs\eta^{\mb
0})},
\]
where $\Cap_{\widetilde{\bs\zeta}}$ stands for the capacity
associated to the process $\widetilde{\bs\zeta}$. To fix ideas,
assume that $\eta=\eta^{0,2}_{\mb0}$ and consider the unitary flow
$\Phi$ from $\bs\eta$ to $\bs\eta^{\mb0}$ given by $\Phi(\bs
\eta^{0,2}_{\mb0}, \bs\eta^{0,1}_{\mb0}) = \Phi(\bs\eta^{0,1}_{\mb
0}, \bs\eta^{0,0}_{\mb0}) = \Phi(\bs\eta^{0,0}_{\mb0}, \bs
\eta^{\mb0}) =1$, where $\bs\eta^{i,j}_{\mb0}$ represents the
equivalence class of $\eta^{i,j}_{\mb0}$. By Thomson's principle
\cite{g1}, since $\mb R_{\widetilde{\bs\zeta}}(\bs\eta, \bs\xi) =
\mb R_{\bs\zeta}(\bs\eta,\bs\xi) \ge R_{\zeta}(\eta,\xi)$, $\eta\in
\Xi\setminus\Gamma$, and since $\mu_{\widetilde{\bs\zeta}}$ is the
uniform measure on $\widetilde\Xi$,
\begin{eqnarray*}
\frac{1}{\Cap_{\widetilde{\bs\zeta}}
(\bs\eta,\bs\eta^{\mb0})} & \le &\frac{1}{\mu_{\widetilde
{\bs\zeta}} (\bs\eta) } \biggl\{ \frac{1}{\mb R_{\bs\zeta} (\bs\eta^{0,2}_{\mb0}, \bs\eta^{0,1}_{\mb
0})} +
\frac{1}{\mb R_{\bs\zeta}(\bs\eta^{0,1}_{\mb0}, \bs\eta^{0,0}_{\mb0})} + \frac{1}{\mb R_{\bs\zeta}(\bs\eta^{0,0}_{\mb0}, \bs\eta^{\mb0})} \biggr\}
\\
& \le& |\widetilde\Xi| \biggl\{ \frac{1}{R_\zeta(\eta^{0,2}_{\mb0}, \eta^{0,1}_{\mb0})} + \frac{1}{R_\zeta(\eta^{0,1}_{\mb0}, \eta^{0,0}_{\mb0})} +
\frac{1}{R_\zeta(\eta^{0,0}_{\mb0}, \eta^{\mb0})} \biggr\} \cdot
\end{eqnarray*}
By \eqref{41} and \eqref{42}, the previous sum is bounded by $C_0
|\widetilde\Xi| e^\beta$ for some finite constant $C_0$, which proves
the lemma since $|\widetilde\Xi| = |\Xi|/L^2$.
\end{pf}

\section{Proof of Theorem \texorpdfstring{\protect\ref{s10}}{1.1}}
\label{sec4}

We prove in this section Theorem \ref{s10}. Instead of working in the
torus $\bb T_L$ we will consider the process $X(t)$ as a random walk
on~$\bb Z^2$. Let $\bar X(t)$ the random walk on~$\bb Z^2$ which jumps
from $\mb x\in\bb Z^2$ to $\mb y\in\bb Z^2$ at rate $r_\beta(\mb y -
\mb x)$ and let $\bar Z^\beta(t) = \bar X(t\theta_\beta\ell^2)/\ell$.
To prove Theorem \ref{s10}, it is enough to show that $\bar Z^\beta(t)$
converges in the Skorohod topology to a two-dimensional Brownian motion
as $\beta\to\infty$.

In view of the representation of the process $\bar X(t)$ in terms of
the orthogonal martingales $M^{\mb x}_t$, $\bar Z^\beta(t) = (\bar
Z^\beta_1(t), \bar Z^\beta_2(t))$ is a two-dimensional martingale.
Since $r_\beta(x_1, x_2) = r_\beta(\pm x_1, \pm x_2)$, the
predictable quadratic covariation of $\bar Z^\beta_i(t)$, $\bar
Z^\beta_j(t)$, $1\le i,j\le2$, denoted by $\langle  \bar Z^\beta_i, \bar
Z^\beta_j \rangle_t$, is given by
\[
\bigl\langle \bar Z^\beta_i, \bar Z^\beta_j
\bigr\rangle_t = t \theta_\beta\sum
_{\mb x = (x_1,x_2)
\in\Lambda_L} x_i x_j r_\beta(\mb
x) = t \theta_\beta \delta_{i,j} \sum
_{\mb x = (x_1,x_2)
\in\Lambda_L} x_i^2 r_\beta(\mb x),
\]
where $\Lambda_L=\{-L, \ldots, L\}^2$. Hence, by the definition of
$\theta_\beta$,
%
\begin{equation}
\label{cl01} \bigl\langle \bar Z^\beta_i, \bar
Z^\beta_j \bigr\rangle_t = (1/2) t
\delta_{i,j}.
\end{equation}

Let $\Delta\bar Z^\beta(t) = \bar Z^\beta(t) - \bar Z^\beta(t-)$,
$t\ge0$. Since $\bar Z^\beta(0) = 0$, by the martingale central
limit theorem, \cite{ek1}, Theorem VII.1.4, in view of \eqref{cl01},
to prove that $\bar Z^\beta(t)$ converges to a two-dimensional
Brownian motion with diffusion matrix equal to $ (1/2) t \bb I$,
we have to show that
%
\begin{equation}
\label{13} \lim_{\beta\to\infty} \mb E \Bigl[ \sup
_{s\le t} \bigl\Vert\Delta\bar Z^\beta(s)\bigr\Vert \Bigr] = 0
\end{equation}
for all $t>0$.

In the remaining of this section, we prove \eqref{13}. Fix
$\delta>0$. By definition of $\bar Z^\beta(s)$, $\sup_{s\le t} \Vert
\Delta\bar Z^\beta(s)\Vert= \sup_{s\le t \ell^2 \theta_\beta} \Vert
\Delta\bar X(s)\Vert/\ell$. Since the Poisson processes $N^{\mb
x}_t$ have no common jumps, for all $s\ge0$
\[
\bigl \Vert\Delta\bar X(s)\bigr\Vert =\biggl \Vert\sum_{\mb x\in\Lambda_L} \mb x
\bigl[N^{\mb x}_s - N^{\mb x}_{s-}\bigr]
\biggr\Vert = \sum_{\mb x\in\Lambda_L} \Vert\mb x \Vert
\bigl[N^{\mb x}_s - N^{\mb x}_{s-}\bigr]
\]
and
\[ \sum_{\mb x \in\Lambda_{\delta\ell}} \Vert\mb x \Vert
\bigl[N^{\mb x}_s - N^{\mb x}_{s-}\bigr] \le
2 \delta\ell \sum_{\mb x \in\Lambda_{\delta\ell}} \bigl[N^{\mb x}_s
- N^{\mb x}_{s-}\bigr] \le 2 \delta\ell,
\]
where $\Lambda_k = \{-k, \ldots, k\}^2$. The expectation appearing in
\eqref{13} is thus bounded by
\[
2\delta + \sum_{\mb x\notin\Lambda_{\delta\ell}} \frac{\Vert\mb
x\Vert}{\ell} \mb E
\Bigl[ \sup_{s\le t \ell^2 \theta_\beta} \bigl[N^{\mb x}_s -
N^{\mb x}_{s-}\bigr] \Bigr] = 2\delta + \sum
_{\mb x\notin\Lambda_{\delta\ell}} \frac{\Vert\mb x\Vert}{\ell} \mb P \bigl[ N^{\mb x}_{t \ell^2 \theta_\beta}
\ge1 \bigr].
\]
Since $\mb P [ N^{\mb x}_{t \ell^2 \theta_\beta} \ge1 ] \le\mb E [
N^{\mb x}_{t \ell^2 \theta_\beta} ] = t \ell^2 \theta_\beta
r_\beta(\mb x)$, to prove \eqref{13} it is enough to show that
for all $\delta>0$
%
\begin{equation}
\label{15} \lim_{\beta\to\infty} \ell \theta_\beta \sum
_{\Vert\mb x\Vert >
\delta\ell} \Vert\mb x\Vert r_\beta(\mb x) = 0.
\end{equation}

Replacing the Euclidean norm $\Vert\cdot\Vert$ by the sum norm
$|\cdot|$ and recalling formula~\eqref{38} for $r_\beta(\mb x)$, we
bound the sum appearing in \eqref{15} by
%
\begin{equation}
\label{34} \ell \theta_\beta \sum_{|\mb x| > \delta\ell}
|\mb x| R_\xi\bigl(\eta^{\mb0},\eta^{\mb x}\bigr) =
\ell \theta_\beta \lambda_\zeta\bigl(\eta^{\mb0}\bigr)
\sum_{|\mb x| > \delta\ell} |\mb x| \mb P^\zeta_{\eta^{\mb0}}
\bigl[ H^+_{\Gamma} = H_{\eta^{\mb x}} \bigr].
\end{equation}
By equation (6.9) in \cite{bl2}, $\lambda_\zeta(\eta^{\mb0}) \le
\lambda(\eta^{\mb0}) = 8 e^{-2\beta} (1+ O(n) e^{-\beta})$. The
sum appearing in \eqref{15} is thus bounded by
\[
C_0 \ell \theta_\beta e^{-2\beta} \sum
_{|\mb x| > \delta\ell} |\mb x| \mb P^\zeta_{\eta^{\mb0}} \bigl[
H^+_{\Gamma} = H_{\eta^{\mb x}} \bigr] = C_0 \ell
\theta_\beta e^{-2\beta} \sum_{k= [\delta\ell]+1}^{2L}
k \mb P^\zeta_{\eta^{\mb0}} \bigl[ H^+_{\Gamma} =
H_{\lozenge_k} \bigr],
\]
where $\lozenge_k = \{\eta^{\mb x}\in\Omega_{L,K} \dvtx |\mb x|=k\}$ and
$[a]$ stands for the integer part of $a>0$. Let $\blacklozenge_k =
\{\eta^{\mb x}\in\Omega_{L,K} \dvtx |\mb x|\ge k\}$. With this notation
we may rewrite the previous sum as
\begin{eqnarray*}
&& \sum_{k= [\delta\ell]+1}^{2L} k \bigl\{ \mb
P^\zeta_{\eta^{\mb0}} \bigl[ H^+_{\Gamma} = H_{\blacklozenge_k}
\bigr] - \mb P^\zeta_{\eta^{\mb0}} \bigl[ H^+_{\Gamma
} =
H_{\blacklozenge_{k+1}} \bigr] \bigr\}
\\
&&\qquad \le C_0 \ell \mb P^\zeta_{\eta^{\mb0}} \bigl[
H^+_{\Gamma} = H_{\blacklozenge_{[\delta\ell]}} \bigr] + \sum
_{k= [\delta\ell]+1}^{2L} \mb P^\zeta_{\eta^{\mb0}}
\bigl[ H^+_{\Gamma} = H_{\blacklozenge_{k}} \bigr],
\end{eqnarray*}
where we performed a summation by parts in the last step. By
Proposition \ref{lc01} and Lemma \ref{lc02}, the previous expression is
bounded by
\[
L \kappa_3 + C_0 \gamma^{([\delta\ell] - 4)/3} \biggl\{\ell +
\frac{\gamma^{1/3}}{1-\gamma^{1/3}} \biggr\},
\]
where $\kappa_3 = \kappa_1 + \sqrt{\kappa_2 |\Xi|/L^2}$ and $\gamma$
is given by \eqref{c04} below. By \eqref{c05}, $\gamma\le
e^{-c_0/2n}$. The previous expression is thus bounded by $L
\kappa_3 + (\ell+n) e^{-c_0 \delta\ell/n}$.

Up to this point, we have shown that the sum appearing in \eqref{15} is
bounded by
\[
\ell \theta_\beta e^{-2\beta} \bigl\{ L \kappa_3 + (
\ell+n) e^{-c_0
\delta\ell/n} \bigr\}.
\]
In view of the assumptions of the theorem and of \eqref{13}, this
proves \eqref{15} and concludes the argument. \qed\vspace*{6pt}

We conclude this section with two results used in the proof of Theorem
\ref{s10}. Recall the definition of $\mc G_{\mb x}$ given in
\eqref{45}. Let $\mf E_{\mb x}$, $\mb x\in\bb T_L$, be the event
that the process $\widehat{\zeta}(t)$ leaves the set $\mc G_{\mb x}$
without visiting the configuration $\eta^{\mb x}$,
\[
\mf E_{\mb x} = \bigl\{ H^{\widehat{\zeta}}_{\mc G_{\mb x}^c} <
H^{\widehat{\zeta}}_{\eta^{\mb x}} \bigr\},
\]
where $H^{\widehat{\zeta}}_{\mc C}$, $\mc C\subset\Xi$, stands for the
hitting time of $\mc C$ by $\widehat{\zeta}$:
\[
H^{\widehat{\zeta}}_{\mc C} = \inf\bigl\{t>0 \dvtx \widehat{\zeta}(t) \in
\mc C\bigr\}.
\]
Let
%
\begin{equation}
\label{c03} \varrho:= \max_{\eta\in\mc V(\eta^{\mb0})} \mb P^{\widehat{\zeta}}_{\eta}
[\mf E_{\mb0}].
\end{equation}

For $\mb x\in\bb T_L$, let $\mc L_{\mb x} = \{\zeta^{\mf l, j}_{\mb x}, 0\le j\le3\}$, $\mc S_{\mb x} = \{\zeta^{\mf s, j}_{\mb x}, 0\le
j\le3\}$, and
\[
\mc G = \bigcup_{\mb x\in\bb T_L} \mc G_{\mb x}, \qquad \mc L =
\bigcup_{\mb x\in\bb T_L} \mc L_{\mb x},\qquad  \mc S = \bigcup
_{\mb x\in\bb T_L} \mc S_{\mb x}.
\]
Define
%
\begin{equation}
\label{c04} \gamma:= \varrho + ( 1 - \varrho ) \max_{\xi\in\mc L_0}
\mb P^{\widehat{\zeta}}_{\xi} [H_{\mc L \setminus\mc L_{0}} < H_{\mc G} ].
\end{equation}

\begin{lemma}
\label{lc02}
For every $k\ge5$,
\[
\mb P^{\widehat{\zeta}}_{\eta^{\mb0}} \bigl[ H^+_{\Gamma} =
H_{\blacklozenge_k} \bigr] \le \gamma^{(k-4)/3}.
\]
\end{lemma}

\begin{pf}
As the chain $\widehat{\zeta}(t)$ jumps from $\eta^{\mb0}$ only to
$\mc V(\eta^{\mb0})$, it is enough to show that
\[
\max_{\eta\in\mc V(\eta^{\mb0})} \mb P^{\widehat{\zeta}}_{\eta} [
H_{\Gamma} = H_{\blacklozenge_k} ] \le \gamma^{(k-4)/3}.
\]
Fix $\eta\in\mc V(\eta^{\mb0})$. To prove the lemma, we deplete
$\Xi$ of its inner configurations, and we keep only the extremal
ones. Let
\[
\Xi_{ \mathrm e} = \bigcup_{\mb x\in\bb T_L} \{ \mc
G_{\mb x} \cup \mc L_{\mb x} \cup\mc S_{\mb x} \}.
\]

Denote by $\widehat\chi(t)$ the trace of $\widehat{\zeta}(t)$ on
$\Xi_{ \mathrm e}$. An inspection shows that the process $\widehat\chi
(t)$ may
only jump from a configuration in $\mc G_{\mb x}$ to a configuration
in $\mc G_{\mb x} \cup\mc G_{\mb x \pm e_i} \cup\mc L_{\mb x} \cup
\mc L_{\mb x+e_2} \cup\mc L_{\mb x-e_1} \cup\mc L_{\mb x-e_1+e_2}
\cup\mc S_{\mb x} \cup\mc S_{\mb x+e_1} \cup\mc S_{\mb x-e_2} \cup
\mc S_{\mb x+e_1-e_2}$. Similarly, the process $\widehat\chi(t)$ may
only jump from a configuration in $\mc L_{\mb x}$ to a configuration
in $\mc L_{\mb x} \cup\mc L_{\mb x\pm e_1} \cup\mc L_{\mb x \pm e_2}
\cup\mc G_{\mb x} \cup\mc G_{\mb x + e_1} \cup\mc G_{\mb x - e_2}
\cup\mc G_{\mb x + e_1-e_2}$. Finally, the process $\widehat\chi
(t)$ may only jump from a configuration in $\mc S_{\mb x}$ to a
configuration in $\mc S_{\mb x} \cup\mc S_{\mb x\pm e_1} \cup\mc
S_{\mb x \pm e_2} \cup\mc G_{\mb x} \cup\mc G_{\mb x - e_1} \cup\mc
G_{\mb x + e_2} \cup\mc G_{\mb x - e_1+e_2}$. In all cases, the
distance from the index $\mb x$ of the starting set to the index $\mb
y$ of the ending set is at most $2$. The neighborhood of a
configuration in $\mc L_{\mb x}$ is illustrated in Figure~\ref{fig7}.

\begin{figure}[t]

\includegraphics{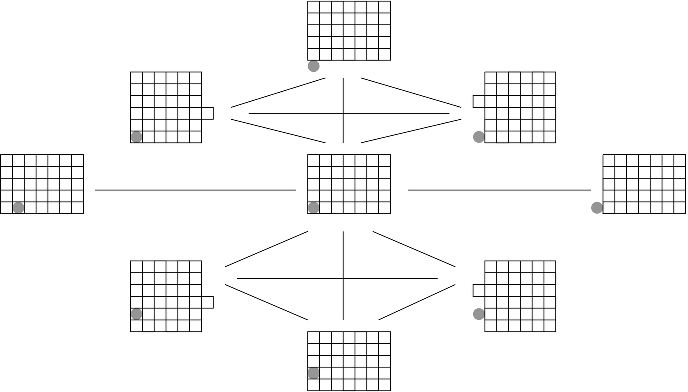}

\caption{The neighborhood of a configuration $\zeta^{\mf l,
\cdot}_{\mb0}$ for $n=6$. We omitted the extra square of the
configurations $\zeta^{\mf l, \cdot}_{\mb x}$. The gray dot
represents the site $\mb0$.}
\label{fig7}
\end{figure}

Recall that $\widehat{\zeta}(0) = \eta^{\mb0}$. Let $\Pi_j =
\bigcup_{\mb x; |\mb x| \ge j} \{\mc G_{\mb x} \cup\mc L_{\mb x} \cup
\mc S_{\mb x}\}$, $j\ge1$, and denote by $\tau_j$ the hitting time of
the set $\Pi_j$ by the process $\widehat{\zeta}(t)$:
\[
\tau_j = \inf\bigl\{ t>0 \dvtx \widehat{\zeta}(t) \in
\Pi_j\bigr\},\qquad j\ge0.
\]
It follows from the conclusions of the previous paragraph that
$\widehat\chi(t)\in\Xi_{ \mathrm e}\setminus\Pi_{j+2}$ if $\widehat\chi
(t-)\in\Xi_{ \mathrm e}\setminus\Pi_j$. In particular, $\tau_j < \tau_{j+2}$
for $j\ge1$.

Let $\mb Y \dvtx \Xi_{ \mathrm e} \to\bb T_L$ be given by
\[
\mb Y (\eta) = \sum_{\mb x\in\bb T_L} \mb x \mb1\{\eta\in\mc
G_{\mb x} \cup\mc L_{\mb x} \cup \mc S_{\mb x}\},
\]
and let $\mb Y_j = \mb Y(\widehat{\zeta}(\tau_j))$. In the formulas
below, $\mc A=\mc L$ if $\mf a =\mf l$ and $\mc A=\mc S$ if $\mf a
=\mf s$. Consider the
events
\begin{eqnarray*}
 \mf E^{1}_j &=& \bigl\{\widehat{\zeta}(
\tau_j) \in\mc G \bigr\} \cap\mf E_{\mb Y_j},
\\
 \mf E^{2, \mf a}_j &=& \bigl\{\widehat{\zeta}(
\tau_j) \in\mc A \bigr\} \cap\{H_{\mc A \setminus\mc A_{\mb Y_j}} \circ
\theta_{\tau_j} < H_{\mc G} \circ\theta_{\tau_j} \},
\\
 \mf E^{3, \mf a}_j &= &\bigl\{\widehat{\zeta}(
\tau_j) \in\mc A \bigr\} \cap\{H_{\mc A \setminus\mc A_{\mb Y_j}} \circ
\theta_{\tau_j} > H_{\mc G} \circ\theta_{\tau_j} \} \cap \mf
E_{\mb Y(\widehat{\zeta}(H_{\mc G} \circ\theta_{\tau_j})) },
\\
\mf E_j &=& \mf E^{1}_j \cup\bigcup
_{\mf a=\mf l, \mf s} \bigl\{ \mf E^{2, \mf a}_j \cup\mf
E^{3, \mf a}_j \bigr\}.
\end{eqnarray*}
Note that $\mf E_j$ is measurable with respect to the $\sigma$-algebra
induced by the stopping time $\tau_{j+3}$ of the Markov process
$\widehat{\zeta}(t)$, which we denote by $\mc
F_{\tau_{j+3}}$. Moreover, $\{H^+_{\Gamma} = H_{\blacklozenge_k}\}
\subset\bigcap_{1\le j\le k-4} \mf E_{j}$. Hence,
\[
\mb P^{\widehat{\zeta}}_{\eta} [ H_{\Gamma} = H_{\blacklozenge_k} ]
\le \mb P^{\widehat{\zeta}}_{\eta} \Biggl[ \bigcap
_{j=1}^{k-4} \mf E_{j} \Biggr].
\]

Fix a positive integer $m\ge1$. Since $\mf E_{3j}$ belongs to $\mc
F_{\tau_{3m}}$ for $j<m$, taking conditional expectation with respect
to $\mc F_{\tau_{3m}}$, by the strong Markov property we obtain that
\[
\mb P^{\widehat{\zeta}}_{\eta} \Biggl[ \bigcap
_{j=1}^{m} \mf E_{3j} \Biggr] = \mb
E^{\widehat{\zeta}}_{\eta} \Biggl[ \mb1 \Biggl\{ \bigcap
_{j=1}^{m-1} \mf E_{3j} \Biggr\}
\gamma_0 \bigl(\widehat{\zeta}(\tau_{3m})\bigr) \Biggr],
\]
where
%
\begin{eqnarray}
\label{48} \gamma_0 (\xi) &=& \mb1\{\xi\in\mc G\} \mb
P^{\widehat{\zeta}}_{\xi} [\mf E_{\mb Y(\xi)}] + \sum
_{\mc A = \mc L, \mc S} \mb1\{\xi\in\mc A\} \mb P^{\widehat{\zeta}}_{\xi}
[H_{\mc A \setminus\mc A_{\mb Y(\xi
)}} < H_{\mc G} ]
\nonumber
\\[-8pt]
\\[-8pt]
\nonumber
&&{} + \sum_{\mc A = \mc L, \mc S} \mb1\{\xi\in\mc A\} \mb
P^{\widehat{\zeta}}_{\xi} [H_{\mc G} < H_{\mc A \setminus\mc A_{\mb Y(\xi)}}, \mf
E_{\mb Y(\widehat{\zeta}(H_{\mc G}))} ].
\end{eqnarray}
By symmetry, the first probability is bounded by $\varrho$, defined
in \eqref{c03}, while by the strong Markov property, the third one is
bounded by
\begin{eqnarray*}
&& \varrho \sum_{\mc A = \mc L, \mc S} \mb1\{\xi\in\mc A\} \mb
P^{\widehat{\zeta}}_{\xi} [H_{\mc G} < H_{\mc A \setminus\mc A_{\mb Y(\xi)}} ]
\\
&&\qquad = \varrho \sum_{\mc A = \mc L, \mc S} \mb1\{\xi\in\mc A\} \bigl\{
1 - \mb P^{\widehat{\zeta}}_{\xi} [H_{\mc A \setminus\mc A_{\mb Y(\xi)}} < H_{\mc G} ]
\bigr\}.
\end{eqnarray*}
Summing this expression with the second one on the right-hand side of
\eqref{48}, we obtain by symmetry that
\[
\max_{\xi\in\Xi} \gamma_0 (\xi) \le \varrho + ( 1 -
\varrho ) \max_{\xi\in\mc L_0} \mb P^{\widehat{\zeta}}_{\xi}
[H_{\mc L \setminus\mc L_{0}} < H_{\mc G} ] = \gamma.
\]
Iterating this argument $(k-4)/3$ times, we get that
\[
\mb P^{\widehat{\zeta}}_{\eta} [ H_{\Gamma} = H_{\blacklozenge_k} ]
\le \gamma^{(k-4)/3},
\]
which proves the lemma.
\end{pf}

\begin{lemma}
\label{s14}
There exists a positive constant $c_0>0$ such that
\[
\max_{\eta\in\mc L} \mb P^{\widehat{\zeta}}_{\eta}
[H_{\mc L \setminus\mc L_{\mb Y(\eta)}} < H_{\mc G} ] \le 1 - \frac{c_0}n \cdot
\]
\end{lemma}

\begin{pf}
By definition of the set $\mc L_{\mb0}$,
\[
\max_{\eta\in\mc L_{\mb0}} \mb P^{\widehat{\zeta}}_{\eta}
[H_{\mc L \setminus\mc L_{\mb0}} < H_{\mc G} ] = \max_{0\le j\le3} \mb
P^{\widehat{\zeta}}_{\zeta^{\mf l,j}_{\mb0}} [H_{\mc L \setminus\mc L_{\mb0}} < H_{\mc G} ].
\]
Consider the case $j=2$ which, by symmetry, is equivalent to the case
$j=0$. The cases $j=1,3$ are examined after. Let $\eta=\zeta^{\mf
l,2}_{\mb0}$, $\mc L^c_{\mb0} = \mc L \setminus\mc L_{\mb
0}$. Intersect the event $\{H_{\mc L^c_{\mb0}} < H_{\mc G}\}$ with
the partition $\{H^+_{\eta} < H_{\mc L^c_{\mb0} \cup\mc G}\}$,
$\{H_{\mc L^c_{\mb0} \cup\mc G} < H^+_{\eta}\}$, and apply the
strong Markov property to obtain that
\[
\mb P^{\widehat{\zeta}}_{\eta} [H_{\mc L^c_{\mb0}} < H_{\mc G} ] =
\frac{\mb P^{\widehat{\zeta}}_{\eta}
 [H_{\mc L^c_{\mb0}} < H^+_{\mc G \cup\{\eta\}}  ]}{
\mb P^{\widehat{\zeta}}_{\eta}
 [H_{\mc L^c_{\mb0} \cup\mc G} < H^+_{\eta}  ]} \cdot
\]
Since $\mb P^{\widehat{\zeta}}_{\eta} [H_{\mc L^c_{\mb0} \cup\mc G} <
H^+_{\eta}] = \mb P^{\widehat{\zeta}}_{\eta} [H_{\mc L^c_{\mb0} } <
H^+_{\mc G\cup\{\eta\}} ] + \mb P^{\widehat{\zeta}}_{\eta} [H_{\mc G}
< H^+_{\mc L^c_{\mb0} \cup\{\eta\}} ]$, we obtain that
%
\begin{equation}
\label{49} \mb P^{\widehat{\zeta}}_{\eta} [H_{\mc L^c_{\mb0}} <
H_{\mc G} ] = \frac{\alpha}{1+\alpha} \qquad\mbox{where }  \alpha =
\frac{\mb P^{\widehat{\zeta}}_{\eta} [H_{\mc L^c_{\mb0}} <
H^+_{\mc G \cup\{\eta\}}]}{
\mb P^{\widehat{\zeta}}_{\eta} [H_{\mc G} < H^+_{\mc L^c_{\mb
0} \cup\{\eta\}} ]} \cdot
\end{equation}
It remains to bound $\alpha$.

For this purpose, we introduce a new process which corresponds to
reflect the process $\widehat{\zeta}(t)$ at some configurations. Let
$\mc N$ be the set of neighbors in $\Xi_{ \mathrm e}$ of $\mc L_{\mb0}$. As
illustrated in Figure~\ref{fig7},
\[
\mc N = \mc L_{\pm e_i} \cup\mc G_{\mb0} \cup\mc G_{e_1}
\cup\mc G_{- e_2} \cup\mc G_{e_1-e_2}.
\]
Let $\mc W$ be the set of configuration which can be reached by
$\widehat{\zeta}(t)$ before hitting $\mc N$:
\[
\mc W = \Bigl\{ \xi\in\Xi: \max_{\zeta\in\mc L_{\mb0}} \mb
P^{\widehat{\zeta}}_\zeta[H_\xi\le H_{\mc N}]>0\Bigr
\}.
\]
Note that only the four configuration $\eta^{i,1}_{\mb0}$, $0\le i\le
3$, of $\mc G_{\mb0}$ belong to $\mc N$. Denote by $\check\zeta(t)$
the process $\widehat{\zeta}(t)$ reflected at $\mc W$. The jump rates of
$\check\zeta(t)$, denoted by $R_{\check\zeta}(\xi,\zeta)$, are given
by
\[
R_{\check\zeta}(\xi,\zeta) = \cases{ e^{-\beta} \mb R(\xi,\zeta), &\quad$
\mbox{if } \xi, \zeta\in\mc W $, \vspace*{2pt}
\cr
0, & \quad $\mbox{otherwise.}$ }
\]

For $\xi,\zeta\in\mc W$, $R_{\check\zeta}(\xi,\zeta) = e^{-\beta}
\mb R(\xi,\zeta) = e^{-\beta} \mb R(\zeta, \xi) = e^{-\beta} R_{\check
\zeta}(\zeta, \xi)$. Therefore, the process $\check\zeta(t)$ is
reversible with respect to the uniform measure on $\mc W$, denoted by
$\mu_{\check\zeta}$. Moreover, since the rates coincide on $\mc W$,
we may couple $\widehat{\zeta}(t)$ and $\check\zeta(t)$ in such a way
that they evolve together until they reach $\mc N$. In particular,
\[
\alpha = \frac{\mb P^{\widehat{\zeta}}_{\eta} [H_{\mc L^c_{\mb0}} <
H^+_{\mc G \cup\{\eta\}}]}{
\mb P^{\widehat{\zeta}}_{\eta} [H_{\mc G} < H^+_{\mc L^c_{\mb
0} \cup\{\eta\}} ]} = \frac{\mb P^{\check\zeta}_{\eta} [H_{\mc L^c_{\mb0}} <
H^+_{\mc G \cup\{\eta\}}]}{
\mb P^{\check\zeta}_{\eta} [H_{\mc G} < H^+_{\mc L^c_{\mb
0} \cup\{\eta\}} ]},
\]
where $\mb P^{\check\zeta}_{\eta}$ stands for the distribution of
the process $\check\zeta(t)$ starting from $\eta$.

Recall that $\eta=\zeta^{\mf l,2}_{\mb0}$. By definition of the
capacity,
%
\begin{equation}
\label{27} \mu_{\check\zeta}(\eta) \lambda_{\check\zeta}(\eta) \mb
P^{\check\zeta}_{\eta} \bigl[H_{\mc L^c_{\mb0}} < H^+_{\mc G \cup\{\eta\}}
\bigr] \le \Cap_{\check
\zeta} \bigl(\mc L^c_{\mb0},\mc G
\cup\{\eta\}\bigr).
\end{equation}

Let $\mc G_+ = \mc G_{\mb x} \cup\mc G_{\mb x+e_1}$ and $\mc G_- =
\mc G_{\mb x-e_2} \cup\mc G_{\mb x+e_1-e_2}$. For the denominator, we
claim that
%
\begin{equation}
\label{28} 2 \mu_{\check\zeta}(\eta) \lambda_{\check\zeta}(\eta) \mb
P^{\check\zeta}_{\eta} \bigl[H_{\mc G} < H^+_{\mc L^c_{\mb
0} \cup\{\eta\}}
\bigr] \ge \Cap_{\check\zeta} \bigl(\mc G_+, \bigl\{\zeta^{\mf l, 0}_{e_2},
\eta\bigr\} \bigr).
\end{equation}
Indeed, from $\eta$ the process may only jump to a configuration in
$\mc L_0$ or to a configuration in $\Omega^2$. If it jumps to a
configuration in $\mc L_0$, since the process is reflected at~$\mc W$,
to reach $\mc G_+$ the process $\check\zeta(t)$ necessarily passes
by $\eta$. Therefore,
\[
\mb P^{\check\zeta}_{\eta} \bigl[ H_{\mc G} <
H^+_{\mc L^c_{\mb0}
\cup\{\eta\}} \bigr] \ge \mb P^{\check\zeta}_{\eta} \bigl[
H_{\mc G_+} < H^+_{\{\zeta^{\mf
l, 0}_{e_2}, \eta\}} \bigr].
\]
By symmetry, $\mb P^{\check\zeta}_{\eta} [ H_{\mc G_+} <
H^+_{\{\zeta^{\mf l, 0}_{e_2}, \eta\}}] = \mb P^{\check
\zeta}_{\zeta^{\mf l, 0}_{e_2}} [ H_{\mc G_+} < H^+_{\{\zeta^{\mf l,
0}_{e_2}, \eta\}} ]$. Since $\mu_{\check\zeta}$ is the uniform
measure and since $\lambda_{\check\zeta}(\eta) \ge e^{-\beta} \mb
R(\eta, \Omega^2) = e^{-\beta} \mb R(\zeta^{\mf l, 0}_{e_2}, \Omega^2)
= \lambda_{\check\zeta} (\zeta^{\mf l, 0}_{e_2})$, we obtain
\eqref{28}.

By \eqref{49}, \eqref{27}, \eqref{28} and Lemma \ref{s15} below,
%
\begin{equation}
\label{lc15} \mb P^{\widehat{\zeta}}_{\zeta^{\mf l,2}_{\mb0}} [H_{\mc L \setminus\mc L_{\mb0}} <
H_{\mc G} ] \le 1 - \frac{c_0}n \cdot
\end{equation}
It remains to consider the case $j=3$ which is equal to the case
$j=1$.

Let $a$ be the probability that the chain $\widehat{\zeta}(t)$ jumps
from $\zeta^{\mf l, 3}_{\mb0}$ to $\zeta^{\mf l, 2}_{\mb0}$. With
the notation introduced in Lemma \ref{ls02},
\[
a = \mb P^{\widehat{\zeta}}_{\zeta^{\mf l, 3}_{\mb0}} [ H_{\zeta^{\mf l, 2}_{\mb0}} =
H_{\mc V(\zeta^{\mf l, 3}_{\mb0})}].
\]
By the strong Markov property and by symmetry,
\[
\mb P^{\widehat{\zeta}}_{\zeta^{\mf l,3}_{\mb0}} [H_{\mc L \setminus\mc L_{\mb0}} < H_{\mc G} ]
\le (1-2a) + 2a \mb P^{\widehat{\zeta}}_{\zeta^{\mf l,2}_{\mb0}} [H_{\mc L \setminus\mc L_{\mb0}} <
H_{\mc G} ].
\]
By \eqref{lc15} and by Lemma \ref{ls02}, the previous expression is
less than or equal to $1-2ac_0/n = 1-c'_0/n$. This completes the proof
of the lemma.
\end{pf}

We are now in a position to obtain a bound for $\gamma$ introduced in
\eqref{c04}. By Lemma~\ref{s17}, $\varrho\le1/2$ for $n\ge
46$. Hence, by Lemma \ref{s14},
%
\begin{equation}
\label{c05} \gamma \le 1 - \frac{c_0}{2n} \le e^{-c_0 /2n}
\end{equation}
for some positive constant $c_0$.

The proof of Lemma \ref{s11} below provides bounds for the
capacity associated to the process $\check\zeta(t)$. Let
$U=\{\zeta^{\mf l, j}_{\mb0} \dvtx 0\le j\le3\} \cup\mc G_{\mb0} \cup
\mc G_{-e_1}\cup\mc G_{e_2} \cup\mc G_{e_2-e_1} \cup\{\zeta^{\mf l,
0}_{e_2}, \zeta^{\mf l, 1}_{-e_1}, \zeta^{\mf l, 2}_{-e_2}, \zeta^{\mf l, 3}_{e_1}\}$.

\begin{lemma}
\label{s15}
There exist constants $0<c_0<C_0<\infty$, independent of
$\beta$, such that for any disjoint subsets $A$, $B$ of $U$,
\[
\frac{c_0}{n^2} \frac{e^{-\beta}}{
|\mc W|} \le \Cap_{\check\zeta} (A,B) \le
\frac{C_0}{n} \frac{e^{- \beta}} {|\mc W|} \cdot
\]
\end{lemma}

In this lemma, by a subset $A$ of $U$, we understand the union of the
sets forming~$U$. This means that either all configurations of $\mc
G_{\mb x}$ belongs to $A$ or none.

\section{Proof of Theorem \texorpdfstring{\protect\ref{s18}}{1.2}}
\label{sec7}


Recall that we denote by $\Xi_\star$ be the set of configurations which
can be reached from $\eta^{\mb0}$ without crossing $\Delta_2$, the
set of configurations with energy greater than $\bb H_{\mathrm{min}} +2$,
defined in \eqref{04}. We first claim that
%
\begin{equation}
\label{36} \lim_{\beta\to\infty} \mb E_{\eta^{\mb0}} \biggl[ \int
_0^t \mb1\bigl\{\eta\bigl(s \ell^2
\theta_\beta\bigr) \notin\Xi_\star\bigr\} \,ds \biggr] = 0.
\end{equation}

Recall the definition of the set $\mc N(\eta^{\mb0})$ introduced in the
proof of Lemma \ref{s13} and define $\mc N(\eta^{\mb x})$, $\mb
x\in\bb T_L$, as the set of configurations in $\mc N(\eta^{\mb0})$
translated by~$\mb x$. Denote by $N_{\mb x}(t)$ the number of jumps
from the configuration $\eta^{\mb x}$ to the set of configurations
$\Delta_2$ in the time interval $[0,t]$. This corresponds to the number
of jumps from $\eta^{\mb x}$ to $\Omega_{L,K} \setminus\mc N(\eta^{\mb
x})$ in the time interval $[0,t]$. It is clear that $N_{\mb x}(t)$,
$\mb x\in\bb T_L$, are independent Poisson processes of intensity
$4(n-2) e^{-3\beta} \mb1\{\eta= \eta^{\mb x}\}$. In particular, if
$N_0(t) = \sum_{\mb x \in\bb T_L} N_{\mb x}(t) $,
\[
\mb P_{\eta^{\mb0}} \bigl[ N_0(t) \ge1 \bigr] \le \mb
E_{\eta^{\mb0}} \bigl[ N_0(t) \bigr] = 4(n-2) e^{-3\beta} \mb
E_{\eta^{\mb0}} \biggl[ \int_0^{t} \mb1\bigl
\{\eta(s) \in\Gamma\bigr\} \,ds \biggr].
\]
By symmetry,
%
\begin{eqnarray}
\label{37} && \mb E_{\eta^{\mb0}} \biggl[ \int_0^{\ell^2 \theta_\beta}
\mb1\bigl\{\eta (s) \in \Gamma\bigr\} \,ds \biggr]\nonumber\\
&&\qquad = \frac{1}{|\bb T_L|} \sum
_{\mb y\in\bb T_L} \mb E_{\eta^{\mb y}} \biggl[ \int
_0^{\ell^2 \theta_\beta} \mb1\bigl\{\eta(s) \in\Gamma\bigr\} \,ds
\biggr]
\\
&&\qquad = \frac{1}{|\bb T_L| \mu_K(\eta^{\mb w})} \sum_{\mb y\in\bb T_L}
\mu_K\bigl(\eta^{\mb y}\bigr) \mb E_{\eta^{\mb y}} \biggl[
\int_0^{\ell^2 \theta_\beta} \mb1\bigl\{\eta(s) \in\Gamma\bigr\}
\,ds \biggr].\nonumber
\end{eqnarray}
Clearly, the sum is bounded above by
\[
\mb E_{\mu_K} \biggl[ \int_0^{\ell^2 \theta_\beta} \mb1
\bigl\{\eta(s) \in\Gamma\bigr\} \,ds \biggr] = \ell^2
\theta_\beta \mu_K(\Gamma),
\]
while the numerator is equal to $\mu_K(\Gamma)$ so that
\[
\mb P_{\eta^{\mb0}} \bigl[ N_0\bigl(\ell^2
\theta_\beta\bigr) \ge1 \bigr] \le C_0 n \ell^2
\theta_\beta e^{-3\beta}.
\]
By assumption \eqref{40}, this expression vanishes as
$\beta\uparrow\infty$. Hence, with a probability converging to $1$,
the process $\eta(t)$ does not hit $\Delta_2$ in the time interval
$[0,\ell^2 \theta_\beta]$ from a configuration in $\Gamma$. Moreover,
on the event $N_0(\ell^2 \theta_\beta) = 0$, in the time interval
$[0,\ell^2 \theta_\beta]$ the process $\eta(t)$ may only jump from a
configuration in $\Gamma$ to a configuration in $\bigcup_{\mb x\in\bb
T_L} \mc N(\eta^{\mb x})$.

Denote by $\mc P_1(\mb x)$, $\mc P$ for plateau, the set of
configurations in $\bb H_2$ which can be reached from configurations
in $\mc N(\eta^{\mb x})$ by rate one jumps. These are the configurations
in $\bb H_2$ which appeared in the proof of Lemma \ref{s13} and which
have been denoted by $\Sigma_{0,1}$ in Figure~\ref{fig1a}.

For each configuration $\eta\in\mc P_1(\mb x)$, denote by $N_1 (\eta,
t)$ the number of jumps from $\eta$ to a configuration in $\Delta_2$
in the time interval $[0,t]$. It is clear that $N_1(\eta, t)$,
$\eta\in\mc P_1(\mb x)$, are orthogonal Poisson processes of
intensity bounded by $C_0 e^{-\beta} \mb1\{\eta(t) = \eta\}$.
Hence, $N_1(t) = \sum_{\eta\in\mc P_1} N_1(\eta, t)$, $\mc P_1 =
\bigcup_{\mb x\in\bb T_L} \mc P_1(\mb x)$, is a Poisson process of
intensity bounded by $C_0 e^{-\beta} \mb1\{\eta(t) \in\mc P_1\}$
and
\[
\mb P_{\eta^{\mb0}} \bigl[ N_1(t) \ge1 \bigr] \le \mb
E_{\eta^{\mb0}} \bigl[ N_1 (t) \bigr] \le C_0
e^{-\beta} \mb E_{\eta^{\mb0}} \biggl[ \int_0^{t}
\mb1\bigl\{\eta(s) \in\mc P_1\bigr\} \,ds \biggr].
\]

Repeating the arguments presented in \eqref{37}, we show that
\[
\mb P_{\eta^{\mb0}} \bigl[ N_1\bigl(\ell^2
\theta_\beta\bigr) \ge1 \bigr] \le \frac{C_0 \ell^2 \theta_\beta e^{-\beta} \mu_K(\mc P_1)}{
\mu_K(\Gamma)} \le
C_0 (\ell/L)^2 \theta_\beta e^{-3\beta} |
\mc P_1|.
\]
Since $|\mc P_1|$ is bounded by $C_0 L^2 (n^2 + L^2)$, since $n\le
L$, the previous expression is less than or equal to $C_0 \ell^2
L^2 \theta_\beta e^{-3\beta}$, which vanishes as
$\beta\uparrow\infty$ in view of \eqref{40}. Therefore, with a
probability converging to $1$, the process $\eta(t)$ does not hit
$\Delta_2$ in the time interval $[0, \ell^2 \theta_\beta]$ from a
configuration in $\mc P_1$. Moreover, on the event $N_1(\ell^2
\theta_\beta)=0$, in the time interval $[0,\ell^2 \theta_\beta]$ the
process $\eta(t)$ may only leave the set $\mc P_1$ to a configuration
in $\Gamma$ or to a configuration in $\mc E^{i,j}_{\mb x}$, $\mb
x\in\bb T_L$, $0\le i,j\le3$.

At this point, we repeat the reasoning developed above for the
configurations in $\Gamma$ to the configurations in $\mc E^{i,j}_{\mb
x}$. Proceeding in this manner, we complete the proof of
\eqref{36}. The main contribution, which explains the need of
assumption \eqref{40}, comes from the subsets of configurations which
are crossed in a transition from a configuration in $\Omega^2$
(resp., $\Omega^4$) to another configuration in $\Omega^2$
(resp., $\Omega^4$). There are $C_0 L^2 n^8$ such configurations.
The details are left to the reader. 

We conclude this section with a similar estimate needed in the proof
of Theorem \ref{sc01}. Under the assumptions of Theorem \ref{s18} and
the hypothesis that $n^7 e^{-\beta} + L^2 e^{-2\beta}\to0$, we claim
that
%
\begin{equation}
\label{c14} \lim_{\beta\to\infty} \mb E_{\eta^{\mb0}} \biggl[ \int
_0^t \mb1\bigl\{\eta\bigl(s \ell^2
\theta_\beta\bigr) \notin \Gamma\bigr\} \,ds \biggr] = 0.
\end{equation}
In view of Theorem \ref{s18}, we may replace $\Gamma^c$ by $\Xi_\star
\cap\bb H_{12}$ in the previous formula. Repeating the arguments
which led to \eqref{37}, we obtain that the previous expectation is
bounded by
\[
\frac{t \mu_K(\Xi_\star\cap\bb H_{12})}{
\mu_K(\Gamma)} = \frac{t e^{-2\beta} |\Xi_\star\cap\bb H_2|}{L^2} + \frac{t e^{-\beta} |\Xi_\star\cap\bb H_1|}{L^2} \cdot
\]
The main components of $\Xi_\star\cap\bb H_1$ are the subsets
$\Omega^2$ and $\Omega^4$ whose cardinality are bounded by $C_0
L^2 n^7$. To estimate the first term in the previous formula, one has
to go through the proofs of Section~\ref{sec6} and recollect all
configurations of $\Xi_\star\cap\bb H_2$. An inspection shows that
$|\Xi_\star\cap\bb H_2| \le C_0 L^2 (L^2+n^8)$. This completes
the proof of \eqref{c14}.

\section{Proof of Theorem \texorpdfstring{\protect\ref{sc01}}{1.3}}
\label{sec08}

Recall that we denote by $\zeta(t)$ the trace of $\eta(t)$ on $\Xi$
and that the center of mass, $\mt M(\eta)$ of a configuration $\eta$
in $\Xi$ is well defined.

We first claim that the process $\hat{\mt M}_\beta(t) = \mt M(\zeta(t
\ell^2 \theta_\beta))/\ell$ converges to a Brownian motion in the
Skorohod topology. The proof of this assertion is divided in two
steps. We first prove the tightness of the sequence and then we
characterize the limit points.

Assume that $\zeta(0) = \eta^{\mb0}$. Let $S_1=0$ and let $T_1$ be
the time of the first jump of $\zeta(t)$. Define inductively $S_{j+1}
= H_\Gamma\circ\vartheta_{T_j} + T_j$, $T_{j+1} = T_1 \circ
\vartheta_{S_j} + S_j$, $j\ge1$. Thus, $[S_j,T_j)$ represent the
successive sojourn times at $\Gamma$. Fix $t>0$ and let $\mf n =
\min\{k\ge1\dvtx T_k \ge t \theta_\beta\ell^2\}$.

Recall that $\xi(t)$ is the trace of $\zeta(t)$ on $\Gamma$, and fix
$\delta>0$, $\varepsilon>0$. By the observation of the previous
paragraph, if
\begin{eqnarray*}
\mathop{\sup_{0\le s\le t \theta_\beta\ell^2}}_{ 0\le r\le\delta
\theta_\beta\ell^2} \bigl\Vert\mt M\bigl(
\xi(s+r)\bigr) - {\mt M}\bigl(\xi(r)\bigr) \bigr\Vert &\le& \varepsilon \ell \quad\mbox{and}
\\
\sup_{T_j \le s \le S_{j+1}} \bigl\Vert\mt M\bigl(\zeta(s)\bigr) - {\mt M}\bigl(
\zeta (T_j)\bigr) \bigr\Vert &\le& \varepsilon \ell\qquad \mbox{for all $j\le\mf
n$},
\end{eqnarray*}
then, the inequality written in the first line of the above
formula holds with $\xi$, $\varepsilon$ replaces by $\zeta$, $3\varepsilon$,
respectively. In particular, since we have shown in the proof of
Theorem \ref{s10} that the chain $\tilde{\mt M}_\beta(t) = \mt
M(\xi(t \ell^2 \theta_\beta))/\ell$ is tight with respect to the
uniform modulus of continuity, to prove the tightness of $\hat{\mt
M}_\beta(t)$ it remains to show that
\[
\lim_{\delta\to0} \limsup_{\beta\to\infty} \mb
P^\zeta_{\eta^{\mb0}} \Biggl[ \bigcup_{j=1}^{\mf n}
\sup_{T_j \le s \le S_{j+1}}\bigl \Vert\mt M\bigl(\zeta(s)\bigr) - {\mt M}\bigl(
\zeta(T_j)\bigr) \bigr\Vert \ge \varepsilon \ell \Biggr] = 0.
\]

Let $\mf e_j = T_j - S_j$, $j\ge1$, so that
$\{\lambda_\zeta(\eta^{\mb0}) \mf e_j \dvtx j\ge1\}$ is a sequence of
mean $1$ i.i.d. exponential random variables. Clearly, $\sum_{1\le
i\le j} \mf e_i \le t \theta_\beta\ell^2$ if $\mf n > j$. Since by
equation (6.9) in \cite{bl2}, $\lambda_\zeta(\eta^{\mb0}) \le
\lambda(\eta^{\mb0}) = 8 e^{-2\beta} (1+ O(n) e^{-\beta}) \le9
e^{-2\beta}$,
\begin{eqnarray*}
\mb P^\zeta_{\eta^{\mb0}} [\mf n > j] &\le& \mb P^\zeta_{\eta^{\mb0}}
\Biggl[\sum_{i=1}^j \lambda_\zeta
\bigl(\eta^{\mb
0}\bigr) \mf e_i \le9 e^{-2\beta} t
\theta_\beta\ell^2 \Biggr] \\
&\le& \frac{9 e^{-2\beta} t \theta_\beta\ell^2}{j}
\cdot
\end{eqnarray*}
Therefore, to prove tightness of the chain $\hat{\mt M}_\beta(t)$ with
respect to the uniform modulus of continuity it remains to show that
\[
\lim_{\delta\to0} \limsup_{\beta\to\infty} e^{-2\beta}
\theta_\beta\ell^2 \mb P^\zeta _{\eta^{\mb0}}
\Bigl[ \sup_{0 \le s \le H_\Gamma} \bigl\Vert\mt M\bigl(\zeta(s)\bigr) \bigr\Vert \ge
\varepsilon \ell \Bigr] = 0.
\]
This follows from the proof of Theorem \ref{s10}.

Let $\tilde{\mt M}_\beta(t) = \mt M(\xi(t \ell^2
\theta_\beta))/\ell$, which converges to a Brownian motion by Theorem~\ref{s10}. We claim that for every $t>0$, $\varepsilon>0$,
%
\begin{equation}
\label{n01} \lim_{\beta\to\infty} \mb P^\zeta_{\eta^{\mb0}}
\Bigl[ \sup_{0\le s\le t} \bigl\Vert\hat{\mt M}_\beta(s) -
\tilde{\mt M}_\beta(s)\bigr \Vert >\varepsilon \Bigr] = 0.
\end{equation}
By definition of $\hat{\mt M}_\beta(t)$, $\tilde{\mt M}_\beta(t)$, it
is enough to show that
%
\begin{equation}
\label{n03} \lim_{\beta\to\infty} \mb P^\zeta_{\eta^{\mb0}}
\Bigl[ \sup_{0\le s\le t \ell^2 \theta_\beta}\bigl
 \Vert\mt M \bigl(\zeta(s)\bigr) - \mt M
\bigl(\xi(s)\bigr) \bigr\Vert >\varepsilon\ell \Bigr] = 0.
\end{equation}

Recall that $\xi(s)$ is the trace of $\zeta(s)$ on $\Gamma$. Hence,
if we define the additive functional $T(s)$ by
\[
T(s) = \int_0^s \mb1\bigl\{\zeta(r) \in
\Gamma\bigr\} \,dr, \qquad s\ge0,
\]
and if $S(s)$ is the generalized inverse of $T$, $S(s) = \sup\{r\ge0
| T(r) \le s\}$, $\xi(s) = \zeta(S(s))$. We may therefore replace
$\mt M (\zeta(s)) - \mt M (\xi(s))$ by $\mt M (\zeta(s)) - \mt M
(\zeta(S(s)))$ in the previous formula.

We claim that for every $t>0$, $\delta>0$,
%
\begin{equation}
\label{n02} \lim_{\beta\to\infty} \mb P^\zeta_{\eta^{\mb0}}
\bigl[ S\bigl(t \theta_\beta \ell^2\bigr) - t
\ell^2 \theta_\beta \ge\delta \ell^2
\theta_\beta \bigr] = 0.
\end{equation}
By definition of $S$, for every $\delta'>0$ $t'>0$, $S(t') <t' +
\delta'$ if $T(t' + \delta')>t'$. Hence, taking $\delta'= \delta
\ell^2 \theta_\beta$ and $t' = t \ell^2 \theta_\beta$, since
$t'-T(t') = \int_{[0,t']} \mb1\{\zeta(s) \notin\Gamma\} \,ds$, we
conclude that the previous probability is less than or equal to
\[
\mb P^\zeta_{\eta^{\mb0}} \biggl[ \int_{0}^{(t+\delta) \theta_\beta
\ell^2}
\mb1\bigl\{\zeta(s) \notin\Gamma\bigr\} \,ds \ge\delta\ell^2
\theta_\beta \biggr].
\]
By Chebyshev's inequality and by translation invariance and by the
arguments used in \eqref{37}, this probability is bounded by
\begin{eqnarray*}
&& \frac{1}{\delta\ell^2 \theta_\beta} \mb E^\zeta_{\eta^{\mb0}} \biggl[ \int
_{0}^{(t+\delta) \theta_\beta
\ell^2} \mb1\bigl\{\zeta(s) \notin\Gamma\bigr\}
\,ds \biggr]
\\
& &\qquad\le \frac{1}{\delta\ell^2 \theta_\beta\mu_\zeta(\Gamma)} \mb E^\zeta_{\mu
_\zeta} \biggl[ \int
_{0}^{(t+\delta) \theta_\beta
\ell^2} \mb1\bigl\{\zeta(s) \notin\Gamma\bigr\}
\,ds \biggr] = \frac{(t+\delta) \mu_\zeta(\Gamma^c)}{\delta\mu_\zeta(\Gamma)} \cdot
\end{eqnarray*}
Since $\mu_\zeta(\Gamma^c)/ \mu_\zeta(\Gamma)$ is bounded by
$C_0n^7e^{-\beta}$, the previous expression vanishes as
$\beta\to\infty$, proving \eqref{n02}.

Recall that we replaced $\mt M (\xi(s))$ by $\mt M (\zeta(S(s)))$ in
\eqref{n03}. Since $S(r)-r$ is a nonnegative increasing function, in
view of \eqref{n02}, to prove \eqref{n01} it is enough to show that
for every $\varepsilon>0$,
\[
\lim_{\delta\to0} \limsup_{\beta\to\infty} \mb
P^\zeta_{\eta^{\mb0}} \Bigl[ \mathop{\sup_{0\le s\le t \ell^2 \theta_\beta}}_{
0\le r \le\delta\ell^2 \theta_\beta}
\bigl\Vert\mt M \bigl(\zeta(s)\bigr) - \mt M \bigl(\zeta(s+r)\bigr) \bigr\Vert >\varepsilon
\ell \Bigr] = 0.
\]
This follows from the tightness of the process $\mt M (\zeta(s))$.

To complete the proof of the theorem, we need to replace $\mt M
(\zeta(s))$ by $\mt M (\eta(s))$. This is simpler and follows the same
strategy of the first part of the proof. Denote by $\bb A$ the event
$\{\eta(s)\notin\Xi_\star\mbox{ for some } 0\le s\le t \ell^2
\theta_\beta\}$. By Theorem \ref{s18}, this event has a vanishing
asymptotic probability. In view of our convention for the center of
mass of configurations in $\Sigma_{0,1} \cup\Sigma_{2,3}$, on $\bb A$
the center of mass $\mt M(\eta(s))$ does not change appreciably during
excursions in $\bb H_2$: on $\bb A$ for all $0\le s\le t \ell^2
\theta_\beta$,
\[
\sup_{S_s \le s\le T_s} \max \bigl\{ \bigl\Vert\mt M\bigl(\eta(s)\bigr) - \mt M
\bigl(\eta(S_s)\bigr) \bigr\Vert,\bigl \Vert\mt M\bigl(\eta(s)\bigr) - \mt M
\bigl(\eta(T_s)\bigr) \bigr\Vert \bigr\} \le \frac{C_0}{n\ell},
\]
where $S_s$ (resp., $T_s$) is the last (resp., first) time before
(resp., after) $s$ in which $\eta(s)$ belongs to $\bb H_{01}$. Therefore,
\[
\mathop{\sup_{0\le s\le t \ell^2 \theta_\beta}}_{
0\le r \le\delta\ell^2 \theta_\beta} \bigl\Vert\mt M \bigl(\eta(s)
\bigr) - \mt M \bigl(\eta(s+r)\bigr) \bigr\Vert \le \frac{C_0}{n\ell} + \mathop{\sup
_{0\le s\le t \ell^2 \theta_\beta}}_{
0\le r \le\delta\ell^2 \theta_\beta} \bigl\Vert\mt M \bigl(\zeta(s)\bigr) - \mt M
\bigl(\zeta(s+r)\bigr) \bigr\Vert.
\]
This proves that the sequence of Markovs chains $\mt M^\beta(t)$ is tight.

To characterize the limit points, we compare $\mt M^\beta(t)$ to
$\hat{\mt M}_\beta(t)$ and we use the fact that $\zeta(t)$ is the
trace of
$\eta(t)$ on $\Xi$. We need to prove \eqref{n01} with the obvious
modifications. The main point in the proof of \eqref{n01} is
assertion \eqref{n02} whose proof reduces to the estimate
\[
\lim_{\beta\to\infty} \mb E_{\eta^{\mb0}} \biggl[ \int
_{0}^{t+\delta} \mb1\bigl\{\eta\bigl(s
\theta_\beta\ell^2\bigr) \notin\Xi\bigr\} \,ds \biggr] = 0.
\]
This estimate has been derived in \eqref{c14}, which completes the
proof of the theorem.

\section{The time-scale \texorpdfstring{$\theta_\beta$}{$theta_beta$}}
\label{sec8}

We prove in this section the bounds \eqref{39} on $\theta_\beta$.
The proof relies on the next lemma.

\begin{lemma}
\label{s11}
Assume that \eqref{c09} holds. Then there exist constants
$0<c_0<C_0<\infty$, independent of $\beta$, such that
\[
\frac{c_0}{n^2} \frac{e^{-2 \beta} \mu_\beta(\eta^{\mb 0})}{
\mu_\beta(\Xi)} \le \Cap_\zeta\bigl(
\eta^{\mb 0}, \Gamma\setminus \bigl\{\eta^{\mb 0}\bigr\}\bigr) \le
\frac{C_0}{n} \frac{e^{-2 \beta} \mu_\beta(\eta^{\mb 0})}{
\mu_\beta(\Xi)} \cdot
\]
\end{lemma}

\begin{pf}
On the one hand, by the Dirichlet principle
\[
\Cap_\zeta\bigl(\eta^{\mb 0}, \Gamma\setminus\bigl\{
\eta^{\mb 0}\bigr\}\bigr) \le \bigl\langle (-L_\zeta f), f\bigr
\rangle_{\mu_\zeta}
\]
for any function $f\dvtx \Xi\to[0,1]$ which vanishes on $\Gamma
\setminus\{\eta^{\mb 0}\}$ and which is equal to $1$ at $\eta^{\mb
0}$. Taking $f = \mb1\{\mc G_0\}$, we obtain that
\[
\Cap_\zeta\bigl(\eta^{\mb 0}, \Gamma\setminus\bigl\{
\eta^{\mb 0}\bigr\}\bigr) \le \sum_{\eta\in\mc G_{\mb 0}}
\mu_\zeta(\eta) R_\zeta\bigl(\eta, \mc G_{\mb 0}^c
\bigr).
\]
By Propositions \ref{s16} and \ref{s08}, Lemma \ref{s19} and since
$\mu_\zeta(\eta) = \mu_K(\eta)/\mu_K(\Xi) =
\mu_\beta(\eta)/\mu_\beta(\Xi)$, the previous sum is bounded by
\begin{eqnarray*}
& &\frac{e^{-\beta} \mu_\beta(\eta^{\mb 0})}{
\mu_\beta(\Xi)} \sum_{\eta\in\mc V(\eta^{\mb 0})} \bigl\{
e^{-\beta} {\mb R} \bigl(\eta, \mc G_{\mb 0}^c\bigr) +
C_0 \kappa_2 e^{-\beta} \bigr\} +
\frac{\mu_\beta(\eta^{\mb 0})}{
\mu_\beta(\Xi)} e^{-2\beta} \kappa_1
\\
&&\qquad \le \frac{C_0 e^{-2 \beta} \mu_\beta(\eta^{\mb 0})}{
n \mu_\beta(\Xi)} \{1 + n \kappa_2 + n \kappa_1
\},
\end{eqnarray*}
for some finite constant $C_0$ whose value change from line to
line. By \eqref{c09}, $n (\kappa_1 + \kappa_2)$ vanishes as
$\beta\uparrow\infty$. This proves the upper bound.

On the other hand, by Thomson's principle \cite{g1}, Proposition 3.2.2,
\[
\frac{1}{\Cap_\zeta(\eta^{\mb 0}, \Gamma\setminus\{\eta^{\mb 0}\})} \le \frac{1}2 \sum_{\eta,\xi\in\Xi}
\frac{1} {\mu_\zeta(\eta)
R_\zeta(\eta,\xi)} \Phi(\eta,\xi)^2,
\]
for any unitary flow $\Phi$ from $\eta^{\mb 0}$ to $\Gamma\setminus
\{\eta^{\mb 0}\}$.

To construct such a flow, recall from \eqref{51} the path $\xi_0 =
\eta^{\mb 0}, \xi_1, \ldots, \xi_M = \eta^{e_1}$, $M=3n-2$, from
$\eta^{\mb 0}$ to $\eta^{e_1}$ obtained by sliding particles around
the square $Q$. Let $\Psi$ be the unitary flow from $\eta^{\mb 0}$ to
$\eta^{e_1}$ defined by $\Psi(\xi_i,\xi_{i+1})=1$, $0\le i<M$. By~\eqref{53} and \eqref{41},
\begin{eqnarray*}
&& \frac{1}2 \sum_{\eta,\xi\in\Xi} \frac{1} {\mu_\zeta(\eta)
R_\zeta(\eta,\xi)}
\Psi(\eta,\xi)^2\\
&&\qquad\le \frac{\mu_\beta(\Xi)}{\mu_\beta(\eta^{\mb 0})} \frac{e^{2\beta} }{
\mb R(\xi_0,\xi_1)} \biggl\{ 1
+ \frac{\kappa_1}{
\mb R(\xi_0,\xi_1) - \kappa_1} \biggr\}
\\
&&\qquad\quad{} + \frac{\mu_\beta(\Xi) e^\beta} {\mu_\beta(\eta^{\mb 0})} \sum_{i=1}^{M-1}
\frac{e^\beta}{\mb R(\xi_i,\xi_{i+1})} \biggl\{ 1 + \frac{C_0 \kappa_2}{\mb R(\xi_i,\xi_{i+1}) -
C_0\kappa_2} \biggr\}.
\end{eqnarray*}
By \eqref{51} and \eqref{52}, $\mb R(\xi_i, \xi_{i+1})\ge n ^{-1}$.
The previous expression is thus less than or equal to
\[
\frac{C_0 \mu_\beta(\Xi) n^2 e^{2\beta}} {\mu_\beta(\eta^{\mb 0})} \biggl\{ 1 + \frac{n \kappa_2}{1 -
n \kappa_2} + \frac{n \kappa_1}{1 - n \kappa_1} \biggr
\}.
\]
By \eqref{c09}, $n \kappa_1$ and $n\kappa_2$ vanishes as
$\beta\uparrow\infty$. This proves the lower bound.
\end{pf}

We are now in a position to prove the bounds \eqref{39}. Recall that
we denote by $R_\xi(\eta^{\mb x},\eta^{\mb y})$, $\mb x \neq \mb
y\in\bb T_L$, the jump rates of the trace process $\xi(t)$. Since
$\zeta(t)$, $\xi(t)$ are the traces of the process $\eta(t)$ on
$\Xi$, $\Gamma$, respectively, and since $\Xi\supset\Gamma$,
$\xi(t)$ is also the trace of $\zeta(t)$ on $\Gamma$. In particular,
by \cite{bl2}, Proposition 6.1,
%
\begin{equation}
\label{38} R_\xi\bigl(\eta^{\mb 0},\eta^{\mb x}
\bigr) = \lambda_\zeta\bigl(\eta^{\mb 0}\bigr) \mb
P^\zeta_{\eta^{\mb 0}} \bigl[ H^+_{\Gamma} = H_{\eta^{\mb
x}}
\bigr].
\end{equation}
Hence, by the definition \eqref{16} of $\theta_\beta$,
%
\begin{equation}
\label{c11} \theta^{-1}_\beta = \sum
_{\mb x\in\bb T_L} \Vert\mb x\Vert^2 r_\beta(\mb x)
= \lambda_\zeta\bigl(\eta^{\mb 0}\bigr) \sum
_{\mb x\in\bb T_L} \Vert\mb x\Vert^2 \mb P^\zeta_{\eta^{\mb 0}}
\bigl[ H^+_{\Gamma} = H_{\eta^{\mb x}} \bigr],
\end{equation}
so that
\[
\lambda_\zeta\bigl(\eta^{\mb 0}\bigr) \mb P^\zeta_{\eta^{\mb 0}}
\bigl[ H_{\Gamma\setminus\{\eta^{\mb 0}\}} < H^+_{\eta^{\mb 0}} \bigr] \le \theta^{-1}_\beta
\le 2L^2 \lambda_\zeta\bigl(\eta^{\mb 0}\bigr) \mb
P^\zeta_{\eta^{\mb 0}} \bigl[ H_{\Gamma\setminus\{\eta^{\mb 0}\}} < H^+_{\eta^{\mb 0}}
\bigr].
\]
Since
\[
\Cap_\zeta\bigl(\eta^{\mb 0}, \Gamma\setminus\bigl\{
\eta^{\mb 0}\bigr\}\bigr) = \mu_\zeta\bigl(\eta^{\mb 0}
\bigr) \lambda_\zeta\bigl(\eta^{\mb 0}\bigr) \mb
P^\zeta_{\eta^{\mb 0}} \bigl[ H_{\Gamma\setminus
\{\eta^{\mb 0}\}} < H^+_{\eta^{\mb 0}}
\bigr],
\]
we conclude that
\[
\frac{1}{\mu_\zeta(\eta^{\mb 0}) } \Cap_\zeta\bigl(\eta^{\mb 0}, \Gamma
\setminus\bigl\{\eta^{\mb 0}\bigr\}\bigr) \le \theta^{-1}_\beta
\le \frac{2L^2}{\mu_\zeta(\eta^{\mb 0}) } \Cap_\zeta\bigl(\eta^{\mb 0}, \Gamma
\setminus\bigl\{\eta^{\mb 0}\bigr\}\bigr).
\]
Assertion \eqref{39} follows now from Lemma \ref{s11}.

We may improve the upper bound above using the arguments presented in
the proof of Theorem \ref{s10}. Assume that \eqref{c10} is in force
and recall the definition of the sets $\lozenge_k$, $\blacklozenge_k$
introduced right after \eqref{34}. By \eqref{c11} and since
$\lambda_\zeta(\eta^{\mb 0}) \le C_0 e^{-2\beta}$,
\[
\theta^{-1}_\beta \le C_0 e^{-2\beta} \sum
_{k=1}^{2L} k^2 \mb
P^\zeta_{\eta^{\mb 0}} \bigl[ H^+_{\Gamma} = H_{\lozenge_k}
\bigr] \le C_0 e^{-2\beta} \sum_{k=1}^{2L}
k \mb P^\zeta_{\eta^{\mb 0}} \bigl[ H^+_{\Gamma} =
H_{\blacklozenge_k} \bigr],
\]
where we performed a summation by parts in the last step. By
Proposition \ref{lc01}, the previous sum is bounded by
\[
C_0 e^{-2\beta} L^2 \kappa_3 +
C_0 e^{-2\beta} \sum_{k=1}^{2L}
k \mb P^{\widehat{\zeta}}_{\eta^{\mb 0}} \bigl[ H^+_{\Gamma} =
H_{\blacklozenge_k} \bigr],
\]
where $\kappa_3 = \kappa_1 + \sqrt{\kappa_2 |\Xi|/L^2}$. By Lemma
\ref{lc02} and \eqref{c05}, the second term is less than or equal to
$C_0 e^{-2\beta} n^2$. By assumption \eqref{c10}, the first one is
bounded above by $C_0 e^{-2\beta} n^2$. This proves \eqref{c12}.

\section*{Acknowledgments} The authors would like to thank
A. Teixeira, F. Manzo and E. Scopolla for fruitful discussions, and
two anonymous referees for their comments.


%




\printaddresses
\end{document}